\documentclass[10pt,leqno]{amsart}

\usepackage{float}
\usepackage{amssymb,amsmath,amsthm,amscd}
\usepackage{graphics}
\usepackage{mathrsfs}
\usepackage{enumerate}
\usepackage{color}
\usepackage{verbatim}
\usepackage{mathtools}
\usepackage{longtable}
\usepackage[bookmarks=true,bookmarksnumbered=true,bookmarkstype=toc=true,pdftitle={\@title},pdfauthor={\@author},pdfstartview={FitBH -32768}]{hyperref}

\numberwithin{equation}{section}
\newcommand{\nc}{\newcommand}
\usepackage{graphicx}

\def\node#1#2{\overset{#1}{\underset{#2}{\circ}}}

\newcommand{\SSSS}{S}

\newcommand{\HA}{\widehat{X}}
\newcommand{\defequiv}{\stackrel{\textrm{def}}{\Longleftrightarrow}}
\newcommand{\kl}{Khovanov-Lauda}
\newcommand{\km}{Kac-Moody}
\newcommand{\bn}{Brauer-Nesbitt}
\newcommand{\bk}{Brundan-Kleshchev}
\newcommand{\ih}{Iwahori-Hecke}
\newcommand{\kor}{K\"ulshammer-Olsson-Robinson}
\newcommand{\asy}{Ando-Suzuki-Yamada}

\newcommand{\MEE}{\mathfrak{m}}

\newcommand{\ul}{\underline}

\newcommand{\RHO}{\rho}
\newcommand{\KA}{\mathcal{K}}
\newcommand{\MO}{\mathscr{O}}
\newcommand{\POW}{\mathsf{Pow}}
\newcommand{\TOR}{\mathsf{Tor}}
\newcommand{\TRIV}{\mathsf{triv}}

\DeclareMathOperator{\ANN}{\mathsf{Ann}}

\DeclareMathOperator{\SPEC}{\mathsf{Spec}}
\DeclareMathOperator{\MSPEC}{\mathsf{max-Spec}}
\DeclareMathOperator{\FITT}{\mathsf{Fitt}}
\DeclareMathOperator{\SECTION}{\mathsf{Sec}}
\newcommand{\CATISO}{\mathsf{ch}}
\newcommand{\HAT}[2]{S^{#2}(#1)}
\newcommand{\OVERy}[2]{\overline{#1}^{\, #2}}
\newcommand{\HATT}[2]{{#1}_{#2}}
\newcommand{\HATTT}[3]{{#1}_{#2}^{#3}}
\newcommand{\CRBIJ}{\varphi}
\newcommand{\KEYBIJECTION}[1]{\beta_{#1}}

\newcommand{\PRIMES}{\mathsf{Prm}}

\newcommand{\NNN}{\Z_{\geq1}}

\newcommand{\IRR}{\mathsf{Irr}}

\newcommand{\PAR}{\mathsf{Par}}
\newcommand{\RPAR}{\mathsf{RP}}
\newcommand{\CPAR}{\mathsf{CRP}}

\newcommand{\C}{{\mathbb C}}

\newcommand{\BLOCK}{\mathsf{Bl}}
\newcommand{\BAR}{\mathsf{bar}}
\newcommand{\BARR}{{\sigma}}

\newcommand{\SIM}{\sim}

\newcommand{\QSH}{\mathsf{QSh}}
\newcommand{\RSH}{\mathsf{RSh}}

\newcommand{\QSHM}{\mathsf{QSh}^\mathsf{M}}
\newcommand{\RSHM}{\mathsf{RSh}^\mathsf{M}}
\newcommand{\Q}{\mathbb {Q}}

\newcommand{\Z}{{\mathbb Z}}

\newcommand{\MP}{{\mathcal{P}}}
\newcommand{\MPC}{{\mathcal{P}^\vee}}
\newcommand{\MA}{\mathscr{A}}
\newcommand{\MAp}[1]{\Z_{(#1)}[v,v^{-1}]}

\newcommand{\MAA}{{R}}

\newcommand{\corr}{\mathbb{F}}

\newcommand{\TRANS}[1]{{}^{\mathrm{tr}}{#1}}
\newcommand{\corrr}[1]{k_{#1}}
\newcommand{\qq}[1]{\eta_{#1}}

\theoremstyle{plain}
\newtheorem{lemma}{Lemma}[section]
\newtheorem{prop}[lemma]{Proposition}
\newtheorem{theorem}[lemma]{Theorem}
\newcommand{\Prop}{\begin{prop}}
\newcommand{\enprop}{\end{prop}}
\newcommand{\Lemma}{\begin{lemma}}
\newcommand{\enlemma}{\end{lemma}}
\newcommand{\Th}{\begin{theorem}}
\newcommand{\enth}{\end{theorem}}
\newtheorem{corollary}[lemma]{Corollary}
\newcommand{\Cor}{\begin{corollary}}
\newcommand{\encor}{\end{corollary}}
\newtheorem{definition}[lemma]{Definition}
\newtheorem{question}[lemma]{Question}
\newcommand{\Question}{\begin{question}}
\newcommand{\enquestion}{\end{question}}
\newcommand{\Def}{\begin{definition}}
\newcommand{\edf}{\end{definition}}

\theoremstyle{definition}
\newtheorem{remark}[lemma]{Remark}

\newtheorem{conjecture}[lemma]{Conjecture}
\nc{\Rem}{\begin{remark}}
\nc{\enrem}{\end{remark}}

\newcommand{\Con}{\begin{conjecture}}
\newcommand{\encon}{\end{conjecture}}

\nc{\emprule}[1]{\rule{#1}{0pt}}
\newcommand{\g}{{\mathfrak{g}}}

\newcommand{\Hom}{\operatorname{Hom}}
\newcommand{\Rank}{\operatorname{rank}}

\newcommand{\isoto}[1][]{\mathop{\xrightarrow[#1]{\rule{0pt}{.9ex}{\raisebox{-.6ex}[0ex][-.7ex]{$\mspace{4mu}\sim\mspace{3mu}$}}}}}

\renewcommand{\hom}{\operatorname{\it \mathscr{H}\kern-.25em om}}

\newcommand{\NPI}[1]{(N^{(p)}_{#1})^{-1}}

\newcommand{\M}{{\mathscr M}}
\newcommand{\N}{{\mathbb{N}}}
\newcommand{\eq}{\begin{eqnarray}}
\newcommand{\eneq}{\end{eqnarray}}
\newcommand{\GPPP}[1]{\mathsf{Proj}_\mathsf{gr}({#1})}

\newcommand{\eqn}{\begin{eqnarray*}}
\newcommand{\eneqn}{\end{eqnarray*}}

\newenvironment{tenumerate}{
  \begin{enumerate}
  
  }{\end{enumerate}}

\nc{\bnum}{\begin{enumerate}[{\rm (i)}]}
\nc{\enum}{\ee}
\nc{\bna}{\begin{enumerate}[{\rm (a)}]}
\nc{\bnA}{\begin{enumerate}[{\rm (A)}]}
\nc{\bnX}{\begin{enumerate}[{\rm (X)}]}
\nc{\ena}{\ee}

\newcommand{\on}{\operatorname}

\newcommand{\bni}{\begin{tenumerate}}
\newcommand{\eni}{\end{tenumerate}}

\newcommand{\QED}{\end{proof}}
\newcommand{\Proof}{\begin{proof}}

\newcommand{\CONG}[1]{\equiv_{#1}}
\newcommand{\PCONG}[1]{\equiv'_{#1}}
\newcommand{\GCONG}{\fallingdotseq}
\newcommand{\FCONG}[1]{\equiv^F_{#1}}
\newcommand{\NCONG}[1]{\not\equiv_{#1}}

\newcommand{\ba}{\begin{array}}
\newcommand{\ea}{\end{array}}

\newcommand{\hs}{\hspace*}

\newcommand{\eqsub}{\begin{subequations}\begin{eqnarray}}
\newcommand{\eneqsub}{\end{eqnarray}\end{subequations}}

\newcommand{\ol}{\overline}

\newcommand{\A}{\mathscr{A}}

\newcommand{\IND}{{{\on{Ind}}}}

\nc{\la}{\lambda}
\nc{\lam}{\lambda}
\nc{\U}[1][\g]{U_q(#1)}
\nc{\te}{\tilde{e}}
\nc{\tei}{\tilde{e}_i}
\nc{\tf}{\tilde{f}}
\nc{\tfi}{\tilde{f}_i}
\nc{\tU}{\widetilde U_q(\g)}
\nc{\tE}{\tilde{E}}
\nc{\tF}{\tilde{F}}

\nc{\BZ}{{\mathbb{Z}}}
\nc{\al}{\alpha}
\nc{\qs}{{q}}
\nc{\lan}{\langle}
\nc{\ran}{\rangle}
\nc{\re}{{\mathrm{re}}}
\nc{\wt}{\operatorname{wt}}
\nc{\Uf}[1][\g]{U^-_q(#1)}
\nc{\Ue}{U^+_q(\g)}
\nc{\eps}{\varepsilon}
\nc{\vphi}{\varphi}
\nc{\sphi}{\varphi^*}
\nc{\seps}{\varepsilon^*}

\nc{\nn}{\nonumber}
\def\max{{\mathop{\mathrm{max}}}}
\nc{\vp}{\varpi}
\nc{\cls}{{\operatorname{cl}}}

\nc{\Wt}{{\operatorname{Wt}}}
\nc{\Us}{U'_q(\g)}
\nc{\La}{\Lambda}
\nc{\ro}{{\rm(}}
\nc{\rf}{{\rm)}}
\nc{\norm}{{\mathrm{norm}}}
\nc{\qbox}{\quad\mbox}
\nc{\braid}{{\mathfrak{B}}}
\nc{\Ad}{\operatorname{Ad}}
\nc{\Aut}{\operatorname{Aut}}
\nc{\dt}[1]{\tilde{\tilde #1}}
\nc{\Sn}{S^{{\mathrm{norm}}}}
\nc{\aff}{{\mathrm{aff}}}
\nc{\rk}{{\mathrm{rk}}}
\nc{\tQ}{\widetilde{Q}}
\nc{\tP}{\widetilde{P}}
\nc{\tW}{\widetilde{\mathscr{W}}}
\nc{\Dyn}{\mathrm{Dyn}}
\nc{\tD}{\widetilde{\Delta}}
\nc{\height}{{\operatorname{ht}}}
\nc{\bl}{\bigl}
\nc{\br}{\bigr}
\nc{\Hecke}{\mathrm{H}}
\nc{\HB}{\Hecke^{\mathrm{B}}}
\nc{\K}{\mathrm{K}}
\newcommand{\scbul}{{\,\raise1pt\hbox{$\scriptscriptstyle\bullet$}\,}}
\nc{\vac}{{\phi}}
\nc{\be}{\begin{enumerate}}
\nc{\ee}{\end{enumerate}}
\nc{\low}{{\mathrm{low}}}
\nc{\upper}{{\mathrm{up}}}
\nc{\Zodd}{\Z_{\mathrm{odd}}}
\nc{\Ft}[1][n]{\mathbb{P}\mathrm{ol}_{#1}}
\nc{\Ftf}[1][n]{\widetilde{\mathbb{P}\mathrm{ol}}_{#1}}
\nc{\KB}{\on{K}^{\mathrm{B}}}
\nc{\Res}{\mathrm{Res}}
\nc{\Fc}[1][{n,m}]{\mathbf{F}_{#1}}
\nc{\tphi}{\widetilde{\varphi}}
\nc{\CO}{\mathscr{O}}
\nc{\CK}{\mathscr{K}}
\nc{\disc}{\mathfrak{d}}
\nc{\tr}{\on{tr}}
\nc{\Gb}{\mathfrak{b}}
\nc{\Gh}{\mathfrak{h}}
\nc{\ga}{\mathfrak{a}}
\nc{\stable}{\mathrm{stable}}
\nc{\X}{\mathfrak{X}}
\nc{\Hilb}{\mathrm{Hilb}}
\nc{\W}{\ensuremath{\mathscr{W}}}
\nc{\Ws}{\ensuremath{\rm W}}
\nc{\opp}{{\on{opp}}}

\nc{\corps}{{\mathbf{k}}}
\nc{\cor}{{\mathbf{k}}}
\nc{\h}{\mathrm{\hslash}}
\nc{\fL}[1][{\h}]{\C(\mspace{-1mu}(#1)\mspace{-1mu})}
\nc{\ad}{\mathrm{ad}}
\newcommand{\Endm}{\operatorname{\mathscr{E}\kern-.1pc\mathit{nd}}}
\newcommand{\Endomo}{\operatorname{\mathscr{E}\kern-.1pc\mathit{nd}}}
\nc{\bc}{\bar{\corps}}
\nc{\reg}{{\mathrm{reg}}}
\nc{\ysq}{\mathbf{y}^2}
\nc{\CH}{\mathsf{char}}

\nc{\sketch}{\Proof}
\nc{\Gm}{\mathbb{G}_{\mathrm{m}}}
\nc{\hGm}{\hat{\mathbb{G}}_{\mathrm{m}}}
\nc{\ug}{\widehat{\mathrm{G}}_{\mathrm{m}}}

\nc{\tL}{\widetilde{\mathscr{L}}}
\nc{\Fr}{\mathcal{F}}
\nc{\E}{\mathcal{E}}

\nc{\ord}{\on{ord}}
\nc{\bM}{\overset{\hs{1.5ex}\rule[-.08ex]{1.8ex}{.08ex}}{\M}}
\nc{\romano}{\mathrm{o}}
\nc{\into}{\hookrightarrow}
\nc{\good}{\mathrm{good}}
\nc{\tA}{\widetilde\A}
\nc{\Vz}{{V}\kern-1.1ex\raisebox{1.5ex}[0ex][0ex]{$\cdot$}}
\nc{\bxes}[1]{\raisebox{.9ex}{$\cdot$}{\kern#1}\raisebox{0ex}{$\cdot$}
{\kern#1}\raisebox{-.9ex}{$\cdot$}}
\nc{\ssum}{\mathop{\mbox{\normalsize{${\sum}$}}}\limits}
\nc{\ct}{{\mbox{\tiny$\mathrm{CT}$}}}
\nc{\pr}{\mathrm{pr}}
\nc{\qr}{\mathrm{rp}}
\newcommand{\p}{{\{p\}}}

\DeclareMathOperator{\MAT}{\mathsf{Mat}}

\DeclareMathOperator{\KER}{\mathsf{Ker}}
\DeclareMathOperator{\COKER}{\mathsf{Coker}}
\nc{\COKERR}[1]{\mathsf{Cok}_{#1}}

\DeclareMathOperator{\SYM}{\mathsf{Sym}}
\DeclareMathOperator{\MULT}{\mathsf{Mult}}

\DeclareMathOperator{\GL}{\mathsf{GL}}
\DeclareMathOperator{\AUT}{\mathsf{Aut}}

\DeclareMathOperator{\PRS}{\pi}

\DeclareMathOperator{\GCD}{}

\DeclareMathOperator{\PROC}{\mathsf{PC}}
\DeclareMathOperator{\DIAG}{\mathsf{diag}}
\DeclareMathOperator{\RED}{\mathsf{red}}

\DeclareMathOperator{\CUT}{\mathsf{cut}}
\DeclareMathOperator{\INFL}{\mathsf{Infl}}

\def\GL{\operatorname{GL}\nolimits}

\nc{\Fs}{\ensuremath{\rm F}}

\nc{\isotf}{\overset{
{\rule{0pt}{.9ex}%
{\raisebox{-.6ex}[0ex][-.7ex]{$\mspace{3mu}\sim\mspace{3mu}$}}}}
{\longleftrightarrow}}
\nc{\tN}{\tilde\N}
\nc{\tens}{\mathop\otimes\limits}
\nc{\super}{\mathrm{super}}
\nc{\Mods}{\on{Mod}_\super}
\nc{\rev}{\mathrm{rev}}
\nc{\clif}{\mathfrak{C}}
\nc{\clifm}{\clif^{-}}
\nc{\Fct}{\mathrm{Fct}}
\nc{\Fcts}{\mathrm{Fct}_\super}

\nc{\Ks}{\on{K^\super}}
\nc{\ts}{\widetilde{s}}
\nc{\KUGIRI}{\circ}
\nc{\Sym}{\mathfrak{S}}
\nc{\FF}{\mathcal{F}}
\nc{\BF}{\mathcal{B}}
\nc{\BFF}{\widetilde{\mathcal{B}}}

\nc{\cc}{\mathfrak{c}}
\nc{\SK}{\mathcal{KS}}
\nc{\noi}{\noindent}
\nc{\odd}{{\mathrm{odd}}}
\nc{\even}{{\mathrm{even}}}
\nc{\bs}{\ol{s}}
\nc{\Khc}[1][n]{\ol{\mathcal{KHC}}_{#1}}
\nc{\Ohc}[1][n]{\ol{\mathcal{OHC}}_{#1}}
\nc{\KHC}[1][n]{\mathcal{K}{\mathcal{HC}}_{#1}}
\nc{\OHC}[1][n]{\mathcal{O}{\mathcal{HC}}_{#1}}

\nc{\IODD}{I_{\odd}}
\nc{\IEVEN}{I_{\even}}
\DeclareMathOperator{\KKK}{\mathsf{K}_0}
\newcommand{\MOD}[1]{{\mathsf{Mod}({#1})}}

\newcommand{\GMOD}[1]{{\mathsf{Mod}_{\mathsf{gr}}({#1})}}
\newcommand{\F}{\mathbb{F}}
\nc{\IRED}{I}

\DeclareMathOperator{\QCHAR}{\mathsf{qchar}}

\nc{\MH}{\mathcal{H}}

\newcommand{\MAPSTO}{\longmapsto}

\DeclareMathOperator{\DEG}{\mathsf{deg}}

\newcommand{\HM}[1]{\widehat{M}_{#1}}
\newcommand{\HMI}[1]{\widehat{M}^{-1}_{#1}}
\newcommand{\HI}{\widehat{I}}

\newcommand{\RP}[1]{\mathsf{RP}_3}

\nc{\bwr}{\mbox{\large$\wr$}}
\nc{\At}[1][{{i,j}}]{\mathscr{A}_{{#1}}}
\nc{\tAt}[1][{{i,j}}]{\widetilde{\mathscr{A}}_{{#1}}}
\nc{\Bt}[1][{{i,j}}]{\mathscr{B}_{{#1}}}
\nc{\tBt}[1][{{i,j}}]{\widetilde{\mathscr{B}}_{{#1}}}
\nc{\prt}[1]{\mathrm{par}(#1)}
\nc{\er}{\mathrm{e}}
\nc{\ec}{\mathrm{e}^-}

\nc{\GCM}{{generalized Cartan matrix}}
\nc{\HCO}{\widehat{\CO}}
\nc{\tCO}{\widetilde{\CO}}
\nc{\red}[1]{{\color{red}{#1}}}
\nc{\T}{\mathbb{T}}
\nc{\HC}{\mathsf{HC}}
\nc{\EV}{\mathsf{ev}}
\nc{\HRC}[1][n]{\widehat{\mathrm{RC}}_{{#1}}}
\nc{\hc}[1][n]{\ol{\mathcal{HC}}_{{#1}}}
\nc{\bphi}{\bar{\phi}}
\nc{\hgt}{\mathrm{ht}}
\newcommand{\wh}{\widehat}
\newcommand{\vph}{\vphantom}

\usepackage{tikz}
\usetikzlibrary{chains}

\tikzset{node distance=1.75em, ch/.style={circle,draw,on chain,inner sep=1.5pt, line width=0.2pt},chj/.style={ch,join},line width=1pt,baseline=-1ex}

\newcommand{\dnode}[2][chj]{%
\node[#1,label={below,scale=0.8:#2}] {};
}

\newcommand{\dnodenj}[1]{%
\dnode[ch]{#1}
}

\newcommand{\dnodebr}[1]{%
\node[chj,label={right, scale=0.8:#1}] {};
}

\newcommand{\dydots}{%
\node[chj,draw=none,inner sep=1pt] {\dots};
}

\begin{document}
\title[On graded Cartan invariants]{On graded Cartan invariants of symmetric groups and Hecke algebras}

\author{Anton Evseev}
\address{School of Mathematics, University of Birmingham, Edgbaston, Birmingham B15 2TT, UK}
\email{a.evseev@bham.ac.uk}
\thanks{The first author was partially supported by the EPSRC Postdoctoral Fellowship EP/G050244 and the EPSRC grant EP/L027283.}

\author{Shunsuke Tsuchioka}
\address{Graduate School of Mathematical Sciences, University of Tokyo,
Komaba, Meguro, Tokyo, 153-8914, Japan}
\thanks{The second author was supported in part by JSPS Kakenhi Grants
 11J08363 and 26800005.}
\email{tshun@kurims.kyoto-u.ac.jp}

\date{May 21, 2015}

\begin{abstract}
We consider graded Cartan matrices of the symmetric groups and the {\ih} algebras of type $A$ at roots of unity. These matrices are $\mathbb Z[v,v^{-1}]$-valued and may also be interpreted as Gram matrices of the Shapovalov form on sums of weight spaces of a basic representation of an affine quantum group. We present a conjecture predicting the invariant factors of these matrices and give evidence for the conjecture by proving its implications under a localization and certain specializations of the ring $\mathbb Z[v,v^{-1}]$.
This proves and generalizes a conjecture of {\asy} on the invariants of these matrices over $\mathbb Q[v,v^{-1}]$ and 
also generalizes the first author's recent proof of the {\kor} conjecture over $\Z$. 
\end{abstract}

\maketitle
\section{Introduction}

The main object of study in this paper is the graded Cartan matrix $C^v_{\MH_n(\corrr{\ell};\qq{\ell})}$ of the {\ih} algebra of type $A$ 
(see Definition \ref{HeckeAlg}) in quantum characteristic $\ell$, whose entries belong to the Laurent polynomial ring $\MA= \Z[v,v^{-1}]$. 
To provide background and motivation, we begin by describing the relevant constructions and results 
for the ungraded case, obtained by substituting $v=1$ (see \S\ref{subsec:ungraded}). In \S\ref{subsec:graded} we move on to the graded case and state conjectures and results on the ``invariant factors'' of $C^v_{\MH_n(\corrr{\ell};\qq{\ell})}$, which are studied in the rest of the paper.
We freely use the notation and conventions of \S\ref{NoCo}.

\subsection{Generalized modular character theory of the symmetric groups}\label{subsec:ungraded}

In~\cite{KOR}, 
K\"ulshammer, Olsson, and Robinson initiated a study of an $\ell$-analogue of the modular character theory of the symmetric group
$\mathfrak{S}_n$ for an arbitrary integer $\ell\geq 2$. They showed that many of the classical combinatorial aspects of representation theory of $\mathfrak S_n$ over a field 
of a prime characteristic $p$ (such as cores, blocks and Nakayama conjecture) generalize to the case when $p$ is not necessarily a 
prime and is replaced by $\ell$.
Our interest focuses on the generalized Cartan matrices defined in~\cite[\S1]{KOR} ($\ell$-Cartan matrices, for short) and, 
in particular, on their Smith normal forms over $\Z$.
It is convenient to define $\ell$-Cartan matrices in terms of Hecke algebras rather than the symmetric groups. 
Throughout, we consider the Hecke algebra $\MH_n(\corrr{\ell}; \qq{\ell})$ defined as usual.

\Def\label{HeckeAlg} 
For a field $\corr$ and $q\in\corr^{\times}$,
$\MH_n(\corr;q)$ is defined to be the $\corr$-algebra 
generated by $\{T_r\mid 1\leq r<n\}$ subject to the relations
\begin{align*}
(T_r+1)(T_r-q)=0,\quad T_sT_{s+1}T_s=T_{s+1}T_sT_{s+1},\quad T_tT_u=T_uT_t
\end{align*}
for $1\leq r\leq n-1$, $1\le s\le n-2$ and $1\leq t,u<n$ such that $|t-u|>1$.
For $\ell\geq 2$, we fix a field $\corrr{\ell}$ which has a primitive $\ell$-th root of unity $\qq{\ell}$. 
\edf

\Def\label{CartanInv} 
Let $A$ be a finite-dimensional algebra over a field $\corr$.
\bna
\item\label{defmodcat} We denote by $\MOD{A}$ the abelian category of finite-dimensional left $A$-modules and $A$-homomorphisms between them.
\item We define the Cartan matrix $C_A$ of $A$ to be the matrix
$([\PROC(D):D'])_{D,D'\in\IRR(\MOD{A})}\in\MAT_{|\IRR(\MOD{A})|}(\Z)$ where $\PROC(D)$ is the projective cover of $D$. 
\ee
\edf

\subsection{The {\kor} conjecture}\label{subsec:kor}

\begin{definition}\label{unimoddef}
Let $X$ and $Y$ be $n\times m$-matrices with entries in a commutative ring $\MAA$. 
The matrices $X$ and $Y$ are said to be unimodularly equivalent over $\MAA$ if $Y=UXV$ for some $U \in \GL_n (R)$ and $V\in \GL_m (R)$. 
In this case, we write $X\equiv_R Y$. 
\end{definition}

Due to a result of Donkin~\cite[\S2.2]{Don}, the matrix $C_{\MH_n(k_{\ell}; \eta_{\ell})}$ is 
unimodularly equivalent over $\Z$ to the aforementioned $\ell$-Cartan matrix of $\mathfrak S_n$. 
Since $k_{\ell}$ is a splitting field for $\MH_n(k_{\ell}; \eta_{\ell})$ (see also ~\cite[\S2.2]{Don}),
the Smith normal form of $C_{\MH_n(k_{\ell}; \eta_{\ell})}$ does not depend on the choice of $k_{\ell}$ or $\eta_{\ell}$.

It is a standard result in modular representation theory (due to {\bn}) that, for a prime $p$ and a finite group $G$, 
the elementary divisors of $C_{\overline{\mathbb F}_p G}$ 
are described in terms of 
$p$-defects of $p$-regular conjugacy classes of $G$. 
When $p$ is replaced with a possibly composite number $\ell$, 
the Smith normal form of $C_{\MH_n(k_{\ell}; \eta_{\ell})}$ is more complicated:

\begin{theorem}\label{thm:kor}
Let $\ell\geq 2$. If $k\in \Z$, write $\ell_k=\ell/\GCD(\ell,k)$. For a partition $\lambda$, define 
\begin{align}
r_{\ell}(\lambda)=
\prod_{k\in\N\setminus \ell\Z}\ell_k^{\lfloor \frac{m_k(\lambda)}{\ell} \rfloor}\cdot 
\bigl\lfloor \frac{m_k(\lambda)}{\ell} \bigr\rfloor!_{\PRS(\ell_k)}.
\label{classicalr}
\end{align}
Then
\begin{align*}
C_{\MH_n(\corrr{\ell};\qq{\ell})}\CONG{\Z}\DIAG(\{r_{\ell}(\lambda)\mid \lambda\in\CPAR_{\ell}(n)\}),
\end{align*}
where 
$\pi(\ell_k)$ is the set of prime divisor of $\ell_k$ and 
$\CPAR_{\ell}(n)$ is the set of $\ell$-class regular partitions of $n$ (see~\S\ref{NoCo} below).
\end{theorem}

This result was proposed as a conjecture by K{\"u}lshammer, Olsson and Robinson (\cite[Conjecture 6.4]{KOR}) and is known as the \emph{KOR conjecture}. 
The determinant of the Cartan matrix $C_{\MH_n(\corrr{\ell};\qq{\ell})}$ was first computed by Brundan and Kleshchev~\cite[Corollary 1]{BK1}
and was shown to agree with the conjecture in~\cite{KOR}. 
Hill~\cite[Conjecture 10.5]{Hil} gave a conjectural description of the invariant factors of the Cartan matrix of each individual block of $\MH_n(\corrr{\ell};\qq{\ell})$ and proved this description in the case when each prime divisor $p$ of $\ell$ appears with multiplicity at most $p$ in the prime decomposition of $\ell$. 
The description was shown to imply Theorem~\ref{thm:kor} by Bessenrodt and Hill~\cite[Theorem 5.2]{BH}.
Finally, Hill's conjecture and hence Theorem~\ref{thm:kor} were proved in full generality by the first author~\cite[Theorem 1.1]{Evs}.  

The proofs in~\cite{Hil} and~\cite{Evs} both use a reduction of the
KOR conjecture to the problem of finding the Smith normal form of a certain $\PAR(d) \times \PAR(d)$-matrix which is smaller than 
$C_{\MH_n(\corrr{\ell};\qq{\ell})}$; here, $d$ is not greater than the $\ell$-weight of a fixed block of $\MH_n(\corrr{\ell};\qq{\ell})$.
The reduction (for an individual block of $\MH_n(\corrr{\ell};\qq{\ell})$) is due to Hill: see~\cite[Theorem 1.1]{Hil}; for an alternative approach, see~\cite[\S3]{Evs}. 
Among the main conjectures and results of the present paper are 
Conjecture~\ref{ourgradedconjecture}, which is a graded version of the reduced problem, and 
Corollary~\ref{mainimplication}, which is a graded version of the reduction. The ungraded versions are recovered by substituting $v=1$.

\subsection{Graded Cartan matrices and Shapovalov forms}\label{subsec:graded}

While the KOR conjecture is now a theorem, the proof in~\cite{Evs} relies on technical combinatorial arguments and does not give a satisfactory conceptual understanding of the result. In particular, unlike in the special case when $\ell$ is a prime and the {\bn} result applies, it is hard to discern a link between the statement or the proof of the KOR conjecture and the group-theoretic structure of $\mathfrak S_n$.
In a search for better understanding, we consider a remarkable 
grading on the Hecke algebras  discovered independently by {\bk}~\cite[Theorem 1.1]{BK2} and Rouquier~\cite[Corollary 3.20]{Rou}.
It is a consequence of an isomorphism between $\MH_n(\corrr{\ell};\qq{\ell})$ and 
a cyclotomic KLR algebra $R^{\Lambda_0}_n(A^{(1)}_{\ell-1})$ defined by {\kl}~\cite[\S3.4]{KL1} and Rouquier~\cite[\S3.2.6]{Rou}. 
A similar isomorphism and grading exist for the degenerate case, i.e., for the symmetric 
group algebra $\mathbb F_p \mathfrak S_n$ (see ~\cite[Theorem 1.1]{BK2} and ~\cite[Corollary 3.17]{Rou}). 
Using the grading, one defines the graded Cartan matrix 
$C^v_{\MH_n(\corrr{\ell};\qq{\ell})}$ with entries in the ring $\MA=\mathbb Z[v,v^{-1}]$ (see Definition~\ref{GCartanInv}). 
It is a refinement of $C_{\MH_n(\corrr{\ell};\qq{\ell})}$ in the sense that
we have $C_{\MH_n(\corrr{\ell};\qq{\ell})}=C^v_{\MH_n(\corrr{\ell};\qq{\ell})}|_{v=1}$. 

\begin{remark}
Rouquier~\cite{Rou2} has shown that interesting gradings are likely to exist for a large class of blocks of arbitrary finite groups. More precisely, he has constructed a grading on local blocks (i.e., blocks with normal defect group) whenever the defect group is abelian and has shown that, subject to the Brou{\'e} abelian defect group conjecture, these gradings can be transferred to arbitrary blocks with abelian defect groups. A study of the corresponding graded Cartan matrices up to unimodular equivalence may be of considerable interest, though is beyond the scope of this paper. 
\end{remark}

An alternative approach to defining $C_{\MH_n(\corrr{\ell};\qq{\ell})}$ is via the Shapovalov form 
on 
the basic representation $V(\Lambda_0)$ of the affine
{\km} Lie algebra of type $A^{(1)}_{\ell-1}$ (see~\cite{BK1, Hil}).
Generalizing to the graded case is natural from this point of view as well, as
one can replace the universal enveloping algebra of the {\km} algebra with
its quantized version $U_v (A^{(1)}_{\ell-1})$. The corresponding quantum
Shapovalov forms were studied by the second author~\cite{Tsu} and are reviewed
in \S\ref{pretsu2} below. The matrix $C^v_{\MH_n(\corrr{\ell};\qq{\ell})}$
can be described in terms of Gram matrices of quantum Shapovalov forms on weight
spaces of $V(\Lambda_0)$ (see Proposition~\ref{prop:Shapblock}). 
Since Shapovalov forms play an important role in representation theory of Lie
algebras and quantum groups, this description provides further motivation for
studying $C^v_{\MH_n(\corrr{\ell};\qq{\ell})}$. 

\subsection{A graded analog of the {\kor} conjecture}\label{subsec:graded_kor}
We propose the following graded version of the KOR conjecture.

\Con\label{gradedKORcon}
For $\ell\geq 2$, we have (see also Definition \ref{GCartanInv})
\begin{equation}\label{eq:grKOR}
C^v_{\MH_n(\corrr{\ell};\qq{\ell})}\CONG{\MA}\DIAG(\{r^v_{\ell}(\lambda)\mid \lambda\in\CPAR_{\ell}(n)\}).
\end{equation}
Here we put $\ell_k=\ell/\GCD(\ell,k)$ and for $\lambda\in\PAR$ define 
\begin{align}
r^{v}_{\ell}(\lambda)=\prod_{k\geq 1}\prod_{t=1}^{\lfloor m_k(\lambda)/\ell\rfloor}\left[\ell_kt_{\PRS(\ell_k)}\right]_{\GCD(\ell,k)t_{\PRS(\ell_k)'}},
\label{quantumr}
\end{align}
where the right-hand side is interpreted according to \S\ref{subsubsec:int} and \S\ref{subsubsec:qi}.
\encon

The second author stated this conjecture in the special case when
$\ell$ is a prime power (see~\cite[Conjecture 6.18]{Tsu}) and computed the
determinant of $C^v_{\MH_n(\corrr{\ell};\qq{\ell})}$, which agrees with the
conjecture (see~\cite[Theorem 6.11]{Tsu}).

\begin{remark}
Conjecture \ref{gradedKORcon} implies Theorem~\ref{thm:kor}:
comparing \eqref{classicalr} and \eqref{quantumr},
we have 
\begin{align*}
\left({\textstyle{\prod_{t=1}^{\lfloor m_k(\lambda)/\ell\rfloor}\left[\ell_kt_{\PRS(\ell_k)}\right]_{\GCD(\ell,k)t_{\PRS(\ell_k)'}}}}\right)|_{v=1}=
\ell_k^{\lfloor \frac{m_k(\lambda)}{\ell} \rfloor}\cdot\bigl\lfloor \frac{m_k(\lambda)}{\ell} \bigr\rfloor!_{\PRS(\ell_k)}.
\end{align*}
\end{remark}

While $C^v_{\MH_n(\corrr{\ell};\qq{\ell})}$ has a description in terms of affine Kazhdan-Lusztig polynomials 
by virtue of the graded version of Lascoux-Leclerc-Thibon-Ariki theory~\cite[Corollary 5.15]{BK3} (see also ~\cite[Remark 5.7]{Tsu}),
there is no easy combinatorial description for the entries of $C^v_{\MH_n(\corrr{\ell};\qq{\ell})}$ in general.
Nonetheless, 
we are able to reduce Conjecture~\ref{gradedKORcon} to a conjecture concerning matrices that do admit such a description up to unimodular equivalence
over $\MA$.

\Def\label{IandJ}
 For $\ell\geq 2$ and $\lambda\in\PAR$, we define $I^{v}_{\ell}(\lambda),J^{v}_{\ell}(\lambda)\in\MA$ by
\begin{align}
I^{v}_{\ell}(\lambda)=\prod_{k\geq 1}\prod_{t=1}^{m_k(\lambda)}
\left[\ell_kt_{\PRS(\ell_k)}\right]_{\GCD(\ell,k)t_{\PRS(\ell_k)'}},\quad
J^{v}_{\ell}(\lambda)=\prod_{k\geq 1}[\ell]_k^{m_k(\lambda)},
\label{IaJ}
\end{align}
where again we put $\ell_k=\ell/\GCD(\ell,k)$. 
\edf

The following conjecture involves a matrix $M_n$, which for the purposes of the statement may be assumed to be the character table of the symmetric group $\mathfrak S_n$ (see Definition~\ref{MM} and Remark~\ref{remark:M} for details).

\Con\label{ourgradedconjecture}
For $\ell\geq 2$ and $n\geq 0$, 
we have the following unimodular equivalence over $\MA$:
\begin{align}
M_n\DIAG(\{J^{v}_{\ell}(\lambda)\mid\lambda\in\PAR(n)\})M_n^{-1}\CONG{\MA} \DIAG(\{I^{v}_{\ell}(\lambda)\mid\lambda\in\PAR(n)\}).
\label{ourgradedconjectureshiki}
\end{align}
\encon

In \S\ref{sec:3}, we will show that Conjecture~\ref{ourgradedconjecture} implies Conjecture~\ref{gradedKORcon} (see Corollary~\ref{mainimplication}). As is mentioned above, this generalizes a reduction  for the ungraded case proved in~\cite{Hil, BH}. 

\subsection{Evidence for Conjecture~\ref{ourgradedconjecture}}


Although there is no \emph{a priori} reason to assert that $C^v_{\MH_n(\corrr{\ell};\qq{\ell})}$ is unimodularly equivalent to a
diagonal matrix since $\MA$ is not a principal ideal domain (PID, for short),
we can give 
evidence that such an equivalence is likely to exist, 
which suggests that a hidden structure lies behind it and that one is unlikely to see this structure just by considering the ungraded case. 

\Th\label{maintheorem}
For $\ell\geq 2$ and $n\geq 0$, 
let $X$ and $D$ denote the matrices on the left-hand and right-hand sides of~\eqref{ourgradedconjectureshiki}. 
Then, we have
\bna
\item\label{ASYc} $X\equiv_{\Q[v,v^{-1}]} D$;
\item\label{subs} for any $0\ne \theta\in\Q$, we have $X|_{v=\theta} \equiv_{\Z[\theta, \theta^{-1}]} D|_{v=\theta}$.
\ee
Hence, 
the unimodular equivalence of Conjecture~\ref{gradedKORcon} holds over 
$\mathbb Q[v,v^{-1}]$ and holds over $\Z[\theta,\theta^{-1}]$ when one substitutes any $\theta\in \Q^{\times}$ for $v$.
\enth

The last statement follows from parts~\eqref{ASYc} and~\eqref{subs} due to Corollary~\ref{cor:sp}. 

\Rem
We note the following consequence and special case:
\bna
\item 
Combined with Proposition~\ref{gradedHillcon},
Theorem \ref{maintheorem} (\ref{ASYc}) settles
affirmatively a conjecture of {\asy} (\cite[Conjecture 8.2]{ASY}) and
further generalizes it to the case of an arbitrary $\ell\geq 2$, not necessarily a prime.
\item The case $\theta=1$ of Theorem \ref{maintheorem} (\ref{subs}) corresponds to the KOR conjecture (Theorem~\ref{thm:kor}). 
\ee
\enrem

Our proof of Theorem~\ref{maintheorem} relies on the fact that the equivalences in the theorem are over PIDs (see Remark~\ref{locgol}). In part, the proof is a generalization of the one in~\cite{Evs}. 

Since $\MA$ is $2$-dimensional,
it appears that completely new ideas will be needed to prove a unimodular equivalence over $\MA$. 
In particular, while the ungraded version of Conjecture~\ref{ourgradedconjecture} is easily reduced to the case when $\ell$ is a prime power (see~\cite{Hil}), there is no such apparent reduction in the graded case.
The authors hope that this paper will help advertise Conjecture \ref{ourgradedconjecture} (and its meaning) to a wide audience not restricted to representation theorists, as
the conjecture is stated purely in the language of combinatorics and linear algebra.

\subsection{Organization of the paper}
In \S\ref{sec:Mn} we introduce the matrix $M_n$, which is the table of values of Young permutation characters of the symmetric group $\mathfrak{S}_n$.
We also introduce a ``$p$-local'' and a multicolored version of $M_n$, and we prove a number of integrality results about these matrices that are needed later. 
In \S\ref{sec:3}, we show how Conjecture~\ref{gradedKORcon} may be interpreted in terms of certain representations of quantum groups. We prove Theorem~\ref{gradedcartan}, which shows that the graded Cartan matrix 
$C_{\MH_n(\corrr{\ell};\qq{\ell})}^v$ (or $C^v_{\F_p\mathfrak{S}_n}$) is unimodularly equivalent to a block-diagonal matrix with blocks of the form given by the left-hand side of~\eqref{ourgradedconjectureshiki}. Using this, we show that Conjecture~\ref{ourgradedconjecture} implies Conjecture~\ref{gradedKORcon}. 
Theorem~\ref{maintheorem} is proved in \S\ref{sec:4} and \S\ref{sec:5}. In \S\ref{sec:4}, we prove Theorem~\ref{maintheorem} (\ref{ASYc}) and reduce Theorem~\ref{maintheorem} (\ref{subs}) to Theorem~\ref{meidai}, which asserts a certain unimodular equivalence over the local ring $\Z_{(p)}$ and is proved in \S\ref{sec:5}. 
In \S\ref{fittingideal} (and \S\ref{varequi}), we discuss unimodular equivalences over arbitrary commutative rings and possible results that would be stronger than Theorem~\ref{maintheorem} but weaker than Conjecture~\ref{ourgradedconjecture}, including possible
further evidence 
in terms of equivalences over PIDs.

\subsection{Notation and conventions}\label{NoCo}
\subsubsection{Commutative rings}\label{CoAl}
All commutative rings are assumed to contain a multiplicative identity, and homomorphisms 
between commutative rings are assumed to respect those identities. 
We denote by $\MSPEC(R)$ the set of maximal ideals of a commutative ring $R$. 
\subsubsection{Matrices}\label{matconv}
Let $\MAA$ be a commutative ring. For any integer $\ell\ge 0$, we denote by $\MAT_{\ell} (\MAA)$ the algebra of all $R$-valued $\ell\times \ell$-matrices.
More generally, $\MAT_S (\MAA)$ is the algebra of $S\times S$-matrices for any finite set $S$.  
For a finite set $S$, $1_S$ denotes the identity $S\times S$-matrix.
For an assignment $S\to R$, $s \mapsto r_s$, 
we denote by $\DIAG(\{ r_s \mid s\in S \})$ the diagonal matrix with the $(s,t)$-entry equal to $\delta_{st} r_s$ for all $s,t\in S$. 
We often denote by $M_{rs}$ the $(r,s)$-entry of a matrix $M$. 
If $S= \bigsqcup_i S_i$ is a disjoint union and $M_i \in \MAT_{S_i} (\MAA)$ for each $i$, then $M=\bigoplus_i M_i$ is the block-diagonal matrix 
given by $M_{rs} = (M_i)_{rs}$ if $r$ and $s$ belong to the same subset $S_i$ and $M_{rs}=0$ otherwise. 
%
We say that matrices $X,Y\in\MAT_m(\MAA)$ are row (resp.\ column) equivalent over $\MAA$ 
if there exists $U\in \GL_m(\MAA)$ such that $X=UY$ (resp.\ $X=YU$).

\subsubsection{Discrete valuation rings}\label{SSSdvr}
When considering a discrete valuation ring $\MAA$
with valuation $\nu\colon K^{\times}\twoheadrightarrow\Z$, where $K$ is the field of fractions of $\MAA$,
we set $\nu(0)=\infty$ where $\infty$ is a symbol satisfying $\infty>c$ for all $c\in\Q$.
For a prime $p$, the valuation $\nu_p\colon \Q^{\times}\twoheadrightarrow\Z$ is defined by $\nu_p (p^m a/b) = m$ for $m\in \Z$ and $a,b\in \Z\setminus p\Z$.
It corresponds to the discrete valuation ring $\Z_{(p)}=\{a/b\in\Q\mid b\not\in p\Z\}$.

\subsubsection{Integers}\label{subsubsec:int}
We write $\N= \{0,1,2,\ldots \}$ and $\PRIMES$ for the set of all prime numbers. For $n\geq 1$, we denote by
$\PRS(n)$ the set of all prime divisors of $n$. 
For $n\geq 1$ and a subset $\Pi\subseteq\PRIMES$, we define
the $\Pi$-part of $n$ by $n_{\Pi}=\prod_{p\in\Pi}p^{\nu_p(n)}$.
We write $\Pi' = \PRIMES \setminus \Pi$ and $p' = \PRIMES \setminus \{ p \}$ for all $p\in \PRIMES$.
For $a,b\geq 1$, $\GCD(a,b)$ is the greatest common divisor of $a$ and $b$. 

\subsubsection{Quantum rings}\label{subsubsec:qi}
Let $v$ be an indeterminate. In much of the paper, we work over the field $\cor=\Q(v)$ and its subring $\MA=\Z[v,v^{-1}]$.
The $\Q$-algebra involution $\BAR\colon \cor\to\cor$ is defined by $\BAR(v) = v^{-1}$.
For $t\in \Z$, we write $\INFL_t\colon\MA\to\MA$ for the ring homomorphism given by $v\mapsto v^t$.
For $m\ge 1$ and $n\in \Z$, the quantum integer $[n]_m$ is defined by 
$[n]_{m}=(v^{mn}-v^{-mn})/(v^{m}-v^{-m})\in\MA$. Note that $[n]_m|_{v=1} = n$. 
We set $[n]_m ! = [n]_m [n-1]_m \cdots [1]_m$. 
For a field $\corr$ and $q\in\corr^{\times}$,
the quantum characteristic of $q$ is defined by $\QCHAR_{q}\corr=\min\{k\geq 1\mid [k]|_{v=q}=0\}$ if the set on the right-hand side is non-empty and is set to be $0$ otherwise.

\subsubsection{Groups and generalized characters}\label{SSSgroups}
Let $G$ be a finite group. If $R$ is a subring of $\C$, we say that a function $\chi \colon G \to \C$ is an \emph{$R$-generalized character} of $G$ if 
$\chi$ belongs to the $R$-span of the irreducible characters of $G$. By a \emph{generalized character} we mean a $\Z$-generalized character. 
If $g,h \in G$, we write $g \equiv_G h$ if $g$ and $h$ are $G$-conjugate. If $p$ is a prime, then, as usual, $g_p, g_{p'} \in \langle g\rangle
\subseteq G$ are 
the $p$-part and the $p'$-part of $g$ respectively, so that $g=g_{p} g_{p'} = g_{p'} g_p$, the order of $g_p$ is a $p$-power and the order of $g_{p'}$ is prime to $p$. 

\subsubsection{Partitions}\label{SSSPartitions}
We write $\varnothing$ for the empty partition.
For a partition $\lambda=(\lambda_1,\lambda_2,\ldots)$, 
we define $m_k(\lambda)=|\{i\geq 1\mid \lambda_i=k\}|$ for $k\geq 1$. Also, $\ell(\lambda)=\sum_{i\geq 1}m_i(\lambda)$ and $|\lambda|=\sum_{i\geq 1}\lambda_i$.
We denote by $\PAR(n)$ (resp.\ $\CPAR_{s}(n),\RPAR_{s}(n)$) the set of 
all (resp.\ $s$-class regular, $s$-regular) partitions of $n\geq 0$. 
Recall that, for $s\geq 1$, a partition $\lambda$ is called 
\bnum
\item \emph{$s$-class regular} if we have $m_{ks}(\lambda)=0$ for all $k\geq 1$,
\item \emph{$s$-regular} if we have $m_{k}(\lambda)<s$ for all $k\geq 1$.
\ee
We put $\PAR=\bigsqcup_{n\geq 0}\PAR(n)$ and 
$\PAR_m(n)=\{(\lambda^{(i)})_{i=1}^{m}\in\PAR^m\mid \sum_{i=1}^{m}|\lambda^{(i)}|=n\}$ for $m,n\geq 0$.

For $n\geq 0$, $p\in\PRIMES$ and $\nu\in\CPAR_{p}(n)$, we define
$\PAR_p(n,\nu)=\{\lambda\in\PAR(n)\mid  \sum_{s\geq 0}m_{jp^s}(\lambda)p^s=m_j(\nu) \, \forall j\in\N\setminus p\Z\}$.
Further, 
$\POW_p(n)=\PAR_{p}(n,(1^n))$ and $\POW_p=\bigsqcup_{n\geq 0}\POW_p(n)$ is the set of the partitions with all parts being powers of $p$.

For $\lambda,\mu\in\PAR$, the partition $\lambda+\mu$ is defined by $m_i(\lambda+\mu)=m_i(\lambda)+m_i(\mu)$ for $i\geq 1$.

\vskip 3mm

\noindent{\bf Acknowledgments.} S.T. thanks Yuichiro Hoshi, Yoichi Mieda and Hiraku Kawanoue for discussions on \S\ref{fittingideal}.
In particular, Theorem~\ref{kawanouethm} is due to Kawanoue (see Remark \ref{kcont}).

\section{The matrix $M_n$}\label{sec:Mn}
\subsection{Definition of $M_n$}\label{SSSsym}
As usual, let $\Lambda=\bigoplus_{n\geq 0}\varprojlim_{m\geq 0}\Z[u_1,\ldots,u_m]^{\mathfrak{S}_m}_n$
be the ring of symmetric functions
(see~\cite[\S6]{Ful} or~\cite[\S I.2]{Mac}) where $\Z[u_1,\ldots,u_m]_n$ is the set of homogeneous polynomials of degree
$n$. 

The ring $\Lambda$ is categorified by the module categories $\{\MOD{\Q\mathfrak{S}_n}\}_{n\geq 0}$.
More precisely, let $\chi_V$ denote the character afforded by a module $V\in\MOD{\Q\mathfrak{S}_n}$. 
For $\mu\in\PAR$, consider the power sum symmetric function $p_{\mu}=\prod_{i=1}^{\ell(\mu)}p_{\mu_i}$, where
$p_k=\sum_{j\geq 1}u_j^{k}$ for $k\geq 1$. Let $C_{\mu}$ be  the
conjugacy class of elements of cycle type $\mu$ in  
$\mathfrak{S}_n$. For $\mu\in \PAR$, let 
\begin{equation}\label{eq:z}
 z_{\mu}=\prod_{i\geq 1}m_i(\mu)!\cdot i^{m_i(\mu)},
\end{equation}
so that $\#C_{\mu}= |\mu|!/z_{\mu}$.  
Then the following character map is an isometry (see~\cite[\S7.3]{Ful}):
\begin{align}\label{charactermap}
\CATISO\colon \bigoplus_{n\geq 0}\KKK(\MOD{\Q\mathfrak{S}_n})\isoto \Lambda,\quad
[V]\MAPSTO \sum_{\mu\in\PAR(n)}\frac{1}{z_\mu}\chi_{V}(C_{\mu})p_{\mu}, 
\end{align}
where we write $\chi_V (C_{\mu})$ for the value of $\chi_V$ on an arbitrary element of $C_{\mu}$. 

\Def\label{MM}
Let $\lambda,\mu \in \PAR(n)$.
Consider the parabolic subgroup 
\[
\mathfrak{S}_{\lambda}=\AUT(\{1,\ldots,\lambda_1\})\times \AUT(\{\lambda_1+1,\ldots,\lambda_1+\lambda_2\})\times \cdots\cong \otimes_{i\geq1}\mathfrak{S}_{\lambda_i}
\] 
of $\mathfrak{S}_n$, and let $\TRIV_{\mathfrak{S}_{\lambda}}$ be its trivial representation.
We set $M_{\lambda,\mu} = \chi_{\IND_{\mathfrak{S}_{\lambda}}^{\mathfrak{S}_n}\TRIV_{\mathfrak{S}_{\lambda}}} (C_{\mu})$ and put 
$M_n=(M_{\lambda,\mu})\in\MAT_{\PAR(n)}(\Z)$. 
\edf

\Rem\label{remark:M}
Recall the complete symmetric function $h_{\mu}=\prod_{i\geq 1}h_{\mu_i}$ for $\mu\in\PAR$ where
\begin{align}
\sum_{n\geq 0}h_nt^n = \prod_{i\geq
1}(1-u_it)^{-1}=\prod_{r=1}^{\infty}\exp\left(\frac{p_rt^r}{r}\right).
\label{hphenkan}
\end{align}
There is a well-known identity 
$\CATISO([\IND_{\mathfrak{S}_{\lambda}}^{\mathfrak{S}_n}\TRIV_{\mathfrak{S}_{\lambda}}])=h_{\lambda}$ for $\lambda\in\PAR(n)$ (see~\cite[\S7.2, Lemma 4]{Ful}).
Further, we have 
\begin{align}
h_{\lambda}=\sum_{\mu\in\PAR(n)}\frac{1}{z_{\mu}}M_{\lambda,\mu}p_{\mu},\quad
p_{\lambda}=\sum_{\mu\in\PAR(n)}M_{\mu,\lambda}m_{\mu}
\label{hp}
\end{align}
for $\lambda\in\PAR(n)$, where $m_{\mu}$ is the monomial symmetric function (i.e., the function whose image 
in $\Z[u_1,\ldots,u_m]_n$ for $m\geq \ell(\lambda)$ is the sum of the elements of the orbit of the
monomial $\prod_{j=1}^{\ell(\mu)}u_j^{\mu_j}$ under the action of $\mathfrak
S_m$ on the variables); see~\cite[\S6, (11), (12)]{Ful}.
Using the second identity~\eqref{hp}, we see that $M_{\lambda,\mu}$ has the following explicit combinatorial descriptions:
\bna
\item $M_{\lambda,\mu}$ is the coefficient of
$\prod_{j=1}^{\ell(\lambda)}u_j^{\lambda_j}$ in $\prod_{i\geq
1}(u_1^i+\cdots+u_{\ell(\lambda)}^i)^{m_i(\mu)}$,
\item $M_{\lambda,\mu}=\#\mathcal{M}_{\lambda,\mu}$
where 
\[
\mathcal{M}_{\lambda,\mu}=\{f\colon \{1,\ldots,\ell(\mu)\}\to\{1,\ldots,\ell(\lambda)\}\mid \textstyle\sum_{j\in f^{-1}(i)}\mu_j=\lambda_i
\text{ whenever } 1\le i\le l(\lambda)\}.
\]
\ee
\enrem

\Rem\label{remark:M2}
It is well known that the $\Z$-span of $\{ \chi_{\IND_{\mathfrak{S}_{\lambda}}^{\mathfrak S_n} \TRIV_{\mathfrak S_{\lambda}}} \mid \lambda\in \PAR(n) \}$ is the whole set of generalized characters of $\mathfrak S_n$ (see ~\cite[\S7.2, Corollary]{Ful}); equivalently, the matrix $M_n$ is row equivalent over $\Z$ to the character table of $\mathfrak S_n$ (in which, as usual, rows correspond to irreducible characters and columns to conjugacy classes, labeled by their cycle types). Therefore, as we claimed in \S\ref{subsec:graded_kor}, the matrix on the left-hand side of~\eqref{ourgradedconjectureshiki} stays in the same unimodular equivalence class if one replaces $M_n$ by the character table of $\mathfrak S_n$. 
\enrem

In the remainder of this section, we prove a number of results on the matrix $M_n$ and some of its analogues, mainly of a combinatorial nature.
Proposition~\ref{nuestimateM} will not be used until \S\ref{CASETWO}. The results in \S\ref{subsec:22} are used in \S\ref{sec:4} and \S\ref{sec:5}, whereas the results of \S\ref{subsec:23} are needed in \S\ref{sec:3}. 


\Prop\label{nuestimateM}
Let $n\geq 0$ and let $\lambda,\mu\in\PAR(n)$.
\bna
\item\label{nuestimateMa} $M_{\lambda,\lambda}=\prod_{j\geq 1}m_j(\lambda)!$ and $M_{\lambda,\lambda}$ divides $M_{\lambda,\mu}$;
\item\label{nuestimateMb} $\ell(\lambda)\leq \ell(\mu)$ if $M_{\lambda,\mu}>0$;
\item\label{nuestimateMc} Let $p\geq 3$ be a prime, and assume that $M_{\lambda,\mu}>0$ and $\lambda\ne \mu$. Then
$\nu_p(M_{\lambda,\mu})>\ell(\lambda)-\ell(\mu)+\sum_{j\geq1}\nu_p(m_{j}(\mu)!)$.
\ee
\enprop

\Proof
(\ref{nuestimateMa}) and (\ref{nuestimateMb}) follow immediately from the combinatorial descriptions in Remark~\ref{remark:M}. 
To prove (\ref{nuestimateMc}),
let $C$ be the set of maps $c\colon \{1,\ldots,\ell(\lambda)\}\to\PAR\setminus\{\varnothing\}$ such that
$\sum_{k=1}^{\ell(\lambda)}c(k)=\mu$ and $|c(k)|=\lambda_k$ for $1\leq k\leq \ell(\lambda)$.
For $c\in C$, we define $\mathcal{M}^c_{\lambda\mu}$ to be the set of maps $f\in \mathcal M_{\lambda\mu}$ such that, whenever $1\le k\le \ell(\lambda)$, 
there is a multiset equality
\[
 \{\mu_j\mid j\in f^{-1}(k)\}=\{c(k)_j\mid 1\leq j\leq \ell(c(k))\}.
\]

It is clear that $\mathcal{M}_{\lambda,\mu}=\bigsqcup_{c\in C}\mathcal{M}^c_{\lambda,\mu}$ (thus, we have $C\ne\emptyset$) and
\begin{align*}
\#\mathcal{M}^c_{\lambda,\mu}=\prod_{j\geq 1}\binom{m_{j}(\mu)}{m_{j}(c(1)),m_{j}(c(2)),\ldots,m_{j}(c(\ell(\lambda)))}.
\end{align*}

It suffices to prove that $\nu_p(\#\mathcal{M}^c_{\lambda,\mu})> \ell(\lambda)-\ell(\mu)+\sum_{j\geq0}\nu_p(m_{j}(\mu)!)$ for $c\in C$.
By Lemma \ref{nuestimate}, 
\begin{align*}
\nu_p(\#\mathcal{M}^c_{\lambda,\mu})-\nu_p(m_{j}(\mu)!)-\ell(\lambda)+\ell(\mu)
=\sum_{k=1}^{\ell(\lambda)}\left( \ell(c(k))-\sum_{j\geq1}\nu_p(m_{j}(c(k))!)-1 \right)\geq 0
\end{align*}
and the equality holds exactly when $\ell(c(k))=1$ for $1\leq k\leq \ell(\lambda)$, i.e., when $\lambda=\mu$.
\QED

\Lemma\label{nuestimate}
Let $p\geq 3$ be a prime and $\lambda\in\PAR\setminus\{\varnothing\}$.
We have $\ell(\lambda)-\sum_{j\geq 1}\nu_p(m_{j}(\lambda)!)\geq 1$, and
the equality holds exactly when $\ell(\lambda)=1$.
\enlemma

\Proof
Note that 
\begin{equation}\label{eq:factorial}
\nu_p(a!) = \sum_{i=1}^{\infty}\lfloor a/p^i \rfloor \leq \sum_{i=1}^{\infty}a/p^i=a/(p-1)
\end{equation} 
for $a\geq 0$.
Thus, 
\begin{align*}
\ell(\lambda)-\sum_{j\geq 1}\nu_p(m_{j}(\lambda)!)\geq \left(1-1/(p-1)\right)\ell(\lambda)>1
\end{align*}
when $\ell(\lambda)\geq 3$.
When $\ell(\lambda)=1,2$, we have $\nu_p(m_{j}(\lambda)!)=0$ for all $j\geq 1$.
\QED

\subsection{$p$-local version $N^{(p)}_{n}$ of $M_n$}\label{subsec:22}\label{SSSploc}

As in~\cite[\S4]{Evs}, we consider a submatrix $N_n^{(p)}$ of $M_n$ and use it to construct a certain block-diagonal matrix $L_n^{(p)}$, which is row equivalent over $\Z_{(p)}$ to $M_n$, for any fixed prime $p$.

\Def
For $p\in\PRIMES$ and $n\geq 0$, we define $N^{(p)}_{n}=M_n|_{\POW_p(n)\times \POW_p(n)}$ and 
\[
L^{(p)}_{n}=\bigoplus_{\nu\in\CPAR_p(n)}\bigotimes_{j\in\N\setminus p\Z}N^{(p)}_{m_j(\nu)}.
\]
\edf

We regard $L^{(p)}_{n}$ as an element of $\MAT_{\PAR(n)}(\Z)$ by using the following identification:
\bna
\item $\PAR(n)=\bigsqcup_{\nu\in\CPAR_p(n)}\PAR_p(n,\nu)$,
\item $\PAR_p(n,\nu)\isoto\prod_{j\in\N\setminus p\Z}\POW_p(m_j(\nu)), \; \lambda\MAPSTO (\lambda^{(j)})_{j\in\N\setminus p\Z}$
where $m_{p^s}(\lambda^{(j)})=m_{jp^s}(\lambda)$.
\ee

\Prop\label{reductiontoN}
Let $p\in\PRIMES$. For a $\Z_{(p)}$-algebra $\MAA$ and a family of homomorphisms $(r_j\colon R\to R)_{j\in\N\setminus p\Z}$, 
assume that
\bnum
\item\label{katei26i} there are maps $f,g\colon \PAR\to\MAA$ such that
\begin{align*}
f(\lambda)=\prod_{j\in\N\setminus p\Z}r_j(f(\lambda^{(j)})),\quad
g(\lambda)=\prod_{j\in\N\setminus p\Z}r_j(g(\lambda^{(j)}))
\end{align*}
for all $k\ge 0$, $\nu\in\CPAR_p(k)$ and $\lambda\in\PAR_p(k,\nu)$, 
where the assignment $\lambda\mapsto (\lambda^{(j)})_{j\in\N\setminus p\Z}$ is defined as above.
\item\label{katei26ii} for all $n\geq 0$, $M_n\DIAG(\{f(\lambda)\mid\lambda\in\PAR(n)\})M_{n}^{-1}$ is $\MAA$-valued.
\ee
Then, we have 
\bna
\item\label{ketsuron26a} $N^{(p)}_{k}\DIAG(\{f(\lambda)\mid\lambda\in\POW_p(k)\})\NPI{k}$ is $\MAA$-valued for all $k\geq 0$,
\item\label{ketsuron26b} 
For a $\Z_{(p)}$-algebra $\MAA'$ with a homomorphism $\phi\colon \MAA\to\MAA'$, the following implication holds:
\begin{align*}
{} &{} 
\forall k\geq 0,\phi(N^{(p)}_{k}\DIAG(\{f(\lambda)\mid\lambda\in\POW_p(k)\})\NPI{k}) \CONG{\MAA'} \phi(\DIAG(\{g(\lambda)\mid\lambda\in\POW_p(k)\}))\\
&\Longrightarrow
\forall n\geq 0,\phi(M_n\DIAG(\{f(\lambda)\mid\lambda\in\PAR(n)\})M_{n}^{-1}\CONG{\MAA'}\phi(\DIAG(\{g(\lambda)\mid\lambda\in\PAR(n)\})).
\end{align*}
\ee
\enprop

\Proof
By~\cite[Lemma 4.8]{Evs}, the matrices $M_n$ and $L^{(p)}_{n}$ are row equivalent over $\Z_{(p)}$ and hence over $R$. 
Thus, by (\ref{katei26ii}) we have
\begin{align*}
M_n\DIAG(\{f(\lambda)\mid\lambda\in\PAR(n)\})M_{n}^{-1}
\CONG{\MAA}
L^{(p)}_{n}\DIAG(\{f(\lambda)\mid\lambda\in\PAR(n)\})(L^{(p)}_{n})^{-1}.
\end{align*}
By (\ref{katei26i}), the right-hand side is just
\begin{align*}
\bigoplus_{\nu\in\CPAR_p(n)}\bigotimes_{j\in\N\setminus p\Z}N^{(p)}_{m_j(\nu)}\DIAG(\{r_j(f(\lambda^{(j)}))\mid\lambda^{(j)}\in\POW_p(m_{j}(\nu))\})\NPI{m_j(\nu)}.
\end{align*}
We have shown that $N^{(p)}_{n}\DIAG(\{f(\lambda)\mid\lambda\in\POW_p(n)\})\NPI{n}$
is a block submatrix of an $\MAA$-valued matrix which is unimodularly equivalent to $M_n\DIAG(\{f(\lambda)\mid\lambda\in\PAR(n)\})M_{n}^{-1}$
over $\MAA$ (note that, by (i), $r_1 (f(\lambda)) = f(\lambda)$ for $\lambda\in \POW_p(n)$). Thus, (\ref{ketsuron26a}) is proved.
Part (\ref{ketsuron26b}) follows from the above equivalences and hypothesis (i). 
\QED

Our next aim is to prove an integrality result (Proposition \ref{intp}), which will be used in \S\ref{CASEONE}.

\Def\label{apn}
Let $p\in\PRIMES$.
For a sequence $\theta=(\theta_j)_{j\geq 0}\in\Z_{(p)}^{\N}$ and $n\geq 0$, we define 
\[
a^{(p)}_\theta(n)=\sum_{\nu\in\POW_p(n)}\frac{1}{z_\nu}
\prod_{j\geq 0}\theta_j^{m_{p^j}(\nu)}.
\]
\edf

\Lemma\label{alemma}
Let $p\in\PRIMES$. For any $\theta\in\Z_{(p)}^{\N}$ and $n\geq 0$, we have
\bna
\item\label{alemmaa} $a_{\theta+\theta'}^{(p)}(n)=\sum_{k=0}^{n}a^{(p)}_{\theta}(k)a^{(p)}_{\theta'}(n-k)$, 
where $(\theta+\theta')_j:=\theta_j+\theta'_j$ for $j\geq 0$,
\item\label{alemmab} $a^{(p)}_\theta(n)\in\Z_{(p)}$ if $\nu_{p}(\theta_j)\geq j+1$ for all $j\geq 0$,
\item\label{alemmac} $a^{(p)}_{(\theta)}(n)\in\Z_{(p)}$ if there exist $s\in \Z_{\ge 1}$ and $c\in \Z_{(p)}$ such that 
$\theta_j = s c^{p^j}$ for all $j\ge 0$. 
\ee
\enlemma

\Proof
Consider the generating function $A_{\theta} = \sum_{n\ge 0} a_{\theta}^{(p)} (n) t^n$. By a straightforward calculation 
similar to the one in the proof of~\cite[Equation (I.2.14)]{Mac}, we obtain the identity
$A_{\theta} = \exp (\sum_{j\ge 0} p^{-j} \theta_j t^{p^j})$.
Hence,
$A_{\theta+\theta'} = A_{\theta}A_{\theta'}$, and part (\ref{alemmaa}) follows by equating coefficients in $t^n$. 
Part (\ref{alemmab}) follows from the identity
\begin{align*}
a^{(p)}_\theta(n)=\sum_{\nu\in\POW_p(n)}\prod_{j\geq 0}\frac{1}{m_{p^j}(\nu)!}\left(\frac{\theta_j}{p^j}\right)^{m_{p^j}(\nu)}.
\end{align*} and the
inequality $\nu_p(d!)\leq d$ (see~\eqref{eq:factorial}). 

To prove (\ref{alemmac}), we recall a corollary of Brauer's characterization of characters.
Let $G$ be a finite group.
Then the characteristic function 
of a $p'$-section $\SECTION_{p'}(x):=\{y\in G \mid y_{p'}\CONG{G}x\}$ 
of any $p'$-element $x\in G$
is an $\MO$-generalized character of $G$ (see ~\cite[Lemma 8.19]{Isa}) for a certain DVR $\MO$
with $\Z_{(p)}\subseteq\MO \subseteq \mathbb C$. In particular, 
the characteristic function of $\SECTION_{p'}(1_{\mathfrak{S}_n})=\bigsqcup_{\nu\in\POW_p(n)}C_{\nu}$
is an $\MO$-generalized character of $\mathfrak{S}_n$. 

We denote by $\langle \cdot , \cdot \rangle_G$ the usual inner product on the complex-valued class functions on $G$, so that 
 $\{\chi_V\mid V\in\IRR(\MOD{\C G})\}$ is an orthonormal basis. 
Due to (\ref{alemmaa}), we may assume that $s=1$, so that $\theta_j = c^{p^j}$ for all $j$. 
We have 
\[
a^{(p)}_{\theta}(n)=
\sum_{\nu\in \POW_p(n)} z_{\mu}^{-1} c^n = 
c^n \langle \chi_{\TRIV_{\mathfrak{S}_n}}|_{\bigsqcup_{\nu\in\POW_p(n)}C_{\nu}}
,\chi_{\TRIV_{\mathfrak{S}_n}}\rangle_{\mathfrak{S}_n} \in \Q \cap \MO = \Z_{(p)}. \qedhere
%
\]
\QED

\Prop\label{usegeneralizedcharacter}
Let $\MAA\subseteq\C$ be a ring, and consider a map $\xi\colon \PAR(n)\to\C$ be a map for some $n\geq 0$.
If the class function $\xi^{\mathsf{cl}}$ defined by $\xi^{\mathsf{cl}}(C_\lambda)=\xi(\lambda)$
for $\lambda\in\PAR(n)$ is an $\MAA$-generalized character of $\mathfrak{S}_n$,
then $M_{n}\DIAG(\{\xi(\lambda)\mid\lambda\in\PAR(n)\})M_{n}^{-1}$ is $\MAA$-valued.
\enprop

\Proof
Let $T_n=(\chi_{V}(C_\lambda))_{V\in\IRR(\MOD{\Q\mathfrak{S}_n}), \, \lambda\in\PAR(n)}$ be the character table of $\mathfrak{S}_n$.
Then, for $V,W\in\IRR(\MOD{\Q\mathfrak{S}_n})$, 
the $(V,W)$-entry of $T_n\DIAG(\{\xi(\lambda)\mid\lambda\in\PAR(n)\})T_n^{-1}$ is
equal to $\langle \xi^{\mathsf{cl}}\chi_V,\chi_W\rangle_{\mathfrak{S}_n}$. Indeed, we have 
\[
 \langle \xi^{\mathsf{cl}}\chi_V,\chi_W\rangle_{\mathfrak{S}_n} = 
\sum_{\lambda\in\PAR(n)} \frac{1}{z_{\lambda}} \chi_V (C_{\lambda}) \xi(\lambda) \chi_W (C_{\lambda}),
\]
and $z_{\lambda}^{-1} \chi_W (C_{\lambda})$ is the $(\lambda,W)$-entry of $T_n^{-1}$ due to the orthogonality relations. 
The result follows since $M_n$ and 
$T_n$ are row equivalent over $\Z$ (see Remark~\ref{remark:M2}).
\QED

\Cor\label{usegeneralizedcharacter2}
Let $p\in\PRIMES$ and $n\geq0$. For a map $\xi\colon \POW_p(n)\to\C$,
if the class function $\xi^{\mathsf{cl}}$ defined by
\begin{align*}
\xi^{\mathsf{cl}}(C_\lambda)=
\begin{cases}
\xi(\lambda) & \text{if }\lambda\in\POW_p(n), \\
0 & \text{if }\lambda\in\PAR(n)\setminus\POW_p(n)
\end{cases}
\end{align*}
is a  $\Z_{(p)}$-generalized character of $\mathfrak{S}_n$, then $N^{(p)}_{n}\DIAG(\{\xi(\lambda)\mid\lambda\in\POW_p(n)\})\NPI{n}$ is $\Z_{(p)}$-valued.
\encor

\Proof
Put $\HM{n}=\bigoplus_{\nu\in\CPAR_{p}(n)}M_n|_{\PAR_p(n,\nu)\times \PAR_p(n,\nu)}\in\MAT_{\PAR(n)}(\Z)$.
Then $M_n$ and $\HM{n}$ are row equivalent over $\Z_{(p)}$ by~\cite[Lemma 4.6]{Evs}. Thus, by Proposition~\ref{usegeneralizedcharacter},
$\HM{n}\DIAG(\{\xi^{\mathsf{cl}}(C_{\lambda})\mid\lambda\in\PAR(n)\})\HMI{n}\in\MAT_{\PAR(n)}(\Z_{(p)})$.
Now $N^{(p)}_{n}\DIAG(\{\xi(\lambda)\mid\lambda\in\POW_p(n)\})\NPI{n}$ is simply the $\POW_p (n) \times \POW_p(n)$-submatrix of this matrix, so 
the result follows. 
\QED

\Prop\label{intp}
Let $p\in\PRIMES$ and $n\geq 0$, $\ell\geq 2$ be integers. Put $r=\nu_p(\ell)$.
Then, for any $a/b\in \Z_{(p)}$ with $a,b\in\Z\setminus p\Z$ and $a^2-b^2\in p\Z$, we have
\begin{align*}
N^{(p)}_{n}\DIAG(\{p^{-r\ell(\lambda)}\prod_{j\geq 0}[\ell]^{m_{p^j}(\lambda)}_{p^{r+j}}|_{v=a/b}\mid\lambda\in\POW_p(n)\})\NPI{n}\in\MAT_{\POW_p(n)}(\Z_{(p)}).
\end{align*} 
\enprop

\Proof
Put $\theta=(\theta_j)_{j\geq 0}\in\Z_{(p)}^{\N}$ where $\theta_j=p^{-r}[\ell]_{p^{r+j}}|_{v=a/b}$.
Consider the map $\xi\colon \POW_p(n)\to\Q$ given by $\nu\mapsto \prod_{j\geq 0}\theta_j^{m_{p^j}(\nu)}$. 
By Corollary \ref{usegeneralizedcharacter2}, it is enough to show that $\xi^{\mathsf{cl}}$ is a $\Z_{(p)}$-generalized character of $\mathfrak{S}_n$.
By Frobenius reciprocity, for all $\lambda \in \PAR(n)$ we have
\begin{align*}
{\textstyle{
\langle \xi^{\mathsf{cl}},\chi_{\IND_{\mathfrak{S}_{\lambda}}^{\mathfrak{S}_n}\TRIV_{\mathfrak{S}_{\lambda}}}\rangle_{\mathfrak{S}_n}=
\langle \Res_{\mathfrak{S}_{\lambda}}^{\mathfrak{S}_n}\xi^{\mathsf{cl}},\chi_{\TRIV_{\mathfrak{S}_{\lambda}}}\rangle_{\mathfrak{S}_{\lambda}}=
\prod_{i=1}^{\ell(\lambda)}a_{\theta}^{(p)} (\lambda_i).}}
\end{align*}
Therefore, since $\{\chi_{\IND_{\mathfrak{S}_{\lambda}}^{\mathfrak{S}_n}\TRIV_{\mathfrak{S}_{\lambda}}}\mid \lambda\in\PAR(n)\}$ is a $\Z$-basis of
the abelian group of generalized characters of $\mathfrak{S}_n$, 
it suffices to show that $a_{\theta}^{(p)}(k)\in\Z_{(p)}$ for all $k\geq 0$.

Let $\theta''_j= \ell_{p'} (a/b)^{-(\ell-1)p^{r+j}}$ and $\theta'_j = \theta_j-\theta''_j$ for $j\ge 0$, so that $\theta=\theta'+\theta''$. 
We know that $a_{\theta''}^{(p)}(k) \in \Z_{(p)}$ by Lemma \ref{alemma} \eqref{alemmac}.
Thus, by Lemma \ref{alemma} \eqref{alemmaa}, it is enough to show that $a_{\theta'}^{(p)}(k) \in \Z_{(p)}$.
By Lemma~\ref{alemma} \eqref{alemmab}, it will suffice to prove that $\nu_p (\theta'_j) \ge j+1$.
Note that
\begin{align*}
\theta'_j =\frac{\sum_{i=0}^{\ell-1}(a/b)^{2ip^{r+j}}-\ell}{p^r}\cdot\left(\frac{a}{b}\right)^{-(\ell-1)p^{r+j}} 
= \frac{\sum_{i=0}^{\ell-1} \left( (a/b)^{2ip^{r+j}}-1 \right)}{p^r}\cdot\left(\frac{a}{b}\right)^{-(\ell-1)p^{r+j}}.
\end{align*}
Since the assumption that $a^2-b^2\in p\Z$ implies that $a^{2ip^{r+j}}-b^{2ip^{r+j}}\in p^{1+r+j}\Z$ for all $i\geq 0$ (see e.g.~Proposition~\ref{elementarynumbertheory} and its proof), we are done.
\QED

\subsection{$\ell$-colored version $M_{\ell,d}$ of $M_n$}\label{subsec:23}\label{SSSmult}

Let $R$ be a commutative ring and $A\in \MAT_{\ell} (R)$ for some $\ell\ge 1$. 
Let $\{ v_1,\ldots, v_{\ell} \}$  be 
the standard basis of the free $R$-module $R^{\ell}$. 
Then the symmetric power $\SYM^m (R^\ell)$ has a basis 
$\{ v_{i_1} v_{i_2} \cdots v_{i_m} \mid (i_1,\ldots,i_m) \in \MULT_m (\ell)\}$ where 
\[
 \MULT_m (\ell) = \{ (i_1,\ldots,i_m)\in \Z^m \mid 1\le i_1\le \cdots \le i_m \le \ell \}. 
\]
Since $\SYM^m$ is a functor from the category of finitely generated $R$-modules to itself, 
the endomorphism of $R^\ell$ given by $A$ induces an endomorphism of $\SYM^m (R^\ell)$, and the $m$-th symmetric power
$\SYM^m (A)$ is defined to be the matrix of this endomorphism with respect to the given basis (see e.g.~\cite[Equation (3.15)]{Evs} for a more explicit description). 
Thus, $\SYM^m (A) \in \MAT_{\MULT_{m} (\ell)} (R)$. 

For $\ell, d\ge 0$, we define
\begin{align*}
\Omega_{\ell,d} = \bigsqcup_{\lambda\in\PAR(d)}\{(\lambda,(i_1,\ldots,i_{\ell(\lambda)}))\mid 
1\le i_j \le \ell \;\, \forall j \text{ and }
 \lambda_j=\lambda_{j+1}\Rightarrow i_j\le i_{j+1}\}.
\end{align*}
There is a bijection $\Omega_{\ell,d} \isoto \PAR_{\ell} (d)$
 given by $(\lambda, \ul  i) \mapsto (\lambda^{(1)}, \ldots, \lambda^{(\ell)})$ where
$\lambda^{(j)}$ consists of the parts $\lambda_k$ such that $i_k = j$ (see~\cite[Notation 3.1]{Hil}). 


\Def\label{def:Sd}
For positive integers $\ell,d$ and $A\in \MAT_{\ell}(\MA)$, we define (see \S\ref{subsubsec:qi})
\begin{align*}
\HAT{A}{d} = \bigoplus_{\lambda\in\PAR(d)}\bigotimes_{t\geq1}\SYM^{m_t(\lambda)}(\INFL_t(A)). 
\end{align*}
We may view $\HAT{A}{d}$ as an $\Omega_{\ell,d} \times \Omega_{\ell,d}$-matrix via the identification
\begin{align*}
 \bigsqcup_{\lambda \in\PAR(d)} \prod_{t\ge 1} \MULT_{m_t (\lambda)} (\ell) &\isoto \Omega_{\ell,d}, \\
 \big( (i_{t,1},\ldots, i_{t, m_t (\lambda)} ) \big)_{t\ge 1} &\mapsto (\lambda, (i_{1,1}, i_{1,2},\ldots, i_{1,m_1 (\lambda)}, i_{2,1}, i_{2,2},\ldots, i_{2, m_2(\lambda)}, \ldots )).
\end{align*}
Further, combining this with the above identification, we may (and do) view $\HAT{A}{d}$ as an element of $\MAT_{\PAR_{\ell} (d)} (\MA)$. 
\edf

\Def\label{MK}
The $\ell$-colored ring of symmetric functions is defined by 
$\HATT{\Lambda}{\ell}=\bigotimes_{t=1}^{\ell}\Lambda^{(t)}$, where each $\Lambda^{(t)}$ is a copy of $\Lambda$. We write $m_{\mu}^{(t)}$ for the image of 
$m_{\mu}$ in $\Lambda^{(t)}$ and adopt a similar convention for the functions $h_{\mu}$ and $p_{\mu}$. 
For $\ell\geq 1$ and $d\geq 0$, we define the matrices $\HATT{M}{\ell,d},\HATT{K}{\ell,d}\in\MAT_{\PAR_{\ell}(d)}(\Z)$ by the following equations:
\begin{align*}
p^{(1)}_{\lambda^{(1)}}\cdots p^{(\ell)}_{\lambda^{(\ell)}}
&= \sum_{(\mu^{(1)},\ldots,\mu^{(\ell)})\in\PAR_{\ell}(d)}
(\HATT{M}{\ell,d})_{(\lambda^{(1)},\ldots,\lambda^{(\ell)}),(\mu^{(1)},\ldots,\mu^{(\ell)})}m^{(1)}_{\mu^{(1)}}\cdots m^{(\ell)}_{\mu^{(\ell)}} \\
&= \sum_{(\mu^{(1)},\ldots,\mu^{(\ell)})\in\PAR_{\ell}(d)}
(\HATT{K}{\ell,d})_{(\lambda^{(1)},\ldots,\lambda^{(\ell)}),(\mu^{(1)},\ldots,\mu^{(\ell)})}h^{(1)}_{\mu^{(1)}}\cdots h^{(\ell)}_{\mu^{(\ell)}}.
\end{align*}
\edf

\begin{remark}\label{rem:Mell}
$M_{1,n} = \TRANS{M_n}$ by~\eqref{hp} and $
 M_{\ell,n} = {\textstyle{\bigoplus_{\substack{\sum_{i=1}^{\ell}n_i=n, n_i\geq 0}} \bigotimes_{i=1}^{\ell}}} M_{1,n_i}$.
\end{remark}

\begin{remark}\label{rem:MK}
$\HATT{M}{\ell,d}$ and $\HATT{K}{\ell,d}$ are column equivalent over $\Z$
since both of 
\begin{align*}
\textstyle
\{\prod_{i\in I}m^{(i)}_{\mu^{(i)}}\mid (\mu^{(i)})_{i\in I}\in\PAR_{\ell}(d)\},\quad
\{\prod_{i\in I}h^{(i)}_{\mu^{(i)}}\mid (\mu^{(i)})_{i\in I}\in\PAR_{\ell}(d)\}
\end{align*}
are bases of the same $\MA$-lattice of the degree $d$ part of $\HATT{\Lambda}{\ell}$ (see ~\cite[\S6, Proposition 1]{Ful}).
\end{remark}


The following result is similar to~\cite[Proposition 2.3]{Tsu} and is proved by essentially the same argument as that given in~\cite[\S5]{BK1}. We include a proof for clarity.

\Prop\label{comp} Let $\corr$ 
be a field of characteristic 0 and $I=\{ 1,\ldots, \ell\}$ for a fixed integer $\ell \ge 0$. 
We regard the polynomial ring $V=\corr[y_n^{(i)}\mid i\in I,n\geq 1]$ as a graded $\corr$-algebra via $\DEG y_n^{(i)}=n$ and
denote by $V_d$ the $\corr$-vector subspace of $V$ consisting of homogeneous elements of degree $d$ for $d\geq 0$.
Assume that 
we are given the following data:
\bna
\item\label{compa} a ring involution $\BARR\colon\corr\isoto\corr$,
\item a family of invertible matrices $A=(A^{(m)})_{m\geq 1}$ where $A^{(m)}=(a^{(m)}_{ij})_{i,j\in I}\in\GL_I(\corr)$,
\item two bi-additive forms $\langle \cdot , \cdot \rangle_S$ and $\langle \cdot, \cdot \rangle_K \colon V\times V\to \corr$ 
such that
\begin{itemize}
\item $\langle cf,g\rangle_X=\BARR(c)\langle f,g\rangle_X$ and $\langle f,cg\rangle_X=c\langle f,g\rangle_X$, 
\item $\langle 1,1\rangle_X=1$, and $\langle 1,f\rangle_X=0$ if $f\in V_d$ for some $d>0$,
\item 
$\langle my_m^{(i)}f,g\rangle_S=\langle f,\sum_{j\in I}a_{ij}^{(m)}\frac{\partial g}{\partial y^{(j)}_m}\rangle_S$ and 
$\langle my_m^{(i)}f,g\rangle_K=\langle f,\frac{\partial g}{\partial y^{(i)}_m}\rangle_K$ 
\end{itemize}
for $X\in\{S,K\}$ and $f,g\in V$, $c\in\mathbb F$, $m\geq1$, 
\item\label{lin_change} a family of new variables $(x_n^{(i)})_{i\in I,n\geq 1}$
such 
that $x_n^{(i)}-y_n^{(i)}\in \corr[y_m^{(i)}\mid 1\leq m<n]\cap V_n$ for all $n\ge 1$. 
\ee
Set $x^{(\ul  i)}_{\lambda}=\prod_{k=1}^{\ell(\lambda)}x^{(i_k)}_{\lambda_k}$ and $y^{(\ul  i)}_{\lambda}=\prod_{k=1}^{\ell(\lambda)}y^{(i_k)}_{\lambda_k}$ for $(\lambda,\ul  i)\in\Omega_{\ell,d}$, 
and define the transition matrix $P = (p_{\lambda,\mu}^{(\ul  i, \ul  j)})\in \GL_{\Omega_{\ell,d}}(\corr)$ by 
$x_{\lambda}^{(\ul  i)}={\textstyle\sum_{(\mu,\ul  j)\in\Omega_{\ell,d}} p^{(\ul  i,\ul  j)}_{\lambda,\mu}y^{(j)}_{\mu}}$.
Then the Gram matrices 
$M_S=(\langle x^{(\vph{\ul j} \ul  i)}_{\lambda},x^{(\ul  j)}_{\mu}\rangle_S)_{(\lambda,\ul i),(\mu,\ul j)\in\Omega_{\ell,d}}$ and 
$M_K=(\langle x^{(\vph{\ul j} \ul  i)}_{\lambda},x^{(\ul  j)}_{\mu}\rangle_K)_{(\lambda,\ul i),(\mu,\ul j)\in\Omega_{\ell,d}}$ are related by the identity
\begin{align}\label{iden_tity}
M_S = \BARR(P)\left({{\bigoplus_{\lambda\in\PAR(d)}\bigotimes_{t\geq 1}\SYM^{m_t(\lambda)}(A^{(t)})}}\right)\BARR(P)^{-1}M_K.
\end{align}
\enprop

\begin{proof}
Let $z_n^{(i)} = \sum_{j\in I} \sigma(a_{ij}^{(n)}) y_n^{(j)}$ for $n>0$ and $i\in I$, and define $z_{\lambda}^{(\ul i)} = \prod_{k=1}^{\ell (\lambda)} z_{\lambda_k}^{(i_k)}$ for all $(\lambda,\ul i) \in \Omega_{\ell,d}$. 
First, we will prove by induction on $d$ that, for all for all $f\in V$ and $(\lambda,\ul i)\in \Omega_{\ell,d}$, we have 
\begin{equation}\label{compclaim}
\langle y_{\lambda}^{(\ul i)}, f \rangle_S = \langle z_{\lambda}^{(\ul i)}, f \rangle_K
\end{equation} 
(cf.~\cite[Lemma 5.2]{BK1}).
We have $\langle 1, f\rangle_S = \langle 1,f \rangle_K$ for all $f\in V$, as both sides are equal to the constant term of $f$, so~\eqref{compclaim} holds when $d=0$. 
If~\eqref{compclaim} holds for some $(\lambda,\ul i)\in \bigcup_{d\ge 0} \Omega_{\ell,d}$ and all $f\in V$, 
then for all $n>0$, $i\in I$ and $f\in V$ we have 
\[
\begin{array}{lclcl}
\left\langle y_n^{(i)} y_{\lambda}^{(\ul i)}, f \right\rangle_{\! S} 
&=&
\left\langle n^{-1} y_{\lambda}^{(\ul i)}, \sum_{j\in I} a_{ij}^{(n)} \frac{\partial f} {\partial y_n^{(j)}} \right\rangle_{\! S} 
&=&
\left\langle n^{-1} z_{\lambda}^{(\ul i)}, \sum_{j\in I} a_{ij}^{(n)} \frac{\partial f}{\partial y_n^{(j)}}  \right\rangle_{\! K}  \\
&=& 
\sum_{j\in I} a_{ij}^{(n)} \left\langle n^{-1} z_{\lambda}^{(\ul i)}, \frac{\partial f}{\partial y_n^{(j)}} \right\rangle_{\! K} 
&=& 
\sum_{j\in I} a_{ij}^{(n)} 
\left\langle y_n^{(i)} z_{\lambda}^{(\ul i)}, f \right\rangle_{\! K} \\
&=&
\left\langle \sum_{j\in I} \sigma(a_{ij}^{(n)}) y_n^{(i)} z_{\lambda}^{(\ul i)}, f \right\rangle_{\! K} 
&=& 
\left\langle z_n^{(i)} z_{\lambda}^{(\ul i)}, f \right\rangle_{\! K},
\end{array}
\]
and therefore~\eqref{compclaim} holds in all cases.

Let $Q = (q_{\lambda,\mu}^{(\ul i, \ul j)})_{(\lambda,\ul i), (\mu, \ul j)\in \Omega_{\ell,d}}\in \MAT_{\Omega_{\ell,d}} (\F)$ be the transition matrix defined by 
\begin{align*}
z_{\lambda}^{(\ul i)} = \sum_{(\mu, \ul j)\in \Omega_{\ell,d}} q_{\lambda,\mu}^{(\ul i,\ul j)} y_{\mu}^{(\ul j)}. 
\end{align*}
For any $(\lambda, \ul i) \in \Omega_{\ell,d}$, we have 
\begin{align*}
z_{\lambda}^{(\ul i)} &= \sum_{\ul j \in I^{\ell (\lambda)}} \sigma(a_{i_1,j_1}^{(\lambda_1)}) \cdots \sigma(a_{i_{\ell(\lambda)}, j_{\ell(\lambda)}}^{(\lambda_{\ell(\lambda)})}) 
y_{\lambda_1}^{(j_1)} \cdots y_{\lambda_{\ell (\lambda)}}^{(j_{\ell(\lambda)})}, 
\qquad \text{whence} \\ 
Q &= \bigoplus_{\lambda\in \PAR(d)} \bigotimes_{t\ge 1} \SYM^{m_t (\lambda)} (\sigma(A^{(t)}))
\end{align*}
(cf.~\cite[Lemma 5.3]{BK1}).
Writing $P^{-1} = (\tilde{p}_{\lambda,\mu}^{(\ul i, \ul j)})_{(\lambda,\ul i), (\mu, \ul j)\in \Omega_{\ell,d}}$, we have 
\begin{align*}
\left\langle x_{\lambda}^{(\vph{\ul j} \ul i)}, x_{\mu}^{(\ul j)} \right\rangle_{\! S} 
&=
\sum_{(\nu, \ul k)\in \Omega_{\ell,d}}  \sigma(p_{\lambda,\nu}^{(\ul i, \ul k)}) \left\langle  y_{\nu}^{(\vph{\ul j} \ul k)}, x_{\mu}^{(\ul j)} \right\rangle_{\! S} \\
&=
\sum_{(\nu, \ul k)\in \Omega_{\ell,d}}  \sigma(p_{\lambda,\nu}^{(\ul i, \ul k)}) \left\langle z_{\nu}^{(\vph{\ul j} \ul k)}, x_{\mu}^{(\ul j)} \right\rangle_{\! K} \\
&=
\sum_{(\nu, \ul k), (\eta, \ul r)\in \Omega_{\ell,d}} \sigma(p_{\lambda,\nu}^{(\ul i, \ul k)}) \sigma(q_{\nu, \eta}^{(\ul k, \ul r)}) 
\left\langle y_{\eta}^{(\vph{\ul j} \ul r)}, x_{\mu}^{(\ul j)} \right\rangle_{\! K}   \\
&=
\sum_{(\nu, \ul k), (\eta, \ul r), (\theta, \ul s)\in \Omega_{\ell,d}} 
\sigma(p_{\lambda,\nu}^{(\ul i, \ul k)}) \sigma(q_{\nu, \eta}^{(\ul k, \ul r)}) \sigma(\tilde{p}_{\eta, \theta}^{(\ul r, \ul s)})
\left\langle x_{\theta}^{(\vph{\ul j} \ul s)}, x_{\mu}^{(\ul j)} \right\rangle_{\! K} 
\end{align*}
for any 
$(\lambda,\ul i), (\mu, \ul j)\in \Omega_{\ell,d}$, where the second equality holds by~\eqref{compclaim}. 
Therefore,
\[
 M_S = \sigma(P) \sigma(Q) \sigma(P)^{-1} M_K = \sigma(P) \left( \bigoplus_{\lambda\in \PAR(d)} \bigotimes_{t\ge 1} \SYM^{m_t (\lambda)} (A^{(t)})
\right) \sigma(P)^{-1} M_K. \qedhere
\] 
\end{proof}

The following is a corollary of the boson-fermion correspondence over $\Z$ (see ~\cite[Corollary 2.1]{DcKK} 
and ~\cite[Proposition 2.4]{Tsu}).

\Prop\label{existsK}
Let $\corr$, $\ell$, $I$, $\sigma$, $V$ and $V_d$ be as in Proposition~\ref{comp}.  
\bna
\item\label{existsKa} There
exists a unique bi-additive non-degenerate map $\langle \cdot, \cdot \rangle_K \colon V\times V\to\corr$ such that
\bnum
\item $\langle af,g\rangle_K=\BARR(a)\langle f,g\rangle_K, \langle f,ag\rangle_K=a\langle f,g\rangle_K$ 
and $\langle f,g\rangle_K=\BARR(\langle g,f\rangle_K)$, 
\item $\langle 1,1\rangle_K=1$ and 
$\langle my_m^{(i)} f,g\rangle_K=\langle f,\frac{\partial g}{\partial y_m^{(i)}}\rangle_K$, 
\ee
for all $f,g\in V$, $a\in \corr$ and $i\in I$.
\item\label{existsKb}
Suppose further that
for each $1\le i\le \ell$
the variables $\{x_n^{(i)}\mid n\geq 1\}$ and $\{y_n^{(i)}\mid n \geq 1\}$
are related by the formal identity 
\begin{align}\label{eq:exp}
1+\sum_{n\ge 1} x_n^{(i)} t^n = \exp \left( \sum_{r\ge 1} y_r^{(i)} t^r \right).
\end{align}
Then, for any $d\geq 0$, the set of Schur functions 
\begin{align*}
\left\{\prod_{i\in I}s_{\lambda^{(i)}}(x^{(i)})\quad\!\!\!\!\big|\quad\!\!\!\!\sum_{i\in I}|\lambda^{(i)}|=d\right\}
\end{align*}
forms an orthonormal basis of the $\Z$-lattice 
$\Z[x_n^{(i)}\mid i\in I,n\geq 1]\cap V_d$ of $V_d$ with respect to 
$\langle \cdot, \cdot \rangle_K$.
Here, $s_{\lambda}(x^{(i)}):=\det(x^{(i)}_{\lambda_k+j-k})_{1\leq j,k\leq |\lambda|}$ for $\lambda\in\PAR$ and
$x^{(i)}_{m}=\delta_{m,0}$ for $m\leq 0$.
\ee
\enprop

Note that the form $\langle \cdot, \cdot\rangle_K \colon V\times V \to \corr$ 
satisfying the conditions of Proposition~\ref{comp} is clearly unique. Also, those conditions are implied by the properties satisfied by the form 
$\langle \cdot, \cdot\rangle_K$ of Proposition~\ref{existsK} (\ref{existsKa}).



\Cor\label{lem:exp}
Assume all the hypotheses of Proposition~\ref{comp}. 
Suppose further that 
the variables $\{x_n^{(i)}\mid i\in I,n\geq 1\}$ and $\{y_n^{(i)}\mid i\in I,n\geq 1\}$ are related as in Proposition \ref{existsK} \eqref{existsKb}.
Then 
\begin{align}
 M_S = K_{\ell,d}^{-1} \left({{\bigoplus_{\lambda\in\PAR(d)}\bigotimes_{t\geq 1}\SYM^{m_t(\lambda)}(A^{(t)})}}\right)K_{\ell,d}M_K.
\label{resms}
\end{align}
and $M_K \in \GL_{\PAR_{\ell} (d)} (\Z)$.
\encor

\begin{proof}
Let $Y=\bigoplus_{\lambda\in\PAR(d)}\bigotimes_{t\geq 1}\SYM^{m_t(\lambda)}(A^{(t)})$.
We identify the ring $V$ with $\corr\otimes\Lambda_{\ell}$ by setting $y_n^{(i)} = p_n^{(i)}/n$. 
Then, comparing the hypothesis with~\eqref{hphenkan}, we see that $x_n^{(i)} = h_n^{(i)}$. 
Define $w_{\mu} = \mu_1 \cdots \mu_{\ell(\mu)}$ for $\mu \in \PAR$ and $i\in I$, and let 
$W = \DIAG \{ w_{\mu^{(1)}} \cdots w_{\mu^{(\ell)}} \mid  (\mu^{(1)},\ldots,\mu^{(\ell)} ) \in\PAR_{\ell} (d) \}$.
It follows from Definition~\ref{MK} that the change-of-basis matrix $P$ of Proposition~\ref{comp} is given by 
$P = K_{\ell,d}^{-1} W$.
Hence, Proposition~\ref{comp} yields
\begin{align*}
M_S = \sigma(K_{\ell,d})^{-1}\sigma (W)Y\sigma(W)^{-1}\sigma(K_{\ell,d})M_K.
\end{align*}
Observing that, when we view $W$ as an $\Omega_{\ell,d}\times\Omega_{\ell,d}$-matrix, each block of $W$ corresponding to a fixed $\lambda\in \PAR(d)$ is a scalar matrix and also that $\sigma (K_{\ell,d}) = K_{\ell,d}$ because $K_{\ell,d}$ is $\Q$-valued, we obtain \eqref{resms}.

Thanks to Proposition ~\ref{existsK}, 
there exists $Q\in \GL_{\PAR_{\ell} (d)} (\Z)$ such that $M_K=\TRANS{Q}\cdot Q$.
\end{proof}

The following result is a quantized version of~\cite[Proposition 3.3]{Hil}, though our proof is different.

\Th\label{inte}
For $\ell\geq 1$ and $A\in\MAT_{\ell}(\MA)$, we have $\HATTT{M}{\ell,d}{-1}\HAT{A}{d}\HATT{M}{\ell,d}\in\MAT_{\PAR_{\ell}(d)}(\MA)$ for any $d\geq 0$.
\enth

\Proof
Let $I=\{1,\ldots,\ell\}$.
By Remark \ref{rem:MK},
it will suffice to prove that $\HATTT{K}{\ell,d}{-1}\HAT{P}{d}\HATT{K}{\ell,d}\in\MAT_{\PAR_{\ell}(d)}(\MA)$.

In the rest of the proof, we identify $\cor\otimes\HATT{\Lambda}{\ell}$ with $V=\cor[y_n^{(i)}\mid i\in I,n\geq 1]$
by identifying $p^{(i)}_n/n$ with $y_n^{(i)}$. Write $A=(a_{ij})_{i,j\in I}$.
Define new variables $x_n^{(i)} \in V$ by the identity~\eqref{eq:exp}. 
Clearly, there exists a unique bi-additive map 
$\langle \cdot, \cdot \rangle_S\colon V\times V\to \cor$ such that
\bna
\item $\langle cf,g\rangle_S=\BAR(c)\langle f,g\rangle_S$, $\langle f,cg\rangle_S=c\langle f,g\rangle_S$,
\item \label{innerprod1} $\langle 1,1\rangle_S=1$, and $\langle 1, f\rangle_S = 0$ if $f$ has zero constant term as a polynomial in the variables $y_n^{(j)}$,
\item\label{innerprod2}
$\langle my_m^{(i)}f,g\rangle_S=\langle f,\sum_{j\in I}\INFL_{m}(a_{ij})\frac{\partial g}{\partial y^{(j)}_m}\rangle_S$
\ee
for $f,g\in V$, $c\in\cor$, $m\geq 1$, $i\in I$. 
Applying Corollary~\ref{lem:exp} with $\F = \cor$, $\sigma = \BAR{}$ and the form $\langle \cdot, \cdot \rangle_K$ supplied by 
Proposition \ref{existsK} (\ref{existsKa}), 
we obtain  
$M_S=\HATTT{K}{\ell,d}{-1}\HAT{A}{d}\HATT{K}{\ell,d} M_K$ (in the notation of Proposition~\ref{comp}) and $M_K\in \GL_{\PAR_{\ell} (d)} (\Z)$.

Thus, it is enough to show that $\langle x^{(\vph{\ul j} \ul  i)}_{\lambda},x^{(\ul  j)}_{\mu}\rangle_S\in\MA$ for $(\lambda,\ul  i),(\mu,\ul  j)\in\Omega_{\ell,d}$,
where $x_{\lambda}^{(\ul  i)}$ is defined as in Proposition~\ref{comp}. We argue by induction on 
$|\lambda|$.
Expanding~\eqref{eq:exp}, we obtain
\begin{equation}\label{eq:xexpansion}
{\textstyle{
x^{(i)}_n=\sum_{\lambda\in\PAR(n)}\prod_{k\geq 1}\frac{(y^{(i)}_k)^{m_k(\lambda)}}{m_k(\lambda)!}}},
\end{equation}
and therefore $\partial x^{(i)}_{m}/\partial y^{(j)}_{n}=\delta_{ij}x^{(i)}_{m-n}$
for $i,j\in I$ and $m,n\geq 1$,
where we put $x^{(i)}_{\gamma}=\delta_{\gamma,0}$ 
for $\gamma\leq 0$ (see also ~\cite[page 129]{DcKK}).
Combining~\eqref{eq:xexpansion} with the defining property~\eqref{innerprod2} of $\langle \cdot, \cdot \rangle_S$, we obtain the identity 
$\langle x_{n}^{(i)} f, g \rangle_S = \langle f, D^{(i)}_n g\rangle_S$ for all $f,g\in V$, $n\ge 1$, $i\in I$, 
where the 
differential operator $D^{(i)}_n \colon V \to V$ is defined by 
\[
 D^{(i)}_n = 
 \sum_{\lambda\in\PAR(n)}\prod_{k\geq 1} 
 \frac{1}{k^{m_k(\lambda)}m_k(\lambda)!}\left( \sum_{j\in I} \INFL_k (a_{ij}) \frac{\partial}{\partial y^{(j)}_k}\right)^{m_k(\lambda)}
\]
Let $V^{\MA} = \MA[x^{(\ul i)}_{\lambda} \mid (\lambda, \ul i) \in \Omega_{\ell,d}]$.
By the inductive hypothesis, it is enough to show that $D^{(i)}_n (V^{\MA}) \subset V^{\MA}$ for all $i\in I$, $n\ge 1$. 
By a straightforward calculation, one obtains the product rule $D^{(i)}_n (fg) = \sum_{s=0}^n D_s^{(i)}(f) D_{n-s}^{(i)}(g)$ for $f,g\in V$. 
Hence, it suffices to prove that
$D^{(i)}_n (x_m^{(j)}) \in V^{\MA}$ for all $i,j\in I$ and $n,m\ge 1$. 
We have  
\[
D^{(i)}_n (x_m^{(j)}) = \left(
\sum_{\lambda\in\PAR(n)}\prod_{k\geq 1}\frac{\INFL_k(a_{ij})^{m_k(\lambda)}}{k^{m_k(\lambda)}m_k(\lambda)!}\right)x_{m-n}^{(j)},
\]
and the result now follows from Lemma \ref{intinfl}.
\QED

\Lemma\label{intinfl}
For any $f\in\MA$, we have $\sum_{\lambda\in\PAR(n)}\frac{1}{z_{\lambda}}\prod_{k\geq1}\INFL_k(f)^{m_k(\lambda)}\in\MA$.
\enlemma

\Proof
For $\theta=(\theta_k)_{k\geq 1}\in\MA^{\Z_{\geq 1}}$ and $n\geq 0$, we define 
$b_\theta(n)=\sum_{\lambda\in\PAR(n)}\frac{1}{z_{\lambda}}\prod_{k\geq 1}\theta_k^{m_k(\lambda)}$ (cf.~Definition \ref{apn}).
Similarly to Lemma \ref{alemma} (\ref{alemmaa}), we have $b_{\theta+\theta'}(n)=\sum_{k=0}^{n}b_{\theta}(k)b_{\theta'}(n-k)$.
Thus, it is enough to show that
$b_{\theta^{\pm}_{m}}(n)\in\MA$
for $m\in\Z$ where $\theta^{\pm}_{m}=(\pm v^m,\pm v^{2m},\pm v^{3m},\ldots)$.
By the orthogonality relations,
we have $\sum_{\lambda\in\PAR(n)}\frac{(\pm 1)^{\ell(\lambda)}}{z_{\lambda}}=(1\pm1)/2$, which implies that 
$b_{\theta^{\pm}_{m}}(n)=(v^{mn}\pm v^{mn})/2$.
\QED

\section{Graded Cartan matrices of symmetric groups and Hecke algebras}\label{sec:3}

In this section we recall the definition of graded Cartan matrices
$C_{\MH_n(\corrr{\ell};\qq{\ell})}^v$ and reduce the problem of finding their
unimodular equivalence classes
to the same problem for the matrix $M_n\DIAG(\{J^{v}_{\ell}(\lambda)\mid\lambda\in\PAR(n)\})M_n^{-1}$ (cf.~Conjecture~\ref{ourgradedconjecture}).

\subsection{Gram matrices of quantized Shapovalov forms}\label{pretsu2}
We now recall some of the definitions and results from~\cite{Tsu} and, in particular, 
define the Gram matrix $\QSHM_{\lambda,\mu}(X)$ of a quantized Shapovalov form (cf.~\cite[Definition 3.13]{Tsu}). 
For the theory of quantum groups, the book ~\cite{Lus} is a standard reference.

Let $X=(a_{ij})_{i,j\in I}$ be a symmetrizable {\GCM} and take the symmetrization $d=(d_i)_{i\in I}$ of $X$, i.e.,
the unique $d\in\NNN^I$ such that $d_ia_{ij}=d_ja_{ji}$ for all $i,j\in I$ and
$\gcd(d_i)_{i\in I}=1$.
We consider a root datum $(\MP,\MPC,\Pi,\Pi^\vee)$ in the following sense: 
\bna
\item $\MPC$ is a free $\Z$-module of rank $(2|I|-\Rank X)$ and $\MP=\Hom_\Z(\MPC,\Z)$,
\item $\Pi^\vee=\{h_i\mid i\in I\}$ is a $\Z$-linearly independent subset of $\MPC$,
\item $\Pi=\{\alpha_i\mid i\in I\}$ is a $\Z$-linearly independent subset of $\MP$,
\item $\alpha_j(h_i)=a_{ij}$ for all $i,j\in I$.
\ee

We denote by $Q^+=\bigoplus_{i\in I}\Z_{\geq 0}\alpha_i$ the positive part of the root lattice and
denote by $\MP^+$
the set of dominant integral weights $\{\lambda\in\MP\mid \forall i\in I,\lambda(h_i)\in\Z_{\geq 0}\}$.
For each $i\in I$, $\Lambda_i\in\MP^+$ is a dominant integral weight 
determined modulo the subgroup $\{\lambda\in\MP\mid \forall i\in I,\lambda(h_i)=0\}$ of $\MP$ 
by the condition that $\Lambda_i(h_j)=\delta_{ij}$ for all $j\in I$.

Recall that
the Weyl group $W=W(X)$ is the subgroup of $\AUT(\MP)$ generated by $\{s_i:\MP\isoto\MP,\lambda\MAPSTO\lambda-\lambda(h_i)\alpha_i\mid i\in I\}$.

\Def 
The quantum group $U_v=U_v(X)$ is the unital associative $\cor$-algebra generated by 
$\{e_i,f_i\mid i\in I\}\cup\{v^h\mid h\in\MPC\}$ with the following defining relations:
\bna
\item $v^0=1$ and $v^hv^{h'}=v^{h+h'}$ for any $h,h'\in\MPC$,
\item $v^{h}e_iv^{-h}=v^{\alpha_i(h)}e_i,v^{h}f_iv^{-h}=v^{-\alpha_i(h)}f_i$ for any $i\in I$ and $h\in\MPC$,
\item $e_if_j-f_je_i=\delta_{ij}(K_i-K_i^{-1})/(v_i-v_i^{-1})$ for any $i,j\in I$,
\item\label{qSerre1} $\sum_{k=0}^{1-a_{ij}}(-1)^k e_i^{(k)}e_je_i^{(1-a_{ij}-k)}=0=\sum_{k=0}^{1-a_{ij}}(-1)^k f_i^{(k)}f_jf_i^{(1-a_{ij}-k)}$
for any $i\ne j\in I$,
\ee
where $K_i=v^{d_ih_i}, v_i=v^{d_i}$ and $e_i^{(n)}=e_i^n/[n]_{d_i}!,f_i^{(n)}=f_i^n/[n]_{d_i}!$.
\edf

Let $U_v^{+},U_v^{0},U_v^{-}$ be the $\cor$-subalgebras of $U_v$ defined by
\begin{align*}
U_v^+=\langle e_i\mid i\in I\rangle,\quad
U_v^-=\langle f_i\mid i\in I\rangle,\quad
U_v^0=\langle v^h\mid h\in\MPC\rangle.
\end{align*}
Then, the following is a triangular decomposition theorem for quantum groups~\cite[\S3.2]{Lus}:
\bnum
\item\label{QPBWa} the canonical map
$U^{-}_v\otimes_{\cor}U^{0}_v\otimes_{\cor}U^{+}_v\to U_v$
is a $\cor$-vector space isomorphism,
\item\label{QPBWb}
$U^{0}_v$ is canonically isomorphic to the group $\cor$-algebra $\cor[\MPC]$.
\ee

For each $\lambda\in \MP^+$, we denote by $V(\lambda)$
the integrable highest weight $U_v$-module with highest weight $\lambda$ and a fixed highest weight vector $1_{\lambda}\in V(\lambda)$.



\Prop[{\cite[Proposition 3.8]{Tsu}}]\label{tsu38}
For $\lambda\in\MP^+$,
there exist unique bi-additive non-degenerate maps $\langle\cdot,\cdot\rangle_{\QSH}:V(\lambda)\times V(\lambda)\to \cor$ and
$\langle\cdot,\cdot\rangle_{\RSH}:V(\lambda)\times V(\lambda)\to \cor$ with 
\bnum
\item\label{tsu38i} $\langle aw_1,w_2\rangle_{Y}=\BAR(a)\langle w_1,w_2\rangle_{Y}$,
$\langle w_1,aw_2\rangle_{Y}=a\langle w_1,w_2\rangle_{Y}$ and 
$\langle w_1,w_2\rangle_{Y}=\BAR(\langle w_2,w_1\rangle_{Y})$,
\item $\langle 1_{\lambda},1_{\lambda}\rangle_{Y}=1$ and 
$\langle uw_1,w_2\rangle_{\QSH}=\langle w_1,\Omega(u)w_2\rangle_{\QSH}$, $\langle uw_1,w_2\rangle_{\QSH}=\langle w_1,\Upsilon(u)w_2\rangle_{\RSH}$.
\ee
for all $Y\in\{\QSH,\RSH\}$ and
for all $w_1,w_2\in V(\lambda)$, $u\in U_v$ and $a\in\cor$.
Here, $\Omega$ and $\Upsilon$ are the $\Q$-antiinvolution and $\Q$-antiautomorphism of $U_v$ defined by
\begin{align*}
\Omega(e_i) &= f_i,\quad
\Omega(f_i)=e_i,\quad
\Omega(v^h)=v^{-h},\quad
\Omega(v)=v^{-1},\\
\Upsilon(e_i) &= v_if_iK_i^{-1},\quad
\Upsilon(f_i)=v_i^{-1}K_ie_i,\quad
\Upsilon(v^h)=v^{-h},\quad
\Upsilon(v)=v^{-1}.
\end{align*}

\enprop

We denote by $P(\lambda):=\{\mu\in\MP\mid V(\lambda)_{\mu}\ne 0\}$ 
the set of weights of $V(\lambda)$, which
is $W$-invariant~\cite[Proposition 5.2.7]{Lus}.
Let $(U_v^{-})^{\MA}$ be the $\MA$-subalgebra of $U^-_v$ generated by
$\{f_i^{(n)}\mid i\in I,n\geq 0\}$. 
The constructions below use the following deep results:
\bna
\item $(U_v^{-})^{\MA}$ is an $\MA$-lattice of $U_v^{-}$ (see ~\cite[Theorem 14.4.3]{Lus}),
\item $V(\lambda)_{\nu}^{\MA}:=V(\lambda)_{\nu}\cap V(\lambda)^{\MA}$ is an $\MA$-lattice of $V(\lambda)_{\nu}$
for $\nu\in P(\lambda)$ where $V(\lambda)^{\MA}:=(U^{-}_v)^{\MA}1_{\lambda}\subseteq V(\lambda)$ (see ~\cite[Theorem 14.4.11]{Lus}).
\ee

\Def[{\cite[Proposition 3.13]{Tsu}}]\label{tsu313}
For $\lambda\in \MP^+$ and $\mu\in P(\lambda)$, we define 
\begin{align*}
\QSHM_{\lambda,\mu}(X)=(\langle w_i,w_j\rangle_{\QSH})_{1\leq i,j\leq \dim V(\lambda)_{\mu}},\quad
\RSHM_{\lambda,\mu}(X)=(\langle w_i,w_j\rangle_{\RSH})_{1\leq i,j\leq \dim V(\lambda)_{\mu}}
\end{align*}
where $\{w_i\mid 1\leq i\leq \dim V(\lambda)_{\mu}\}$ is an $\MA$-basis of $V(\lambda)_{\mu}^{\MA}$.
\edf

For any $n\ge 0$, define the equivalence relation $\GCONG$ on $\MAT_n (\MA)$ as follows: 
\[
Y\GCONG Z\defequiv \exists P\in\GL_{n}(\MA), \; \BAR(\TRANS{P})YP=Z.
\] 
For $Z\in\{\QSHM_{\lambda,\mu}(X),\RSHM_{\lambda,\mu}(X)\}$, 
the equivalence class of $Z$ under $\GCONG$ does not depend on the choice of the basis in Definition~\ref{tsu313}. 
Thus, the $\MA$-unimodular equivalence classes of $Z$ are uniquely determined.
Note that by construction $\TRANS{Z}=\BAR(Z)$.
The following is implicit in ~\cite[Proposition 3.16]{Tsu}. 

\Prop\label{Aaux}
For $\lambda\in \MP^+$ and $\mu\in P(\lambda)$, there exists an
$\MA$-basis of $V(\lambda)_{\mu}^{\MA}$ whose associated
$\QSHM_{\lambda,\mu}(X)$ is an $\MA^{\BAR}$-valued symmetric matrix.
\enprop

\Proof
Take an $\MA$-basis $(v_b)_b$ of $V(\lambda)_{\mu}^{\MA}$ of the form $v_b=G_b1_{\lambda}$
with $G_b\in (U^{-}_v)^{\MA}$ and $G_b=\overline{G_b}$, where the bar involution $\overline{\mathstrut\quad\!\!}:U_v\to U_v$ is
defined by 
\begin{align*}
\overline{e_i}=e_i,\quad
\overline{f_i}=f_i,\quad
\overline{v^h}=v^{-h},\quad
\overline{v}=v^{-1}.
\end{align*}
This is possible using the lower canonical basis of $U^-_v$ (see the last paragraph of ~\cite{Ka2}) or using ~\cite[Theorem 6.5]{Lak}.

Let $\HC\colon U_v\to U_v^0$ and $\EV_{\lambda}\colon U_v^0\to \cor$ be the following maps: 
\bnum
\item the Harish-Chandra projection $\HC\colon U_v\twoheadrightarrow U_v^0$, which is the $\cor$-linear projection
from $U_v=U_v^0\oplus((\sum_{i\in I}f_iU_v)+(\sum_{i\in I}U_ve_i))$ onto $U_v^0$,
\item the evaluation map $\EV_{\lambda}\colon U_v^0\to \cor$, which is 
the $\cor$-algebra homomorphism determined by the assignment $\EV_{\lambda}(v^h)=v^{\lambda(h)}$ for each $h\in\MPC$.
\ee
These maps exist by parts (\ref{QPBWa}), (\ref{QPBWb}) in the triangular decomposition theorem respectively.

By the construction of $\langle\cdot,\cdot\rangle_{\QSH}$ (see the proof of ~\cite[Proposition 3.8]{Tsu}), we have
\begin{align}
\langle
v_b,
v_b'\rangle_{\QSH}
=
\EV_{\lambda}(\HC(\Omega(G_b)G_{b'})).
\label{innervalue}
\end{align}

Since $\HC(\Omega(G_b)G_{b'})\in U_v^0\cap U_v^{\MA}$, where 
$U_v^{\MA}$ is an $\MA$-subalgebra of $U_v$ generated by $\{v^h,e_i^{(n)},f_i^{(n)}\mid i\in I,n\geq 0,h\in\MPC\}$, and
it is known (see ~\cite[Theorem 4.5]{Lus2} or ~\cite[Theorem 6.49]{DDPW}) 
that $U_v^0\cap U_v^{\MA}$ is the $\MA$-subalgebra of $U^0_v$ generated by 
\begin{align*}
\left\{
v^h,\left[\substack{K_i;0 \\ n}\right]:=
{\textstyle\prod_{j=1}^{n}\frac{K_iv_i^{-j+1}-K_i^{-1}v_i^{j-1}}{v_i^j-v_i^{-j}}}\quad\!\!\!\!\big|\quad\!\!\!\! i\in I,n\geq 1,h\in\MPC
\right\},
\end{align*}
\eqref{innervalue} is $\MA$-valued. Since $\Omega(G_b)G_{b'}$ is bar-invariant,
\eqref{innervalue} is $\MA^{\BAR}$-valued due to the isomorphism $U^{0}_v\cong\cor[\MPC]$. (For an estimate of \eqref{innervalue} 
when $G_b$ is the lower canonical basis, see ~\cite[Problem 2]{Ka1}.)
\QED

\Cor\label{Aaux2}
For $\lambda\in \MP^+$ and $\mu\in P(\lambda)$, we have 
$\QSHM_{\lambda,\mu}(X)\CONG{\MA}\TRANS{\QSHM_{\lambda,\mu}(X)}$.
\encor

The proof of Proposition \ref{Aaux} also shows that 
$\RSHM_{\lambda,\mu}(X)$ is $\MA$-valued, 
which is again implicit in ~\cite[Proposition 3.16]{Tsu}.

\Prop[{\cite[Proposition 3.16]{Tsu}}]\label{tsu316}
For $\lambda\in \MP^+$ and $\mu\in P(\lambda)$, there exists an
$\MA$-basis of $V(\lambda)_{\mu}^{\MA}$ whose associated
$\QSHM_{\lambda,\mu}(X)$ and $\RSHM_{\lambda,\mu}(X)$ satisfy $D\QSH^M_{\lambda,\mu}(X)=\RSH^M_{\lambda,\mu}(X)$
for a diagonal matrix $D$ all of whose diagonal entries belong to $v^{\Z}$.
\enprop

\subsection{Specialization to the basic representations}\label{SSADE}

Let $X=(a_{ij})_{i,j\in I}$ be a Cartan matrix of type A,D,E 
and let $\HA=X^{(1)}$ be the extended (generalized) Cartan matrix of $X$ indexed by $\HI=\{0\}\sqcup I$ as in Figure \ref{untwisted}.
Let $(a_i)_{i\in\HI}$ be the numerical labels of $\HA$ in Figure ~\ref{untwisted} and let $\delta=\sum_{i\in \HI}a_i\alpha_i$.
We set $U_v = U_v (\HA)$ and apply the notation of \S\ref{pretsu2} to this algebra. 
By~\cite[Lemma 12.6]{Kac}, we have 
$P(\Lambda_0)=\{w\Lambda_0-d\delta\mid w\in W,d\geq 0\}$.

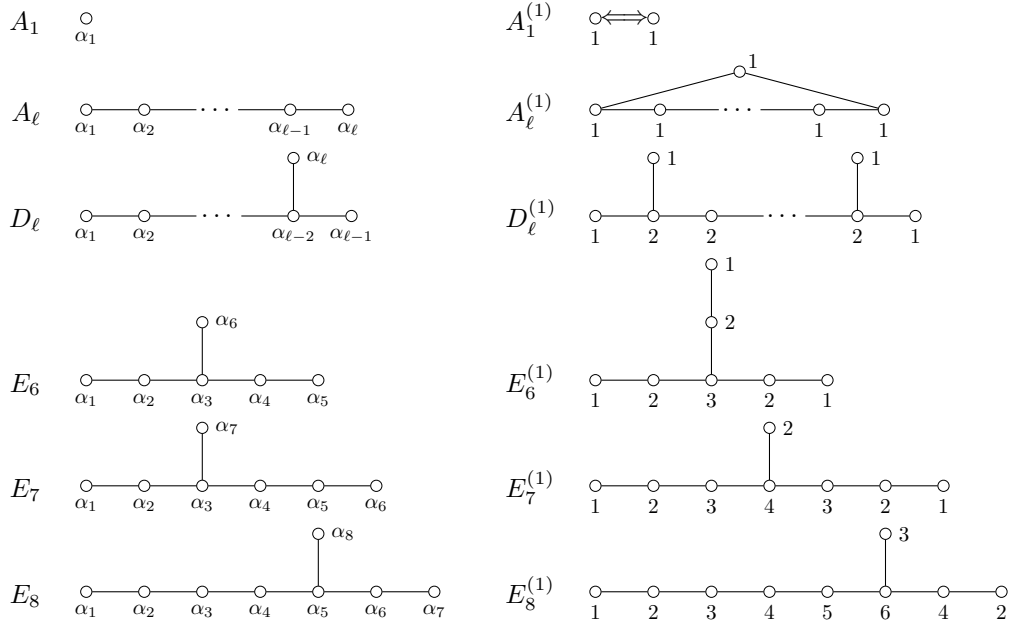
\begin{figure}
\[
\begin{array}{r@{\quad}l@{\qquad}l@{\quad}l}
A_1 & \begin{tikzpicture}[start chain]
\dnode{$\alpha_1$}
\end{tikzpicture} &
A_1^{(1)} & \begin{tikzpicture}[start chain]
\dnodenj{1}
\dnodenj{1}
\path (chain-1) -- node[anchor=mid] {\(\Longleftrightarrow\)} (chain-2);
\end{tikzpicture} \\
%
A_\ell & \begin{tikzpicture}[start chain]
\dnode{$\alpha_1$}
\dnode{$\alpha_2$}
\dydots
\dnode{$\alpha_{\ell-1}$}
\dnode{$\alpha_\ell$}
         \end{tikzpicture}
&
A_\ell^{(1)}  & 
\begin{tikzpicture}[start chain,node distance=1ex and 2em]
\dnode{1}
\dnode{1}
\dydots
\dnode{1}
\dnode{1}
\begin{scope}[start chain=br going above]
\chainin(chain-3);
\node[ch,join=with chain-1,join=with chain-5,label={[inner sep=1pt]10, scale=0.8:\(1\)}] {};
\end{scope}
\end{tikzpicture}\\
D_\ell & 
\begin{tikzpicture}
\begin{scope}[start chain]
\dnode{$\alpha_1$}
\dnode{$\alpha_2$}
\node[chj,draw=none] {\dots};
\dnode{$\alpha_{\ell-2}$}
\dnode{$\alpha_{\ell-1}$}
\end{scope}
\begin{scope}[start chain=br going above]
\chainin(chain-4);
\dnodebr{$\alpha_\ell$}
\end{scope}
\end{tikzpicture} &
D_\ell^{(1)}  & \begin{tikzpicture}
\begin{scope}[start chain]
\dnode{1}
\dnode{2}
\dnode{2}
\dydots
\dnode{2}
\dnode{1}
\end{scope}
\begin{scope}[start chain=br going above]
\chainin(chain-2);
\dnodebr{1};
\end{scope}
\begin{scope}[start chain=br2 going above]
\chainin(chain-5);
\dnodebr{1};
\end{scope}
\end{tikzpicture} \\
E_6& \begin{tikzpicture}
\begin{scope}[start chain]
\dnode{$\alpha_1$}
\dnode{$\alpha_2$}
\dnode{$\alpha_3$}
\dnode{$\alpha_4$}
\dnode{$\alpha_5$}
\end{scope}
\begin{scope}[start chain=br going above]
\chainin (chain-3);
\dnodebr{$\alpha_6$}
\end{scope}
\end{tikzpicture} &
E_6^{(1)}  & 
\begin{tikzpicture}
\begin{scope}[start chain]
\foreach \dyi in {1,2,3,2,1} {
\dnode{\dyi}
}
\end{scope}
\begin{scope}[start chain=br going above]
\chainin(chain-3);
\dnodebr{2}
\dnodebr{1}
\end{scope}
\end{tikzpicture}
 \\
E_7& 
\begin{tikzpicture}
\begin{scope}[start chain]
\dnode{$\alpha_1$}
\dnode{$\alpha_2$}
\dnode{$\alpha_3$}
\dnode{$\alpha_4$}
\dnode{$\alpha_5$}
\dnode{$\alpha_6$}
\end{scope}
\begin{scope}[start chain=br going above]
\chainin (chain-3);
\dnodebr{$\alpha_7$}
\end{scope}
\end{tikzpicture} &
E_7^{(1)} & 
\begin{tikzpicture}
\begin{scope}[start chain]
\foreach \dyi in {1,2,3,4,3,2,1} {
\dnode{\dyi}
}
\end{scope}
\begin{scope}[start chain=br going above]
\chainin(chain-4);
\dnodebr{2}
\end{scope}
\end{tikzpicture}\\
E_8& 
\begin{tikzpicture}
\begin{scope}[start chain]
\dnode{$\alpha_1$}
\dnode{$\alpha_2$}
\dnode{$\alpha_3$}
\dnode{$\alpha_4$}
\dnode{$\alpha_5$}
\dnode{$\alpha_6$}
\dnode{$\alpha_7$}
\end{scope}
\begin{scope}[start chain=br going above]
\chainin (chain-5);
\dnodebr{$\alpha_8$}
\end{scope}
\end{tikzpicture}
 &
E_8^{(1)} & \begin{tikzpicture}
\begin{scope}[start chain]
\foreach \dyi in {1,2,3,4,5,6,4,2} {
\dnode{\dyi}
}
\end{scope}
\begin{scope}[start chain=br going above]
\chainin(chain-6);
\dnodebr{3}
\end{scope}
\end{tikzpicture} \\
\end{array}
\]
\caption{Finite and untwisted affine Dynkin diagrams of types A,D,E.}
\label{untwisted}
\end{figure}

\Def\label{tsu55}
For $d\ge 0$ and $w\in W$, we define $C_{d}^v(X)$ to be $\QSHM_{\Lambda_0,w\Lambda_0-d\delta}(\HA)$.
For $\ell\geq 2$, we put $C_{\ell,d}^v=C_{d}^v(A_{\ell-1})$.
\edf

The equivalence class of $C_{d}^v(X)$ under $\GCONG$ does not depend on the choice of $w\in W$~\cite[Proposition 3.18]{Tsu}.
The following is implicit in the proof of ~\cite[Theorem 4.4]{Tsu}. 
For convenience, we give a proof.

\begin{theorem}\label{thm:Cequiv}
Let $X=(a_{ij})_{i,j\in I}$ be a Cartan matrix of type A, D or E, where  $I=\{1,\ldots,\ell\}$.
For any $d\ge 0$, we have 
$C^v_{d}(X)\CONG{\MA}\HATTT{M}{\ell,d}{-1}\HAT{[X]}{d}\HATT{M}{\ell,d}$
where $[X]= ([a_{ij}])\in\MAT_{I}(\MA)$.
\end{theorem}

\begin{proof}
Let $I= \{1,\ldots,\ell\}$. 
As in the proof of~\cite[Theorem 4.4]{Tsu},
$V(\Lambda_0)_{\Lambda_0-d\delta}^{\MA}$ can be regarded as an $\MA$-lattice of the polynomial ring $\cor[h_{i,-r}\mid i\in I,r\geq 1]$.
More precisely, defining new variables $y_r^{(i)}$ and $x_r^{(i)}$ (for $i\in I$, $r\ge 1$) by 
$y_r^{(i)} = h_{i,-r}/[r]$ and ~\eqref{eq:exp}, we have
\bnum
\item $V(\Lambda_0)_{\Lambda_0-d\delta}$ has a $\cor$-basis 
$\{ y_{\lambda}^{(\ul  i)} \mid (\lambda, \ul  i) \in \Omega_{\ell,d} \}$, 
\item $V(\Lambda_0)^{\MA}_{\Lambda_0 - d\delta}$ has an $\MA$-basis 
$\{ x_{\lambda}^{(\ul  i)} \mid (\lambda, \ul  i) \in \Omega_{\ell,d} \}$,
\ee
where $x_{\lambda}^{(\ul  i)}$ and $y_{\lambda}^{(\ul  i)}$ are defined as in Proposition~\ref{comp}. 
Moreover, by an identity in the proof of~\cite[Theorem 4.4]{Tsu}\footnote{Our $x_r^{(i)}$ and $y_r^{(i)}$ correspond respectively 
to $\tilde{P}^-_{i,r}$ and $h'_{i,-r}$ in \emph{loc. cit.}}
(together with the definition of $\langle \cdot, \cdot\rangle_{\QSH}$), 
we have
\begin{align*}
\left\langle sy_s^{(i)} H ,y_{r_1}^{(i_1)}\cdots y_{r_m}^{(i_m)}\right\rangle_{\!\QSH} 
=
\left\langle
H,{\textstyle{\sum_{k=1}^{m}}} \delta_{s,r_k}[a_{i,i_k}]_s y^{(i_1)}_{r_1}\cdots y^{(i_{k-1})}_{r_{k-1}} y^{(i_{k+1})}_{r_{k+1}}\cdots y^{(i_m)}_{r_m}
\right\rangle_{\!\!\QSH}
\end{align*}
for $H\in\cor[h_{i,-r}\mid i\in I,r\geq 1]$ and $i,i_k\in I, s,r_k\geq 1$. We can rewrite this identity as 
\[
 \left\langle s y_s^{(i)} H, H' \right\rangle_{\!\QSH} =
\Big\langle H, {\textstyle{\sum_{j=1}^{\ell}}} [a_{i,j}]_s \frac{\partial}{\partial y_{s}^{(j)}} H' \Big\rangle_{\!\!\QSH}.
\]
Therefore, 
by Corollary ~\ref{lem:exp}, we have
$
(\langle x^{(\ul  i)}_{\lambda},x^{(\ul  j)}_{\mu}\rangle_{\QSH})_{(\lambda,\ul  i),(\mu,\ul  j)\in\Omega_{\ell,d}}
=\HATTT{K}{\ell,d}{-1}\HAT{[X]}{d}\HATT{K}{\ell,d} M_K
$
where $M_K \in \GL_{\PAR_{\ell-1}(d)} (\Z)$.
By Remark \ref{rem:MK}, we are done.
\end{proof}

\Lemma\label{exQAT}
For $\ell\geq 1$, there exist $Q_{\ell},T_{\ell}\in \GL_{\ell}(\MA)$ such that $Q_{\ell}[A_{\ell}]T_{\ell}=[A'_{\ell}]$ where
$A'_{\ell}=\DIAG(\{1,\ldots,1,\ell+1\})\in\MAT_{\ell}(\Z)$.
\enlemma

\Proof
Define $Q_{\ell}\in \MAT_{\ell}(\MA)$ by 
\[
(Q_{\ell})_{ij}=
\begin{cases}
v^{j}[i] & \text{if } j=i \text{ or } j=i+1,\\
v^{i}[j] & \text{if } j<i,\\
0 & \text{otherwise}.
\end{cases} 
\]
A straightforward calculation shows that the matrix $Q_{\ell} [A_{\ell}]$ is upper-triangular with  diagonal entries 
$1,\ldots,1, v^{\ell} [\ell+1]$ and hence is column equivalent to $[A'_{\ell}]$ over $\MA$. 
Also, $\det( Q_{\ell} [A_{\ell}] ) = v^{\ell} [\ell+1]$.
We have $\det ([A_{\ell}]) = [\ell+1]$ by an easy inductive argument (cf.~\cite{Tsu}, proof of Corollary 4.5).  
Hence, $\det(Q_{\ell}) = v^{\ell}$, so $Q_{\ell} \in \GL_{\ell} (\MA)$.
\QED

\Th\label{gradedcartan}
For $\ell\geq 2$ and $d\geq 0$, 
we have
\begin{align}\label{eq:gc}
C^v_{\ell,d}\CONG{\MA}
\bigoplus_{s=0}^{d}
\left(M_s \DIAG(\{{\textstyle\prod_{i\geq 1}[\ell]^{m_i(\lambda)}_{i}}\mid\lambda\in\PAR(s)\})M_{s}^{-1} \right)^{\oplus |\PAR_{\ell-2}(d-s)|}.
\end{align}
\enth

\Proof
By Theorem~\ref{thm:Cequiv}, we have $C^v_{\ell,d}\CONG{\MA}\HATTT{M}{\ell-1,d}{-1}\HAT{[A_{\ell-1}]}{d}\HATT{M}{\ell-1,d}$.
Let $Q_{\ell-1}$ and $T_{\ell-1}$ be the matrices supplied by Lemma~\ref{exQAT}. 
By the functoriality of symmetric powers,
$\HAT{Q_{\ell-1}}{d}\HAT{[A_{\ell-1}]}{d}\HAT{T_{\ell-1}}{d}=\HAT{[A'_{\ell-1}]}{d}$.
Further, the matrices $\HATTT{M}{\ell-1,d}{-1}\HAT{Q_{\ell-1}}{d}\HATT{M}{\ell-1,d} $ and 
$\HATTT{M}{\ell-1,d}{-1}\HAT{T_{\ell-1}}{d}\HATT{M}{\ell-1,d}$ belong to $\GL_{\PAR_{\ell-1}(d)}(\MA)$.
Indeed, these matrices are $\MA$-valued by Theorem~\ref{inte}, and their determinants are invertible elements of $\MA$ since that is the case 
for the determinants of $Q_{\ell-1}$, $T_{\ell-1}$. 
Therefore,
\begin{align}\label{eq:gcproof}
C_{\ell,d}^v \equiv_{\MA} 
\HATTT{M}{\ell-1,d}{-1}\HAT{[A_{\ell-1}]}{d}\HATT{M}{\ell-1,d}
=\HATTT{M}{\ell-1,d}{-1}\HAT{[A'_{\ell-1}]}{d}\HATT{M}{\ell-1,d}.
\end{align}
It follows from Definition~\ref{def:Sd} that (see \S\ref{matconv})
\begin{align*}
 \HAT{[A'_{\ell-1}]}{d} = {{\textstyle\bigoplus_{\substack{d_i\geq 0 \\ \sum_{i=1}^{\ell-1}d_i=d}}  
\left(\left(\bigotimes_{j=1}^{\ell-2}1_{\PAR(d_j)}\right)
\otimes 
\DIAG(\{ \textstyle\prod_{i\ge 1} [\ell]_i^{m_i(\lambda)} \mid \lambda \in \PAR(d_{\ell-1}) \})\right)}}.
\end{align*}
Substituting this identity and the formula of Remark \ref{rem:Mell} into \eqref{eq:gcproof}, 
we obtain
\begin{equation}\label{eq:gc2}
C^v_{\ell,d}\CONG{\MA}
\bigoplus_{s=0}^{d}
\left(\HATTT{M}{1,s}{-1}\DIAG(\{{\textstyle\prod_{i\geq 1}[\ell]^{m_i(\lambda)}_{i}}\mid\lambda\in\PAR(s)\})\HATT{M}{1,s}\right)^{\oplus |\PAR_{\ell-2}(d-s)|}.
\end{equation}

By Corollary \ref{Aaux2}, we have $C^v_{\ell,d}\CONG{\MA}\TRANS{C^v_{\ell,d}}$.
Hence, transposing both sides of~\eqref{eq:gc2} and using the fact that $\TRANS{M_{1,s}}=M_{s}$ (see Remark~\ref{rem:Mell}), 
we obtain \eqref{eq:gc}.
\QED

\Rem
In the rest of the paper, we 
will see an implication of Conjecture~\ref{ourgradedconjecture}
for ``invariant factors'' of $C^v_{\ell,d}$ 
(Proposition~\ref{gradedHillcon}) 
and give evidence for Conjecture~\ref{ourgradedconjecture}  (Theorem \ref{maintheorem}).
For Cartan matrices $X$ of the other simply-laced finite types (D and E), we can prove the existence of $Q_{X},T_{X}\in \GL_{I}(\MA)$ such that
\bna
\item
$Q_{X}[X]T_{X}=\DIAG(\{1,\ldots,1,\det[X]\})$ for $X\ne D_{2m}$, 
\item
$Q_{X}[X]T_{X}=\DIAG(\{1,\ldots,1,[2],[2]_{2m-1}\})$ for $X=D_{2m}$, 
\ee
where $m\geq 2$ (for the ungraded case $v=1$, see ~\cite[Table 1]{Hil}). For the value of $\det[X]$, see ~\cite[proof of Corollary 4.5]{Tsu}.
These results allow us to analyze $C^v_d(X)$ further:
a conjectural formula for invariant factors of $\QSHM_{\Lambda_0,\mu}(Z)$ for $\mu\in P(\Lambda_0)$
and evidence for it in the spirit of this paper when $Z=X^{(1)}$
and $X$ is of type D or E
as well as
for the \emph{twisted} affine A,D,E cases
will be given elsewhere.
Results on these invariant factors
would provide information on modular reductions of 
$V(\Lambda_0)^{\MA}$, namely, on the structure of 
the $F\otimes_{\MA} U_v^{\MA}$-module $F\otimes_{\MA} V(\Lambda_0)^{\MA}$ and its unique simple quotient, where $F$ is any field, viewed as an $\MA$-module via a 
fixed ring homomorphism $\MA\to F$. 

\enrem

\subsection{Graded Cartan matrices and implications of Conjecture \ref{ourgradedconjecture}}\label{subsec:gradeddef}

\Def\label{GCartanInv} 
Let $A$ be a finite-dimensional graded algebra over a field $\corr$, i.e., $A$ has a decomposition $A=\bigoplus_{i\in\Z}A_i$ 
into $\corr$-vector spaces such that $A_iA_j\subseteq A_{i+j}$ for all $i,j\in\Z$.
\bna
\item We denote by $\GMOD{A}$ the abelian category of finite-dimensional left graded $A$-modules and 
degree preserving $A$-homomorphisms between them. The $n$-component of $M\in\GMOD{A}$ is denoted by $M_n$. 
For $M\in\GMOD{A}$ and $k\in \Z$, 
the shifted graded module $M\langle k\rangle$ of $M$ is defined to be the same module as $M$ with the grading given by 
$(M\langle k\rangle)_n=M_{k+n}$ for all $n\in\Z$.
\item 
Fix a grading on each simple $A$-module, and let $\mathcal S (A)$ be the resulting set of graded simple modules. 
We define the graded Cartan matrix $C^v_A$ of $A$ by 
\begin{align*}
C^v_A = 
\left({\textstyle\sum_{k\in\Z}[\PROC(D):D'\langle -k\rangle]v^{k}}\right)_{D,D'\in\mathcal S(A)}\in\MAT_{\mathcal S(A)}(\MA),
\end{align*}
where $\PROC(D)$ is the projective cover of $D\in\GMOD{A}$. 
\item
Let $\GPPP{A}$ be the full subcategory of $\GMOD{A}$ consisting of 
graded projective $A$-modules. The Cartan pairing is defined as follows:
\begin{align*}
\langle \cdot, \cdot \rangle \colon
[\GPPP{A}]\times [\GMOD{A}]\longrightarrow\MA,\quad
 \langle [P],[M] \rangle = \sum_{k\in\Z}\dim_{\corr}\Hom_{A}(P,M\langle k\rangle)v^k,
\end{align*}
where $[M]$ denotes the image of $M$ in the graded Grothendieck group $[\GMOD{A}]$ of $\GMOD{A}$,
which has an $\MA$-module structure given by $v[N]=[N\langle -1 \rangle]$ for $N\in \GMOD{A}$.
\ee
\edf

\Rem\label{candqrel}
\begin{enumerate}[(a)]
\item 
Each simple $A$-module has a unique grading up to grading shift (see~\cite[Theorem 9.6.8]{NVO}). 
Moreover, each simple graded $A$-module has a unique graded projective cover. Consequently, changing $\mathcal S(A)$ results in $C^v_A$ being conjugated by a diagonal matrix with integer powers of $v$ on the diagonal. Certainly, the $\MA$-unimodular equivalence class of $C^v_A$ does not depend on the choice of $\mathcal S(A)$,
\item\label{candqrelb} $C^v_A = (\langle [\PROC(D)], [\PROC(D')] \rangle)_{D',D\in \mathcal S(A)}$ when
$\corr$ is a splitting field for $A$,
\item $C^v_A$ is a refinement of $C_A$ in the sense that $C^v_A|_{v=1}=C_{A}$.
\end{enumerate}
\enrem

Let $\ell\ge 2$ and $n\ge 0$. 
As usual, a partition $\rho$ is an \emph{$\ell$-core} if $\rho$ contains no rim $\ell$-hooks. 
We denote by $\BLOCK_{\ell}(n)$ the set of tuples $(\rho,d)$ where $\rho$ 
is an $\ell$-core and $d\geq 0$ is an integer such that $|\rho|+\ell d=n$.
It is well known that the set $\BLOCK_{\ell}(n)$ parameterizes the blocks of $\MH_n(\corrr{\ell};\qq{\ell})$ (see~\cite{DJ}).
When $\ell=p$ is a prime, $\BLOCK_{\ell}(n)$ parameterizes the blocks of $\F_p\mathfrak{S}_n$.
We denote by $B^{(\ell)}_{\rho,d}$ the corresponding block algebra of $A:=\MH_n(\corrr{\ell};\qq{\ell})$ 
or of $A:=\F_p\mathfrak{S}_n$ for $(\rho,d)\in\BLOCK_{\ell}(n)$ (for the latter case, $\ell=p$ is a prime).

From now on, we view $B^{(\ell)}_{\rho,d}$ as a graded algebra, with the grading defined by~\cite[Corollary 1]{BK1} (cf.~\S\ref{subsec:graded}). Consequently, 
$A$ becomes graded. 
Clearly, we have
\begin{align}
C^v_A \CONG{\MA}\textstyle\bigoplus_{(\rho,d)\in\BLOCK_{\ell}(n)}C^v_{B^{(\ell)}_{\rho,d}}.
\label{juyopoint2}
\end{align}
In fact, the two sides are equal if appropriate choices are made. 

By~\cite[Theorem 4.18]{BK3}, there is an isomorphism $\iota \colon [\GPPP{A}] \isoto V(\Lambda_0)^{\MA}$ as $U_v(A^{(1)}_{\ell-1})$-modules, 
which identifies the Cartan pairing $\langle \cdot, \cdot\rangle$ 
with the form $\langle \cdot, \cdot \rangle_{\RSH}$ on $V(\Lambda_0)^{\MA}$.
For $(\rho,d)\in\BLOCK_{\ell}(n)$, 
we have $\iota([\GPPP{B^{(\ell)}_{\rho,d}}])=V(\Lambda_0)^{\MA}_{\Lambda_0 - \beta_{\rho,d}}$ where $\beta_{\rho,d} \in \sum_{i\in\wh I}\Z_{\geq 0}\alpha_i$ 
is defined as in~\cite[Definition 5.5(c)]{Tsu} under the identification $\wh I\cong \Z/\ell \Z$.
Noting Remark \ref{candqrel} (\ref{candqrelb}), 
we have $C_{B_{\rho,d}^{(\ell)}}^v\GCONG \RSHM_{\Lambda_0,\Lambda_0 - \beta_{\rho,d}}(A^{(1)}_{\ell-1})$ (see Definition \ref{tsu313}).

By Proposition \ref{tsu316}, Definition \ref{tsu55} and the fact that $\Lambda_0 - \beta_{\rho,d}=w\Lambda_0 -d\delta$ for some $w\in W(A^{(1)}_{\ell-1})$,
we obtain the following result, which is implicit in the proof of~\cite[Theorem 5.6]{Tsu}.

\begin{prop}\label{prop:Shapblock}
Let $\ell\geq 2$ and $n\geq 0$. For any $(\rho,d)\in\BLOCK_{\ell}(n)$,
we have $C_{B_{\rho,d}^{(\ell)}}^v \equiv_{\MA} C_{\ell,d}^v$. 
\end{prop}

The following is an immediate consequence of Theorem \ref{gradedcartan}.

\begin{prop}\label{gradedHillcon}
Let $\ell\geq 2$ and let $d\geq 0$.
If Conjecture~\ref{ourgradedconjecture} is true, then 
\begin{align}
C^v_{\ell,d}\CONG{\MA}
\DIAG\left(\bigsqcup_{s=0}^{d}\{I^v_{\ell}(\lambda)\mid\lambda\in\PAR(s)\}^{|\PAR_{\ell-2}(d-s)|}\right).
\label{ketsuron1}
\end{align}
\end{prop}

\Lemma[{\cite[Lemma 5.5]{BH}}]\label{bhlem}
For any $\ell\geq 2$ and $n\geq 0$, we have the multiset identity 
\begin{align*}
\bigsqcup_{(\rho,d)\in\BLOCK_{\ell}(n)}\bigsqcup_{s=0}^{d}\bigsqcup_{\lambda\in\PAR(s)}\{\CUT_{\ell}(\lambda)\}^{|\PAR_{\ell-2}(d-s)|}
=
\{ \RED_{\ell}(\lambda)\mid \lambda\in\CPAR_{\ell}(n) \}
\end{align*}
where the maps $\CUT_{\ell},\RED_{\ell}\colon \PAR\to\PAR$ are defined as follows for $k\ge 1$:
\begin{align*}
m_k(\RED_{\ell}(\lambda))=\lfloor m_k(\lambda)/\ell\rfloor,\quad
m_k(\CUT_{\ell}(\lambda))=
\begin{cases}
m_k(\lambda) & \text{if } k\notin \ell\Z, \\
0 & \text{otherwise.}
\end{cases}
\end{align*}
\enlemma

Note that $r^{v}_{\ell}(\lambda)=I^{v}_{\ell}(\RED(\lambda))$ and $I^{v}_{\ell}(\lambda)=I^{v}_{\ell}(\CUT(\lambda))$ for all $\lambda \in \PAR$. 
Combining these identities and Lemma \ref{bhlem} with \eqref{juyopoint2} and
Proposition \ref{prop:Shapblock}, we see the following implication.

\Cor\label{mainimplication}
Conjecture \ref{ourgradedconjecture} implies Conjecture \ref{gradedKORcon}.
\encor

\Rem
When $\ell=p^r$ is a prime power, the equivalence~\eqref{ketsuron1}
is nothing but~\cite[Conjecture 6.8]{Tsu}. Similarly, Conjecture~\ref{gradedKORcon} reduces to~\cite[Conjecture 6.18]{Tsu} in this case.
Indeed, the Laurent polynomials $I^v_{p,r} (\lambda)$ and $r^v_{p,r}(\lambda)$ defined in~\emph{loc.~cit.} satisfy 
$I^v_{p,r}(\lambda) = I^v_{p^r} (\lambda)$ and $r^v_{p,r}(\lambda) = r_{p^r}^v (\lambda)$.
\enrem

\section{Combinatorial reductions}\label{sec:4}

\subsection{Variants of unimodular equivalences}\label{varequi}
\Def\label{DVariants}
Let $\MAA$ be a commutative ring, and
let $Y$ and $Z$ be $n\times m$-matrices with entries in $\MAA$.
We say that $Y$ and $Z$ are
\bna
\item unimodularly pseudo-equivalent over $\MAA$ (abbreviated as $Y\PCONG{\MAA} Z$) if we have
$\COKERR{Y}\cong \COKERR{Z}$ as $\MAA$-modules where $\COKERR{T}=\COKER(\MAA^m\to\MAA^n,\boldsymbol{v}\mapsto T\boldsymbol{v})$ for $T\in\{Y,Z\}$,
\item 
Fitting equivalent (abbreviated as $Y\FCONG{\MAA} Z$) 
if $\COKERR{Y}$ and $\COKERR{Z}$ have the same Fitting invariants (see~\cite[\S3.1]{Nor}), i.e.,
we say that $Y\FCONG{\MAA} Z$ if $\FITT_d(Y)=\FITT_d(Z)$ 
whenever $0\leq d< r:=\min\{m,n\}$
where the $d$-th Fitting ideal $\FITT_d(T)$ of $T\in\{Y,Z\}$ over $R$ is the ideal of $R$ generated 
by all $(r-d)\times (r-d)$-minors of $T$.
\ee
\edf

\Prop\label{threeequiveasyfacts}
The following general statements hold:
\bna 
\item\label{genericcase} $Y\CONG{\MAA} Z\Longrightarrow Y\PCONG{\MAA} Z\Longrightarrow Y\FCONG{\MAA} Z$.
\item\label{basechange} for a ring homomorphism $\phi\colon \MAA\to\MAA'$ (see \S\ref{CoAl}), we have the implications
$Y\equiv_R Z\Longrightarrow\phi(Y)\equiv_{R'} \phi(Z)$,
$Y\equiv'_R Z\Longrightarrow\phi(Y)\equiv'_{R'} \phi(Z)$ and 
$Y\equiv^F_R Z\Longrightarrow\phi(Y)\equiv^F_{R'} \phi(Z)$.
\item\label{directsumbehav} Let $(X_\lambda)_{\lambda\in\Lambda}$ and $(Y_\lambda)_{\lambda\in\Lambda}$ be
families of $\MAA$-valued matrices where $\Lambda$ is a finite set and for each $\lambda\in \Lambda$ the matrix $X_{\lambda}$ has the same dimensions as $Y_{\lambda}$.
Then, for any $\SIM\in\{\CONG{\MAA},\PCONG{\MAA},\FCONG{\MAA}\}$, we have the implication $\forall\lambda\in\Lambda,X_\lambda\SIM Y_{\lambda}\Longrightarrow \bigoplus_{\lambda\in \Lambda}X_\lambda\SIM\bigoplus_{\lambda\in \Lambda}Y_\lambda$.
\item\label{elemdivthm} $Y\FCONG{\MAA} Z\Longrightarrow Y\CONG{\MAA} Z$ when $\MAA$ is a PID.
\item\label{locgloFit} $Y\FCONG{\MAA} Z\Longleftrightarrow\forall\MEE\in\MSPEC(\MAA),Y\FCONG{\MAA_{\MEE}} Z$,
\item\label{localringcase} $Y\PCONG{\MAA} Z\Longrightarrow Y\CONG{\MAA} Z$ when $\MAA$ is a semiperfect ring.
\ee
\enprop

\Proof
(\ref{genericcase}) is obvious and (\ref{elemdivthm}) is ``Elementary Divisor Theorem''.
The cases of $\equiv$ and $\equiv^F$ in (\ref{basechange}) are obvious.
The right exactness of the functor $\MAA'\otimes_{\MAA}\textrm{-}$ implies that 
$\MAA'\otimes_{\MAA}\COKERR{Y}\cong \COKERR{\phi(Y)}$.
Thus, the case $\equiv'$ follows.
When $\SIM\in\{\CONG{\MAA},\PCONG{\MAA}\}$, (\ref{directsumbehav}) is obvious.
The case $\FCONG{\MAA}$ follows from the equality 
$\FITT_d(Y\oplus Z)=\sum_{d_1+d_2=d}\FITT_{d_1}(Y)\FITT_{d_2}(Z)$ (see ~\cite[\S3.1, Exercise 3]{Nor}).
(\ref{locgloFit}) follows from the fact that for ideals $I$ and $J$ in $R$, we have 
$I=J \Longleftrightarrow \forall \MEE\in \MSPEC(\MAA), I_{\MEE} = J_{\MEE}$
(see e.g.~\cite[Chapter IV, Corollary 1.4]{Kun}). 
For (\ref{localringcase}), when $\MAA$ is a local ring, 
for any given two $\MAA$-module surjections $\alpha\colon R^{k}\twoheadrightarrow M, \, \beta\colon R^{k}\twoheadrightarrow N$,
we can lift any $\MAA$-module isomorphism $f\colon M\isoto N$ to the isomorphism $g\colon R^{k}\isoto R^{k}$ such that $f\circ\alpha=\beta\circ g$
by the Nakayama Lemma. Thus, (\ref{localringcase}) holds when $\MAA$ is local. 
Since a semiperfect ring is the same thing as a finite direct product of local rings~\cite[(23.11)]{Lam}, 
(\ref{localringcase}) follows by (\ref{basechange}) (see also ~\cite[(4.3)]{LR}).
\QED

By the reasoning used to prove Corollary \ref{gradedHillcon} and Corollary \ref{mainimplication},
Proposition \ref{threeequiveasyfacts} (\ref{basechange}) and (\ref{directsumbehav}) imply:

\begin{corollary}\label{cor:sp}
Let $\MAA$ be a commutative ring with a ring homomorphism $\phi \colon \MA \to \MAA$ and $\SIM\in\{\CONG{\MAA},\PCONG{\MAA},\FCONG{\MAA}\}$.
Suppose that Conjecture~\ref{ourgradedconjecture} holds when we specialize $\MA$ and $\CONG{\MA}$ to $\MAA$ 
and $\SIM$ respectively via $\phi$, i.e., that 
\begin{align*}
 \phi\left(M_n\DIAG(\{J^{v}_{\ell}(\lambda)\mid\lambda\in\PAR(n)\})M_n^{-1} \right)
\SIM \DIAG(\{ \phi(I^{v}_{\ell}(\lambda))\mid\lambda\in\PAR(n)\})
\end{align*}
for all $n\ge 0$. 
Then we have $\phi(Y) \SIM \phi(Z)$ if either 
\begin{enumerate}[(i)]
 \item $Y$ and $Z$ are the matrices on the two sides of~\eqref{ketsuron1}, or 
 \item $Y$ and $Z$ are the matrices on the two sides of~\eqref{eq:grKOR}. 
\end{enumerate}
\end{corollary}

Throughout, we omit $\phi(\textrm{-})$ if $\phi$ is evident when we apply Proposition \ref{threeequiveasyfacts} (\ref{basechange}).

\subsection{A pseudo-equivalence over $\Z_{(p)} [v,v^{-1}]$}\label{SSpseudo}

\Def\label{scacycpol} For $n\geq 3$, we denote by $\Phi_n\in\Z[v]$ the $n$-th cyclotomic polynomial
and put $\Psi_n=v^{-\phi(n)/2}\Phi_n\in\MA^{\BAR}$ where $\phi$ is the Euler function: $\phi(n)=\#(\Z/n\Z)^{\times}$.
\edf

It is easy to see that, for $n,m\geq 1$,
\begin{align}
\textstyle
[n]_m=\prod_{b \le 3, \, 2mn\in b\Z, \, 2m\not\in b\Z}\Psi_b.
\label{hantei1}
\end{align}
Thus, each $I^{v}_{\ell}(\lambda)$ and $J^{v}_{\ell}(\lambda)$ is a product of certain scaled cyclotomic polynomials $\Psi_b$. 

\Def\label{Drho}
Let $p\in \PRIMES$ and $z\in\N\setminus p\Z$. Let $P=\prod_{i\in I}\Psi_{b_i}$ be a finite product of scaled cyclotomic polynomials
(with $b_i\geq 3$ for all $i\in I$, as in Definition \ref{scacycpol}). We define $\RHO^{(p)}_z(P)=\prod_{(b_i)_{p'}=z}\Psi_{b_i}$.
\edf

Recall the famous equality $\#\CPAR_{s}(n)=\#\RPAR_{s}(n)$ for $s\geq 1$ and $n\geq 0$.
We reserve the symbol $\CRBIJ_{s,n}$ for an arbitrary bijection $\CRBIJ_{s,n}\colon \RPAR_{s}(n)\isoto\CPAR_{s}(n)$ and put
$\CRBIJ_{s}=\sqcup_{n\geq 0}\, \CRBIJ_{s,n}$.
As a standard choice, we can take the Glaisher bijection (see~\cite[\S4]{ASY}, for example) for $s\geq 2$ or the Sylvester bijection for $s=2$ 
(see ~\cite{Bes}, for example).

\Def\label{keybijection}
Fix $M\ge 1$. For any $\lambda\in \PAR$, consider the decomposition
$\lambda=\lambda_{\mathsf{div}}+\lambda_{\mathsf{reg}}$ defined by
$m_{a}(\lambda_{\mathsf{div}})=M\lfloor m_a(\lambda)/M\rfloor$, $m_{a}(\lambda_{\mathsf{reg}})=m_a(\lambda)-m_{a}(\lambda_{\mathsf{div}})$ 
for $a\geq 1$.
We define a size-preserving auto-bijection $\KEYBIJECTION{M}\colon \PAR\isoto\PAR$ by
$\KEYBIJECTION{M}(\lambda)=\mu+\CRBIJ_{M}(\lambda_{\mathsf{reg}})$
where $m_{aM}(\mu)=m_a(\lambda_{\mathsf{div}})/M$
for $a\geq 1$ and $m_{b}(\mu)=0$ for all $b\not\in M\Z$.
\edf

\begin{definition}\label{def:fg}
 For $\ell\ge 2$, $k,t\ge 1$ and $p\in \PRIMES$, define
\begin{align*}
g_{k,t}^{(\ell,p)}=
\begin{cases}
[\ell_{p'}]_{k_{p'}(\ell t)_\p} & \text{if } \nu_p(k)\geq \nu_p(\ell), \\
[\ell t_\p/k_\p]_{k} & \text{if } \nu_p(k)< \nu_p(\ell), 
\end{cases}
\end{align*}
and set $I^v_{\ell,p}(\lambda)=\prod_{k\geq 1}\prod_{t=1}^{m_k(\lambda)}g_{k,t}^{(\ell,p)}$ for $\lambda \in \PAR$. 
\end{definition}

Further, we define $f_{k,t}^{(\ell)}= [\ell_k t_{\pi(\ell_k)}]_{(\ell,k) t_{\pi(\ell_k)'}}$ 
and note that $I_{\ell}^v (\lambda) = \prod_{k\ge 1} \prod_{t=1}^{m_k(\lambda)} f_{k,t}^{(\ell)}$.

\begin{prop}
\label{areductionthm}
Let $p$ be a prime, $\ell\geq 2$, and $z\in\N\setminus p\Z$. For any $\lambda\in\PAR$, we have 
\begin{align*}
\RHO^{(p)}_z(I^v_{\ell}(\lambda))=\RHO^{(p)}_z(I^v_{\ell,p}(\KEYBIJECTION{z/\GCD(z,2\ell)}(\lambda))).
\end{align*}
\end{prop}

First, we need two lemmas. 
Fix $p$ and $z$ to be as in the statement of the proposition. For any $k,t\ge 1$, define
\[
\begin{split}
 \mathcal F_{k,t,z}^{(\ell,p)} &= \{ s\ge 0 \mid 2\ell t \in z p^s \Z \text{ and } 2 (\ell,k) t_{\pi(\ell_k)'} \notin zp^s \Z \}, \\
 \mathcal G_{k,t,z}^{(\ell,p)} &= \begin{cases}
                                  \{ s\ge 0 \mid 2 \ell t_\p k_{p'} \in zp^s \Z \text{ and } 2 (\ell t)_\p k_{p'} \notin zp^s \Z\} & \text{if } \nu_p(k) \ge \nu_p(\ell), \\
                                   \{s\ge 0 \mid 2 \ell t_\p k_{p'} \in zp^s \Z \text{ and } 2k\notin zp^s \Z \}  & \text{if } \nu_p (k) <\nu_p(\ell).
                                 \end{cases}
\end{split}
\]

The following is an immediate consequence of~\eqref{hantei1} and the definitions:
\begin{lemma}\label{lem:prod}
 For $k,t\ge 1$, we have $\rho_z^{(p)} (f_{k,t}^{(\ell)}) = \prod_{s\in \mathcal F_{k,t,z}^{(\ell,p)}} \Psi_{zp^s}$ and 
$\rho_z^{(p)} (g_{k,t}^{(\ell,p)}) = \prod_{s\in \mathcal G_{k,t,z}^{(\ell,p)}} \Psi_{zp^s}$.
\end{lemma}

Define $M=z/(z,2\ell)$. 

\begin{lemma}\label{lem:ktM}
 For any $k,t\ge 1$, we have $\rho_z^{(p)} (f_{k,tM}^{(\ell)}) = \rho_z^{(p)} (g_{kM,t}^{(\ell,p)})$. 
\end{lemma}

\begin{proof}
Due to Lemma~\ref{lem:prod}, it is enough to show that $\mathcal F_{k,tM,z}^{(\ell,p)} = \mathcal G_{kM,t,z}^{(\ell,p)}$. 
Fix $s\ge 0$: we will show that
$s\in \mathcal F_{k,tM,z}^{(\ell,p)}$ if and only if $s\in \mathcal G_{kM,t,z}^{(\ell,p)}$.
Note that $M\notin p\Z$.
If $2 \ell t\notin p^s \Z$, then $s$ belongs to neither of the sets in question, for the first conditions in the definitions of those sets fail. 
Thus, we may assume that $2 \ell t\in p^s \Z$. Since we always have $2\ell M \in z\Z$ (due to the definition of $M$), now the first conditions in the definitions 
of $\mathcal F_{k,tM,z}^{(\ell,p)}$ and $\mathcal G_{kM,t,z}^{(\ell,p)}$ are guaranteed to hold. So we may focus on the second conditions: it remains to show that
\[
 2(\ell,k)(tM)_{\pi(\ell_k)'} \in zp^s \Z \; \Longleftrightarrow \;
\begin{cases}
 2(kM)_{p'} (\ell t)_\p \in z p^s \Z & \text{if } \nu_p (k) \ge \nu_p (\ell), \\
 2kM\in zp^s \Z & \text{if } \nu_p (k) < \nu_p(\ell).  
\end{cases}
\]
This follows from the conjunction of the following two equivalences: 
\begin{align}
2 (\ell,k) (tM)_{\pi (\ell_k)'} \in p^s \Z & \; \Longleftrightarrow \;
\begin{cases}
 2\ell t \in p^s \Z & \text{if } \nu_p (k) \ge \nu_p (\ell), \\
 2k \in p^s \Z & \text{if } \nu_p (k) < \nu_p (\ell)
\end{cases} \quad \text{and} \label{eq:ktM1}
\\
2 (\ell,k) (tM)_{\pi(\ell_k)'} \in z \Z  & \; \Longleftrightarrow \;
2kM \in z \Z. \label{eq:ktM2}
\end{align}
The equivalence~\eqref{eq:ktM1} is immediate in each of the cases on its right-hand side, so it remains only to prove~\eqref{eq:ktM2}.

We always have
\[
 (2 (\ell,k) tM)_{\pi(\ell_k)'} = (2\ell t M)_{\pi(\ell_k)'} \in z_{\pi (\ell_k)'} \Z
\]
since $2\ell M\in z\Z$. Further, $(2k)_{\pi (\ell_k)'} \in (2\ell)_{\pi(\ell_k)'}\Z$, so $(2kM)_{\pi (\ell_k)'} \in (2\ell M)_{\pi (\ell_k)'} \Z \subseteq z_{\pi(\ell_k)'} \Z$. 
This means that the truth values of the statements on both sides of~\eqref{eq:ktM2} do not change if we replace $z$ by $z_{\pi(\ell_k)}$. In other words, it is enough to show that
for all $q\in \pi(\ell_k)$,
\begin{equation}\label{eq:ktM3}
 \nu_q (2(\ell,k)) \ge \nu_q (z) \; \Longleftrightarrow \; \nu_q (2kM) \ge \nu_q (z).
\end{equation}
Now $\nu_q (k) < \nu_q (\ell)$, so $\nu_q (2(\ell,k))= \nu_q (2k)$. Using the definition of $M$, we obtain $\nu_q (2kM) = \nu_q (2k) + \nu_q(z) - \nu_q ((2\ell, z))$. If $\nu_q(2\ell) \ge \nu_q (z)$, then
$\nu_q (2kM) = \nu_q (2k)$ and the equivalence~\eqref{eq:ktM3} is clear. Otherwise, we have $\nu_q (2k) < \nu_q (2\ell) < \nu_q(z)$ and neither side of~\eqref{eq:ktM3} holds. 
\end{proof}

\begin{proof}[Proof of Proposition~\ref{areductionthm}]
Fix $\lambda\in \PAR$, and let $\lambda_{\mathsf{div}},\lambda_{\mathsf{reg}},\mu$ be as in Definition \ref{keybijection}.
It is clear that 
$I^v_{\ell,p}(\KEYBIJECTION{M}(\lambda))=I^v_{\ell,p}(\mu)I^v_{\ell,p}(\CRBIJ_{M}(\lambda_{\mathsf{reg}}))$.
In the expansion $I_{\ell}^v (\lambda) = \prod_{k\ge 1} \prod_{t=1}^{m_k(\lambda)} f_{k,t}^{(\ell)}$,
only $t\in M\Z$ contribute to $\RHO^{(p)}_{z}(I^v_{\ell}(\lambda))$
by Lemma~\ref{lem:prod}, as $2\ell t \in z\Z$ implies $t\in M\Z$. 
Further, $\RHO^{(p)}_z(I^v_{\ell,p}(\CRBIJ_{M}(\lambda_{\mathsf{reg}})))=1$ by the same lemma, 
as $2\ell t_\p k_{p'} \notin z\Z$ for any $k\notin M\Z$ and $t\ge 1$. 
It follows that 
\begin{align*}
\RHO^{(p)}_{z}(I^v_{\ell}(\lambda))&= \prod_{k\ge 1} \prod_{t=1}^{\lfloor m_k (\lambda)/M \rfloor}\RHO^{(p)}_{z}(f^{(\ell)}_{k,tM}), \\
\RHO^{(p)}_{z}(I^v_{\ell,p}(\beta_M (\lambda)))&= \prod_{k\ge 1} \prod_{t=1}^{\lfloor m_k (\lambda)/M \rfloor}\RHO^{(p)}_{z}(g^{(\ell,p)}_{kM,t}).
\end{align*}
The two right-hand sides are equal by Lemma~\ref{lem:ktM}.
\end{proof}

\Prop\label{CRT}
Let $\MAA$ be a commutative ring and let $a\in \MAA$.
Suppose that $a=\prod_{\lambda\in \Lambda}\prod_{x\in T_{\lambda}}x$ for 
a finite set $\Lambda$ and a family of finite multisets $(T_{\lambda}\subseteq \MAA)_{\lambda\in\Lambda}$ such that
any $x\in T_{\lambda}$ and $x'\in T_{\lambda'}$ are coprime (i.e., $xy+x'y'=1$ for some $y,y'\in\MAA$) whenever $\lambda\ne\lambda'$.
Then, as $\MAA$-modules, we have
\begin{align*}
R/(a)\cong
\bigoplus_{\lambda\in\Lambda}R/({\textstyle{\prod_{x\in T_{\lambda}}}}x).
\end{align*}
\enprop

\Proof
Observe that $(\prod_{x\in T_{\lambda}}x)_{\lambda\in\Lambda}$ are pairwise coprime (in the above sense): this follows from the elementary fact that if $x,y, z\in R$ and $x,y$ are both coprime to $z$, then $xy$ is coprime to $z$.
Now the proposition follows from the Chinese remainder theorem for ideals. 
\QED

\Cor\label{reductiontop}
For $p\in\PRIMES$ and $n\geq 0,\ell\geq 2$, we have 
\begin{align*}
\DIAG(\{I^v_{\ell}(\lambda)\mid\lambda\in\PAR(n)\})
\PCONG{\MAp{p}}
\DIAG(\{I^v_{\ell,p}(\lambda)\mid\lambda\in\PAR(n)\}).
\end{align*}
\encor

\Proof
Whenever $3\leq b<c$ and $c/b$ is not a $p$-power, 
there exist $u,w\in\MAp{p}$ such that $\Psi_bu+\Psi_cw=1$ (see~\cite[Lemma 2]{Fil}).
By Proposition \ref{CRT}, we have
\begin{align}
\COKERR{\DIAG(\{f(\lambda)\mid\lambda\in\PAR(n)\})}\cong
\bigoplus_{\lambda\in\PAR(n)}\bigoplus_{z\in \N\setminus p\Z}\MAp{p}/(\rho_z^{(p)} (f(\lambda)))
\label{cokerf}
\end{align}
as $\MAp{p}$-modules
for any $f\in\{I^v_{\ell},I^v_{\ell,p}\}$. By Proposition \ref{areductionthm}, the isomorphism class of \eqref{cokerf} 
does not depend on the choice of $f$.
\QED

\subsection{A conditional proof of Theorem~\ref{maintheorem}}
\begin{proof}[Proof of Theorem~\ref{maintheorem} (\ref{ASYc})]
For $\lambda\in\PAR$ and $\ell\geq 2$ we have $I^v_{\ell,p}(\lambda)=J^{v}_{\ell}(\lambda)$ for every sufficiently large $p\in\PRIMES$ 
(as $g_{k,t}^{(\ell,p)}=[\ell]_k$ for $p>\max(k,t,\ell)$).
Thus, Theorem \ref{maintheorem}(\ref{ASYc}) is a consequence of Corollary \ref{reductiontop} (note that $M_n \in \GL_{\PAR (n)} (\Q)$).
\end{proof}

Recall the matrix $N_n^{(p)}$ defined in \S\ref{subsec:22}. 
Applying Proposition \ref{reductiontoN} (\ref{ketsuron26a}) for
\begin{align}
\MAA=\MAp{p},\quad
f=J^v_{\ell},\quad
g=I^v_{\ell,p},\quad
(r_j=\INFL_j\colon \MAA\to\MAA,v\mapsto v^j)_{j\in \N \setminus p\Z},
\label{ingre}
\end{align}
we get the following.

\Prop\label{Zpvalued}
For any $p\in \PRIMES$, $\ell\geq 2$ and $n\geq 0$, the matrix $N^{(p)}_{n}\DIAG(\{J^{v}_{\ell}(\lambda)\mid\lambda\in\POW_p(n)\})\NPI{n}$ is $\MAp{p}$-valued.
\enprop

Further,
in \S\ref{sec:5}, we will prove the following result. 

\Th\label{meidai}
Suppose that $0\ne \theta=a/b\in\Q$, where $a,b\in \Z$ and $\GCD(a,b)=1$. Let $p$ be a prime such that $a,b\notin p\Z$. Then,
for any $\ell\geq 2$ and $n\geq 0$, we have
\begin{align}
\label{mokuteki}
N^{(p)}_{n}\DIAG(\{J^{v}_{\ell}(\lambda)|_{v=\theta}\mid\lambda\in\POW_p(n)\})\NPI{n}\CONG{\Z_{(p)}} 
\DIAG(\{I^{v}_{\ell,p}(\lambda)|_{v=\theta}\mid\lambda\in\POW_p(n)\}).
\end{align}
\enth

\begin{proof}[Proof of Theorem~\ref{maintheorem} (\ref{subs}) assuming Theorem~\ref{meidai}]\label{condproof}
Fix $0\ne \theta =a/b \in \Q$ with $a,b\in \Z$, $(a,b)=1$. 
By Proposition \ref{threeequiveasyfacts} (\ref{elemdivthm}) and (\ref{locgloFit}),
it is enough to show that, 
\begin{align}\label{eq:pfmain1}
M_n\DIAG(\{J^{v}_{\ell}(\lambda)|_{v=\theta}\mid\lambda\in\PAR(n)\})M_n^{-1}
\CONG{\Z_{(p)}} \DIAG(\{I^{v}_{\ell}(\lambda)|_{v=\theta}\mid\lambda\in\PAR(n)\})
\end{align}
for any $p\in\{p\in\PRIMES\mid p\Z\not\ni ab\}\cong\SPEC(\Z[a/b,b/a])$.
Applying Proposition \ref{reductiontoN} (\ref{ketsuron26b}) for
$\MAA'=\Z_{(p)},\phi=(\MAA\to\MAA',v\mapsto\theta)$
in addition to \eqref{ingre}, we obtain 
\begin{align*}
 M_n\DIAG(\{J^{v}_{\ell}(\lambda)|_{v=\theta}\mid\lambda\in\PAR(n)\})M_n^{-1} 
\CONG{\Z_{(p)}} \DIAG(\{I^{v}_{\ell,p}(\lambda)|_{v=\theta}\mid\lambda\in\PAR(n)\}).
\end{align*}
The unimodular equivalence \eqref{eq:pfmain1} now follows by substituting $v=\theta$ to Corollary \ref{reductiontop} and
Proposition \ref{threeequiveasyfacts} (\ref{genericcase}).
\end{proof}



\section{Proof of Theorem \ref{meidai}}\label{sec:5}
\subsection{Elementary prime power estimates}
The following fact is classical. 
\Prop\label{elementarynumbertheory}
Let $p$ be a prime. Suppose that $x,y\in\Z\setminus p\Z$ satisfy $d:=\nu_p (x-y) \ge 1$. 
If either $p\geq 3$ or $d\geq 2$, then $\nu_p (x^n - y^n) = d+\nu_p(n)$ for all $n\ge 1$. 
\enprop

\Proof
We have $x = y+p^d z$ for some $z\in \Z\setminus p\Z$.
The binomial expansion yields 
\[
x^n -y^n = np^d z y^{n-1} + \sum_{k=2}^n \binom{n}{k} p^{kd} z^k y^{n-k},
\]
so it suffices to show that 
$\nu_p(\binom{n}{k})+kd>d+\nu_p(n)$ for $2\leq k\leq n$.
Since $\nu_p(\binom{n}{k})\geq \nu_p(n)-\nu_p(k!)$, it is enough to prove the inequality $kd-\nu_p(k!)-d>0$. Using~\eqref{eq:factorial}, we easily see that
$\nu_p (k!)\le k-1$ and that this inequality is strict unless $k=p=2$. It follows that the desired inequality holds unless we have $d=1$ and $k=p=2$, which is ruled out by the hypothesis. 
\QED

\Cor\label{CASEONEESTIMATE}
Let $p\in\PRIMES$ and let $a,b\in\Z\setminus p\Z$ with $a^2-b^2\in p\Z$.
Then, we have $\nu_p([n]_m|_{v=a/b})=\nu_p(n)$ for all $n,m\geq 1$.
\encor

\Proof
We may assume that $a^2\ne b^2$: otherwise, $[n]_m|_{v=a/b}=\pm n$. Consider $d\geq 1$ and $z\in\Z\setminus p\Z$ such that $a^2-b^2=p^dz$.
Note that $d\geq 2$ if $p=2$.
By Proposition \ref{elementarynumbertheory}, we have
\[
\nu_p([n]_m|_{v=a/b})=\nu_p\left(\frac{a^{2nm}-b^{2nm}}{a^{2m}-b^{2m}}\right)=(\nu_p(nm)+d)-(\nu_p(m)+d)=\nu_p(n). \qedhere
\]
\QED

\Cor\label{CASETWOESTIMATE}
Let $p\geq 3$ be a prime and $a,b\in\Z\setminus p\Z$. Suppose that $a^2-b^2\not\in p\Z$ and $a^{2n}-b^{2n}\in p\Z$ for some $n\geq 2$.
Put $\gamma=\nu_p(a^{t_0}-b^{t_0})$ where $t_0=\min\{t\geq 1\mid a^{2t}-b^{2t}\in p\Z\}$ ($t_0$ exists and divides $n$).
Then $\nu_p([n]_{p^s}|_{v=a/b})=\nu_p(n)+s+\gamma$ for any $s\geq 0$.
\encor

\Proof
Note that $t_0\not\in p\Z$.
We have
\begin{align*}
\nu_p([n]_{p^s}|_{v=a/b})=\nu_p\left(\frac{a^{2np^s}-b^{2np^s}}{a^{2p^s}-b^{2p^s}}\right)
&=\nu_p(a^{2np^s}-b^{2np^s}) \\
&=\nu_p(2np^s/t_0)+\gamma = \nu_p (n) + s+\gamma,
\end{align*}
where the third equality follows from Proposition~\ref{elementarynumbertheory}. 
\QED

\begin{prop}
\label{CASETHREEESTIMATE}
Let $p\in\PRIMES$ and $n\ge 1$. Suppose that $a,b\in\Z\setminus p\Z$ satisfy $a^{2n}-b^{2n}\not\in p\Z$. 
Then, $\nu_p([n]_{p^s}|_{v=a/b})=0$ for any $s\geq 0$.
\end{prop}

\Proof
The hypothesis implies that $a^{2n p^s} - b^{2np^s} \notin p\Z$, 
whence we also have $a^{2p^s} - b^{2 p^s} \notin p\Z$. 
Since $\nu_p([n]_{p^s}|_{v=a/b}) = \nu_p ( (a^{2np^s} - b^{2np^s} )/ (a^{2p^s} - b^{2 p^s}) )$, the result follows. 
\QED

\subsection{Some definitions and results from~\cite[\S5]{Evs}}
\label{SSevs}
For the remainder of \S\ref{sec:5}, we fix a prime $p$ and an integer $n\ge 0$. 
The matrices considered in the sequel implicitly depend on these parameters. 
Let
$\ell \ge 2$ and $\theta=a/b\in \Q\setminus \{0\}$ be as in the statement of Theorem~\ref{meidai}. 
We set $r= \nu_p (\ell)$. In what follows, diagonal matrices are generally denoted by lower-case letters. 

Define the matrices
$b^{(\ell, \theta)}=\DIAG(\{J^{v}_{\ell}(\lambda)|_{v=\theta}\mid\lambda\in\POW_p(n)\})$ and $z=\DIAG(\{z_{\lambda}\mid\lambda\in\POW_p(n)\})$, where
$z_{\lambda}$ is given by~\eqref{eq:z}. 

\Lemma[{\cite[Lemma 5.1]{Evs}}]\label{invtrans}
The matrices $\NPI{n}$
 and 
$z^{-1}(\TRANS{N_n^{(p)}})$ are column equivalent over $\Z_{(p)}$.
\enlemma

We write $N= N_n^{(p)}$. 
It follows from the lemma that the left-hand side of Theorem~\ref{meidai} is unimodularly equivalent over $\Z_{(p)}$ to the matrix 
$Y:=Nb^{(\ell,\theta)} z^{-1}(\TRANS{N})$, so Theorem~\ref{meidai} is equivalent to the identity
\begin{equation}\label{eq:Y}
Y\equiv_{\Z_{(p)}} \DIAG(\{I^{v}_{\ell,p}(\lambda)|_{v=\theta}\mid\lambda\in\POW_p(n)\}).
\end{equation}

\Def\label{sixmats}
\begin{enumerate}[(a)]
\item For $\lambda\in \POW_p$, we define partitions 
$\lambda^{<r},\lambda^{\geq r},\OVERy{\lambda}{r}\in\POW_p$
by setting $m_{p^i}(\lambda^{\ge r})=m_{p^{r+i}}(\lambda)$,
\begin{align*}
m_{p^i}(\lambda^{<r})=
\begin{cases}
m_{p^i}(\lambda) & \text{if } i<r, \\
0 & \text{if } i\geq r,
\end{cases}
\quad
m_{p^i}({\OVERy{\lambda}{r}})=
\begin{cases}
m_{p^i}(\lambda) & \text{if }i<r, \\
\sum_{j\geq r}p^{j-r}m_{p^j}(\lambda) & \text{if }i=r, \\
0 & \text{if }i>r,
\end{cases}
\end{align*}
for all $i\ge 0$.
\item For $\lambda\in \POW_p$, we set $x_{\lambda}=\prod_{s\geq 0}m_{p^s}(\lambda)!$ and $y_{\lambda}=\prod_{s\geq 0}p^{sm_{p^s}(\lambda)}$, so that
$z_{\lambda}=x_{\lambda} y_{\lambda}$.
\item 
We define the following seven elements of $\MAT_{\POW_p(n)}(\Z)$:
$x=\DIAG(\{ x_{\lambda} \}_{\lambda} )$, 
$x_{<r}=\DIAG(\{x_{\lambda^{<r}}\}_{\lambda})$,
$x^{\geq r}=\DIAG(\{x_{\lambda^{\geq r}}\}_{\lambda})$,
$y^{<r}=\DIAG(\{y_{\lambda^{<r}}\}_{\lambda})$,
$y^{\geq r}=\DIAG(\{y_{\lambda^{\geq r}}\}_{\lambda})$,
$\widetilde{y}^{(r)}=\DIAG(\{{\textstyle{\prod_{i\geq r}p^{r m_{p^i}(\lambda)}}}\}_{\lambda})$ and
$C^{(r)}$, where the latter is given by
\begin{align*}
(C^{(r)})_{\lambda,\mu}=
\begin{cases}
(N^{(p)}_{|\lambda^{\geq r}|})_{\lambda^{\geq r},\mu^{\geq r}} & \text{if } \OVERy{\lambda}{r}=\OVERy{\mu}{r},\\
0 & \text{if } \OVERy{\lambda}{r}\ne\OVERy{\mu}{r}.
\end{cases}
\end{align*}
Here, $\lambda,\mu$ run over all elements of $\POW_p(n)$. 
\end{enumerate}
\edf

Put $\KA^{(p,r)}=\{\lambda\in\POW_p\mid \OVERy{\lambda}{r}=\lambda\}\subseteq \POW_p$.
For $\kappa\in \KA^{(p,r)}$, set $\POW_{p,r}(n,\kappa):=\{\lambda\in\POW_p(n)\mid \OVERy{\lambda}{r}=\kappa\}$.
Observe that there is a bijection
\begin{align}\label{eq:kappa}
\POW_{p,r}(n,\kappa)\isoto\POW_p(m_{p^r}(\kappa)),\quad
\lambda\MAPSTO\lambda^{\geq r}.
\end{align}

We will call a matrix $Z\in\MAT_{\POW_p(n)}(\Q)$ \emph{block-diagonal} if $Z_{\lambda,\mu}=0$ 
for all $\lambda,\mu\in\POW_p(n)$ with $\OVERy{\lambda}{r}\ne\OVERy{\mu}{r}$. In particular, $C^{(r)}$ is block-diagonal. 
Applying Lemma \ref{invtrans}
to each $\kappa\in\KA^{(p,r)}_n:=\POW_p(n)\cap\KA^{(p,r)}$
and noting that $\POW_p(n)=\bigsqcup_{\kappa\in\KA^{(p,r)}_n}\POW_{p,r}(n,\kappa)$, we see that there exists a block-diagonal matrix $W^{(r)}\in \GL_{\POW_p(n)}(\Z_{(p)})$ such that
$(C^{(r)})^{-1}W^{(r)}=(x^{\ge r} y^{\ge r})^{-1}\cdot\TRANS{C^{(r)}}$.
We define $A^{(r)}=N(C^{(r)})^{-1}$ and $U^{(r)}=(x^{<r})^{-1} A^{(r)}$, so that $N=x^{<r} U^{(r)} C^{(r)}$. 

In \S\ref{CASEONE}, \S\ref{CASETWO}, and \S\ref{CASETHREE}, we consider separate cases and use 
Corollaries~\ref{CASEONEESTIMATE} and~\ref{CASETWOESTIMATE} and Proposition~\ref{CASETHREEESTIMATE} respectively.
The cases of \S\ref{CASEONE} and \S\ref{CASETWO} will require the following specialization of \cite[Lemma 5.6]{Evs}.


\Lemma\label{DVRlinearalgebra}
Let $\MAA$ be a DVR with valuation $\nu\colon K^{\times}\twoheadrightarrow\Z$, where $K$ is the field of fractions of $\MAA$.
Let $I$ be a finite set. 
Suppose that $P$, $Q$, $s=\DIAG(\{s_i\mid i\in I\})$ and $t=\DIAG(\{t_i\mid i\in I\})$ are elements of $\GL_I(K)$
such that 
\begin{align*}
\nu(P_{ij}-\delta_{ij})>\frac{\nu(t_i)-\nu(t_j)}{2}\quad \text{and} \quad
\nu(Q_{ij}-\delta_{ij})>\frac{\nu(t_j)-\nu(t_i)}{2}
\end{align*}
for all $i,j\in I$. 
Then $sPtQs \CONG{\MAA} s^2t$. 
\enlemma

\begin{proof}
 Apply~\cite[Lemma 5.6]{Evs} with $\alpha_i = v(t_i)/2$ and $\beta_i = -v(t_i)/2$. Verifying the hypotheses is straightforward. 
\end{proof}

\subsection{Case $a^2-b^2\in p\Z$}\label{CASEONE}
This is a generalization of the case $v=1$, and we generalize the proof in~\cite[\S5]{Evs}, Proposition~\ref{intp} being an extra needed ingredient. 

Observe that $z=x^{<r} x^{\geq r} y^{<r} y^{\geq r} \widetilde{y}^{(r)}$.
Put 
$b^{(<r,\ell,\theta)}=\DIAG(\{J^{v}_{\ell}(\lambda^{<r})|_{v=\theta}\mid\lambda\in\POW_p(n)\})$, 
$b^{(\geq r,\ell,\theta)}=b^{(\ell,\theta)} (b^{(<r,\ell,\theta)})^{-1}$
and $d=b^{(<r,\ell,\theta)}(x^{<r}y^{<r})^{-1}$.
Let
\[ X=C^{(r)}b^{(\geq r,\ell,\theta)}(\widetilde{y}^{(r)})^{-1}(C^{(r)})^{-1}W^{(r)}. \]
Note that all the matrices in this product are block-diagonal, so $X$ is block-diagonal. 
Setting also $V=X \cdot \TRANS{U^{(r)}} \cdot X^{-1}$, we have
\begin{align}
\notag Y
&=
Nb^{(\ell,\theta)}_{n}z^{-1}\cdot\TRANS{N}  \displaybreak[1] \\
\notag &=
x^{<r} U^{(r)} C^{(r)} b^{(<r,\ell,\theta)} b^{(\geq r,\ell,\theta)} (x^{<r} x^{\geq r} y^{<r}
y^{\geq r} \widetilde{y}^{(r)})^{-1}\cdot \TRANS{C^{(r)}}\cdot \TRANS{U^{(r)}}x^{<r} \displaybreak[1] \\
\label{eq:l1} &=
x^{<r} U^{(r)} C^{(r)} d\cdot b^{(\geq r,\ell,\theta)} (\widetilde{y}^{(r)})^{-1}
(x^{\geq r} y^{\geq r})^{-1}\cdot \TRANS{C^{(r)}}\cdot \TRANS{U^{(r)}}x^{<r} \displaybreak[1] \\
\label{eq:l2} &=
x^{<r} U^{(r)} C^{(r)} d\cdot b^{(\geq r,\ell,\theta)} (\widetilde{y}^{(r)})^{-1}
(C^{(r)})^{-1} W^{(r)} \cdot \TRANS{U^{(r)}}x^{<r} \displaybreak[1] \\
\notag &=
x^{<r} U^{(r)} C^{(r)} d\cdot (C^{(r)})^{-1}
C^{(r)} b^{(\geq r,\ell,\theta)} (\widetilde{y}^{(r)})^{-1} (C^{(r)})^{-1}W^{(r)} \cdot \TRANS{U^{(r)}}x^{<r} \displaybreak[1] \\
\label{eq:l3} &= 
x^{<r} U^{(r)} C^{(r)} d (C^{(r)})^{-1} X \cdot \TRANS{U^{(r)}}x^{<r} \displaybreak[1] \\
\label{eq:l4} &=
x^{<r} U^{(r)} d\cdot X \cdot \TRANS{U^{(r)}}x^{<r} \displaybreak[1] \\
\label{eq:l5} &= x^{<r} U^{(r)} d VX x^{<r} \displaybreak[1] \\
\label{eq:l6} &= x^{<r} U^{(r)} d V x^{<r} X \displaybreak[1] \\
\label{eq:l7} &\CONG{\Z_{(p)}} 
x^{<r} U^{(r)} d Vx^{<r}.
\end{align}
Here, Equations~\eqref{eq:l1},~\eqref{eq:l2},~\eqref{eq:l3} and~\eqref{eq:l5} follow from the defining equations of the matrices $d$, $W^{(r)}$, $X$ and $V$ respectively. 
Equations~\eqref{eq:l4} and~\eqref{eq:l6} follow from the facts that the matrices $C^{(r)}$ and $X$ are block diagonal and that any block-diagonal matrix commutes with 
$b^{(<r,\ell,\theta)}$, $x^{<r}$, $y^{<r}$, and hence also with $d$. 


The equivalence~\eqref{eq:l7} is due to the fact that $X\in \GL_{\POW_p (n)} (\Z_{(p)})$, which may be proved as follows. 
Note that $(\widetilde{y}^{(r)})_{\lambda,\lambda} = p^{r\ell(\lambda^{\ge r})}$ for all $\lambda\in \POW_p (n)$.
We have
\begin{equation}\label{eq:X}
\begin{split}
 & C^{(r)} b^{(\ge r, \ell,\theta)} (\widetilde{y}^{(r)})^{-1} (C^{(r)})^{-1}  \\
&= \bigoplus_{\kappa\in \mathcal K^{(p)}} N_{m_{p^r} (\kappa)}^{(p)} 
\DIAG(\{ p^{-r\ell (\mu)} \prod_{j\ge 0} [\ell]_{p^{r+j}}^{m_{p^j} (\mu)} \mid \mu\in \POW_p (m_{p^r} (\kappa))\}) (N_{m_{p^r} (\kappa)}^{(p)})^{-1}, 
\end{split}
\end{equation}
where the right-hand side is interpreted via the identification~\eqref{eq:kappa}: this identity is readily verified from the definitions of the matrices involved. By Proposition~\ref{intp}, the right-hand side of~\eqref{eq:X} is $\Z_{(p)}$-valued. 
By Corollary~\ref{CASEONEESTIMATE}, we have
\[
 \nu_p ((b^{(\ge r, \ell, \theta)})_{\lambda,\lambda}) = \nu_p \left( \prod_{i\ge r} [\ell]_{p^i}^{m_{p^i (\lambda)}} \right) 
=r \ell(\lambda^{\ge r}) = \nu_p ((\widetilde{y}^{(r)})_{\lambda,\lambda}).
\]
for $\lambda \in \POW_p (n)$.
So the $p$-adic valuation of the determinant of the left-hand side of~\eqref{eq:X} is $0$. Therefore, the left-hand side of~\eqref{eq:X} belongs to 
$\GL_{\POW_p (n)} (\Z_{(p)})$. Since $W^{(r)} \in \GL_{\POW_p (n)} (\Z_{(p)})$, we see that $X\in \GL_{\POW_p (n)} (\Z_{(p)})$, as claimed. 

We will complete the proof by applying Lemma~\ref{DVRlinearalgebra} to the product $x^{<r} U^{(r)} d Vx^{<r}$. 
For $\lambda\in\POW_p(n)$,  define 
\begin{align*}
f^{(r)}_{\lambda}=\sum_{0\leq s<r}(r-s)m_{p^s}(\lambda)\quad \text{and} \quad 
e^{(r)}_{\lambda}=\sum_{0\leq s<r}\nu_p(m_{p^s}(\lambda)!).
\end{align*}
We have $\nu_p((x^{<r}_{n})_{\lambda,\lambda})=e^{(r)}_{\lambda}$. Using Corollary~\ref{CASEONEESTIMATE}, we obtain
\begin{align*}
\nu_p(d_{\lambda,\lambda})&=\nu_p (b^{(<r,\ell,\theta)}_{\lambda,\lambda}) - \nu_p (y^{<r}_{\lambda,\lambda}) - \nu_p (x^{<r}_{\lambda,\lambda}) 
= \sum_{0\le s<r} \nu_p ([\ell]_{p^s}^{m_{p^s} (\lambda)} )  - \sum_{0\le s<r} sm_{p^s} (\lambda) - e_{\lambda}^{(r)} \\
&= f_{\lambda}^{(r)} - e_{\lambda}^{(r)} =: k_{\lambda}^{(r)} \qquad \text{and} \\
\nu_p(I^{v}_{\ell,p}(\lambda)|_{v=\theta}) &= \sum_{s\ge 0}  \sum_{t=1}^{m_{p^s}(\lambda)} \nu_p (g_{p^s,t}^{(\ell,p)}|_{v=\theta}) 
=\sum_{0\le s<r} \sum_{t=1}^{m_{p^s} (\lambda)} (r+ \nu_p (t) -s) 
=f^{(r)}_{\lambda}+e^{(r)}_{\lambda}
\end{align*}
(cf.~Definition~\ref{def:fg}).
The hypotheses of Lemma~\ref{DVRlinearalgebra} are verified as follows. By~\cite[Lemma 5.4]{Evs}, we have 
$\nu_p(U^{(r)}_{\lambda,\mu}-\delta_{\lambda,\mu})>\max\{ k^{(r)}_{\lambda}-k^{(r)}_{\mu}, -1 \}$
for all $\lambda,\mu\in \POW_p (n)$, which implies the first desired inequality, namely
\begin{equation}\label{eq:Uineq}
 \nu_p (U^{(r)}_{\lambda,\mu} - \delta_{\lambda,\mu}) > \frac{ k_{\lambda}^{(r)} - k_{\mu}^{(r)} }{2}.
\end{equation}
The second desired inequality concerns $V= X\cdot \TRANS(U^{(r)}) \cdot X^{-1}$ and follows from~\eqref{eq:Uineq} 
because $X \in \GL_{\POW_p (n)} (\Z_{(p)})$ 
is block-diagonal 
and the right-hand side of~\eqref{eq:Uineq} depends only on $\OVERy{\lambda}{r}$ and $\OVERy{\mu}{r}$.

By Lemma~\ref{DVRlinearalgebra}, we have $Y \equiv_{\Z_{(p)}} (x^{<r})^2 d$, and~\eqref{eq:Y} follows because 
$\nu_p ((x^{<r}_{\lambda,\lambda})^2 d_{\lambda,\lambda}) = f_{\lambda}^{(r)} + e_{\lambda}^{(r)} = \nu_p (I_{\ell,p}^v (\lambda) |_{v=\theta})$
for all $\lambda \in \POW_p (n)$.  

\subsection{Case $a^2-b^2\not\in p\Z$ and $a^{2\ell}-b^{2\ell}\in p\Z$}\label{CASETWO}
Note that the assumption implies that $p\geq 3$. 
Let $\gamma$ be as in Corollary \ref{CASETWOESTIMATE}.
Applying that corollary, we obtain
$\nu_p(g^{(\ell,p)}_{p^s,t}|_{v=\theta})=\gamma+r+\nu_p(t)$ for all $t\ge 1$ and $s\ge 0$ (see Definition~\ref{def:fg}). Hence,
\begin{equation}\label{eq:case2nupI}
\nu_p(I^{v}_{\ell,p}(\lambda)|_{v=\theta})=(\gamma+r)\ell(\lambda)+\sum_{s\geq 0}\nu_p(m_{p^s}(\lambda)!) = (\gamma+r)\ell(\lambda) + \nu_p (x_{\lambda}).
\end{equation}

Consider the matrix $K\in\MAT_{\POW_p(n)}(\Q)$ such that $N=xK$.
For each $\lambda\in \POW_p (n)$, we have $M_{\lambda,\lambda}=x_{\lambda}$ by Proposition \ref{nuestimateM} (\ref{nuestimateMa}), 
so $K_{\lambda,\lambda} =1$  (in fact, $K$ is $\Z$-valued by the same Proposition). 
We have $Y=xKb'(\TRANS{K})x$ where $b'=b^{(\ell,\theta)}z^{-1}$.
We will apply Lemma~\ref{DVRlinearalgebra} to this product.
Using Corollary~\ref{CASETWOESTIMATE}, we obtain
\begin{align}
 \nu_p (b'_{\lambda,\lambda}) &= \nu_p (J_{\ell}^v (\lambda)|_{v=\theta}) - \nu_p (z_{\lambda}) 
= \sum_{s\ge 0} m_{p^s} (\lambda) (r+s+\gamma) - \sum_{s\ge 0} (sm_{p^s} (\lambda) + \nu_p (m_{p^s} (\lambda)) ) \notag \\
\label{eq:nupbd} &= (\gamma+r) \ell(\lambda) - \nu_p (x_{\lambda}).
\end{align}

In order to verify the hypotheses of Lemma~\ref{DVRlinearalgebra}, we only need to show that 
$
\nu_p(K_{\lambda,\mu} - \delta_{\lambda,\mu})>(\nu_p(b'_{\lambda,\lambda})-\nu_p(b'_{\mu,\mu}))/2
$
 for all $\lambda,\mu \in \POW_p (n)$. 
This inequality is immediate if $M_{\lambda,\mu}=0$ or if $\lambda = \mu$ (as $K_{\lambda,\lambda}=1$). 
In the remaining case, we have 
\begin{align*}
&{}\nu_p(K_{\lambda,\mu})-\frac{\nu_p(b'_{\lambda,\lambda})-\nu_p(b'_{\mu,\mu})}{2} \\ 
&= 
 (\nu_p(M_{\lambda,\mu})- \nu_p (x_{\lambda})) + \frac{\nu_p (x_{\lambda}) - \nu_p(x_{\mu}) + (\gamma+r) (\ell(\mu)-\ell(\lambda))}{2} \\
&= \frac{\nu_p(M_{\lambda,\mu})-\nu_p(x_{\lambda})}{2}+
\frac{\nu_p(M_{\lambda,\mu})-\nu_p(x_{\mu})+\ell(\mu)-\ell(\lambda)}{2}+\frac{\ell(\mu)-\ell(\lambda)}{2}(\gamma+r-1) >0,
\end{align*}
as the first, second, and third summands are nonnegative by parts \eqref{nuestimateMa}, \eqref{nuestimateMc}, and \eqref{nuestimateMb} of Proposition~\ref{nuestimateM} respectively; moreover, the second summand is positive. 

By Lemma~\ref{DVRlinearalgebra}, $Y \equiv_{\Z_{(p)}} b' x^2$. It follows from~\eqref{eq:case2nupI} and~\eqref{eq:nupbd} that 
$\nu_p (b'_{\lambda,\lambda} x^2_{\lambda}) = \nu_p (I^{v}_{\ell,p}(\lambda)|_{v=\theta})$, so~\eqref{eq:Y} holds.

%
%

\subsection{Case $a^{2\ell}-b^{2\ell}\not\in p\Z$}\label{CASETHREE}
By Proposition~\ref{CASETHREEESTIMATE}, the determinants of the matrices on both sides of \eqref{mokuteki} are invertible in $\Z_{(p)}$.
Since both of these matrices are $\Z_{(p)}$-valued (see Proposition \ref{Zpvalued}), 
they are both unimodularly equivalent over $\Z_{(p)}$ to the identity matrix.

This completes the proof of Theorem \ref{meidai} and hence of Theorem~\ref{maintheorem}.

\section{Remarks on possible generalizations of Theorem~\ref{maintheorem}}\label{fittingideal}

Our aim here is to demonstrate how far we still are from proving Conjecture~\ref{ourgradedconjecture} and to discuss  natural statements that are stronger 
than Theorem~\ref{maintheorem} but are weaker than Conjecture~\ref{ourgradedconjecture}, as well as implications between those statements. Proving some of them -- if indeed they are true -- would provide further evidence for Conjecture~\ref{ourgradedconjecture} and would be of interest in its own right.

\Rem\label{locgol}
In the proof of Theorem~\ref{maintheorem} (cf.~\S\ref{condproof}), we have used 
the fact that the local-global correspondence holds when $\MAA$ is a PID
by Proposition \ref{threeequiveasyfacts} (\ref{elemdivthm}) and (\ref{locgloFit}), i.e.,
\begin{align*}
Y\CONG{\MAA} Z\Longleftrightarrow\forall\MEE\in\MSPEC(\MAA),Y\CONG{\MAA_{\MEE}} Z.
\end{align*}
\enrem

An advantage of considering unimodular pseudo-equivalences $\PCONG{\MAA}$ is that:
\Prop\label{locgol2}
Let $\MAA$ be a $1$-dimensional Noetherian domain.
For $n\times m$-matrices $Y,Z$ with entries in $\MAA$, we have
\begin{align}
Y\PCONG{\MAA} Z\Longleftrightarrow\forall\MEE\in\MSPEC(\MAA),Y\PCONG{\MAA_{\MEE}} Z
\label{locgolpseudo}
\end{align}
if $\COKERR{T}=\TOR_{\MAA}(\COKERR{T})(:=\{x\in\COKERR{T}\mid\exists a\in\MAA\setminus\{0\},ax=0\})$ 
for $T\in\{Y,Z\}$ (for example, when $n=m$ and $\det T\ne 0$ for $T\in\{Y,Z\}$).
\enprop

\Proof
The $\Rightarrow$ direction follows from Proposition~\ref{threeequiveasyfacts} (\ref{basechange}), so we need only prove the $\Leftarrow$ direction.
Since $R$ is an integral domain, the intersection of any two non-zero ideals of $R$ is non-zero, and in particular $I:=\ANN_{\MAA}(\COKERR{Y})\cap\ANN_{\MAA}(\COKERR{Z})\ne 0$.
Clearly, $Y\PCONG{\MAA} Z \Leftrightarrow Y\PCONG{\MAA'} Z$ where $\MAA':=\MAA/I$.
Since $\MAA'$ is Artinian, $\MSPEC(\MAA')$ is a finite set and 
the natural ring homomorphism $\MAA'\to\prod_{\MEE\in\MSPEC(\MAA')}\MAA'_{\MEE}$ is an isomorphism (see ~\cite[(24.C)]{Mat}).
Thus, $Y\PCONG{\MAA'} Z\Leftrightarrow\forall\MEE\in\MSPEC(\MAA'),Y\PCONG{\MAA'_{\MEE}} Z$ by Proposition \ref{threeequiveasyfacts} (\ref{basechange}).
Let $\phi\colon \MAA\twoheadrightarrow\MAA'$ be the natural surjection.
Since  $\MAA'_{\MEE}\cong \MAA_{\phi^{-1}(\MEE)}/I_{\phi^{-1}(\MEE)}$ (see ~\cite[Example 4.18 (a)]{Kun},
if $Y\PCONG{\MAA_{\phi^{-1}(\MEE)}}Z$ for all $\MEE\in\MSPEC(\MAA')$, then $Y\PCONG{\MAA} Z$.
Noting that $I\subseteq \phi^{-1}(\MEE)\in\MSPEC(\MAA)$ for all $\MEE\in\MSPEC(\MAA')$, we deduce the result.  
\QED

An advantage of considering Fitting equivalences $\FCONG{\MAA}$ is that 
for a large class of rings $\MAA$ 
we have an algorithm to decide whether two explicitly given matrices $Y$ and $Z$ are Fitting equivalent or 
not (see ~\cite[Chapter VIII]{RFW} and references therein).
If we have $Y\FCONG{\MAA} Z$, then by Proposition \ref{threeequiveasyfacts} (\ref{elemdivthm})
it is not possible to demonstrate that $Y\NCONG{\MAA} Z$ by localization or specialization to a PID $\MAA'$.
Thus, as far as unimodular equivalences over PIDs are concerned, 
the ultimate piece of evidence for Conjecture \ref{ourgradedconjecture} would be to prove 
that $X\FCONG{\MA} D$, where $X$ and $D$ are the matrices on the two sides of~\eqref{ourgradedconjectureshiki}. 


\begin{remark}
If $X\FCONG{\MA} D$, then, in particular, $X\CONG{\mathbb F_p[v,v^{-1}]} D$ for any prime $p$. 
Whether or not the latter equivalence holds is an interesting intermediate open problem. 
\end{remark}

\Prop\label{corkawanoue}
Let $X$ and $Y$ be $n\times m$-matrices with entries in $\MA$.
If $X|_{v=\theta} \FCONG{\Z[\theta,\theta^{-1}]} Y|_{v=\theta}$ for all $\theta\in\overline{\Q}\setminus\{0\}$,
then $X\FCONG{\MA} Y$.
\enprop

We conclude the paper by proving Proposition~\ref{corkawanoue}, which implies that, in order to show that $X\FCONG{\MA} D$, it would suffice to 
generalize Theorem~\ref{maintheorem} (\ref{subs}) by proving that 
$X|_{v=\theta} \FCONG{\Z[\theta,\theta^{-1}]} D|_{v=\theta}$ for all non-zero algebraic numbers $\theta$. 
Despite Proposition \ref{threeequiveasyfacts} (\ref{genericcase}) and 
Proposition \ref{locgol2},
proving the equivalence $X\FCONG{\Z[\theta,\theta^{-1}]} D$ for an arbitrary $\theta\in \ol{\Q} \setminus \{ 0\}$ (if it is true)
is likely to be considerably more difficult than proving Theorem~\ref{maintheorem} because $\Z[\theta,\theta^{-1}]$ 
is not integrally closed (equivalently, it is not a Dedekind domain) in general. 
However, it may be possible to use the methods of the present paper to prove that 
$X \equiv_{\MO_{\theta}} D$, where $\MO_{\theta}$ is the integral closure of the ring 
$\Z[\theta,\theta^{-1}]$ in its field of fractions, at least for some classes of algebraic numbers $\theta$. 
Establishing whether $X$ and $D$ are Fitting equivalent -- or, indeed, settling Conjecture~\ref{ourgradedconjecture} -- is likely to require new ideas.

\begin{proof}[Proof of Proposition~\ref{corkawanoue}]
The proposition is an immediate corollary of Theorem~\ref{kawanouethm}.
\QED

In the following, let $\SSSS$ be the set of non-constant irreducible polynomials in $\Z[v]$.
Let $\theta\in\overline{\Q}\ \setminus \{0\}$ be a root of $f\in\SSSS$.
For an ideal $I$ of $\MA$, we denote by $I|_{v=\theta}$  the image of $I$ under the ring surjection
$\pi_{\theta}\colon \MA \twoheadrightarrow \Z[\theta,\theta^{-1}]$ given by $v\mapsto \theta$. 
Then, by Gauss's Lemma we have $\KER\pi_{\theta}=\MA f$.

\Th\label{kawanouethm}
Let $I$ and $J$ be ideals of $\MA$. 
If $I|_{v=\theta}= J|_{v=\theta}$ in $\Z[\theta,\theta^{-1}]$ for 
all $\theta\in\overline{\Q}\setminus\{0\}$, then  $I=J$.
\enth

\begin{lemma}\label{lem:ca1}
 Let $R$ be a Noetherian commutative ring. For any 
ideal $I$ of $R$, we have 
\begin{align*}
I = \bigcap_{I \subseteq \MEE \in \MSPEC(R)} \bigcap_{n\ge 1} (I+ \MEE^n).
\end{align*}
\end{lemma}

\begin{proof}
Replacing $R$ with $R/I$, we may assume that $I=0$. Let 
$J=\cap_{\MEE \in \MSPEC(R)} \cap_{n\ge 1} \MEE^n$.
For $\MEE \in \MSPEC(R)$, we have $J_{\MEE} \subseteq \cap_{n\ge 1} \MEE_{\MEE}^n$, and 
$\cap_{n \ge 1} \MEE_{\MEE}^n=0$ in $R_{\MEE}$ by Krull intersection theorem. So $J_{\MEE}=0$ for all $\MEE$, 
whence $J=0$ (see the proof of Proposition \ref{threeequiveasyfacts} (\ref{locgloFit})). 
\end{proof}

\begin{lemma}\label{lem:ca2}
Let $\MEE\in\MSPEC(\Z[v])$ and let $n\ge 1$.
Then, $\MEE^n\cap\SSSS$ is an infinite set.
\end{lemma}

\begin{proof}
It is well known that $\MEE = (p, h)$ for some $p\in \PRIMES$ and 
non-constant monic irreducible polynomial $h$ which remains irreducible in $\F_p[v]$ (see~\cite[Exercise 7.9]{GP}).
For any $q\in\PRIMES$ with $q\ne p$, put $f_q:=p^n + qh^n\in \MEE^n$.
Then $f_q$ is primitive by construction and is in $\SSSS$ by Eisenstein's criterion (applied to the prime $q$). 
\end{proof}

\begin{proof}[Proof of Theorem~\ref{kawanouethm}]
For $\MEE \in \MSPEC(\MA)$ 
and $n\ge 1$, there exists $f\in \MEE^n\cap S$ such that $f\ne \pm v$ by Lemma~\ref{lem:ca2} applied to 
$\MEE \cap \Z[v] \in \MSPEC(\Z[v])$. 
By the hypothesis, we have $I|_{v=\theta} = J|_{v=\theta}$ for a root $\theta\in \ol{\Q}\setminus\{0\}$ of $f$, 
whence $I+\MA f=\pi^{-1}_{\theta}(I|_{v=\theta})=\pi^{-1}_{\theta}(J|_{v=\theta})=J+\MA f$. 
Since $\MA f \subseteq \MEE^n$, it follows that $I+\MEE^n = J+ \MEE^n$. By Lemma~\ref{lem:ca1}, we have $I=J$.
\end{proof}

\Rem\label{kcont}
We learned Theorem \ref{kawanouethm} from Hiraku Kawanoue.
His proof yields the existence of $f\in\SSSS$ such that $I+\MA f\ne J+\MA f$ for ideals $I\ne J\subseteq \MA$
and can be applied when we replace $\Z$ by any unique factorization domain $\MAA$ which has infinitely many prime elements 
modulo $\MAA^{\times}$. 
In order to keep this section short, we adapted the proof to one sufficient for Proposition \ref{corkawanoue}.
While the above proof depends on the description of $\MSPEC(\Z[v])$ and does not allow the indicated generalization, 
it shares the same spirit with Kawanoue's.
\enrem


\section*{Index of notation}

The following index gives references to subsections where symbols are defined:

\begin{longtable}{lll}
$\mathfrak S_n$ & symmetric group & \ref{subsec:ungraded} \\
$\mathcal H_n (\mathbb F;q)$ & Hecke algebra & \ref{subsec:ungraded} \\
$\eta_{\ell} \in k_{\ell}$ & a primitive $\ell$-th root of unity in a field & 
\ref{subsec:ungraded}  \\
$\MOD{A}$ & the category of finite-dimensional left $A$-modules & \ref{subsec:ungraded} \\
$\PROC(D)$ & projective cover of $D$ & \ref{subsec:ungraded} \\
$C_A$ &  Cartan matrix of an algebra 
$A$ & \ref{subsec:ungraded} \\
$\equiv_R$ & unimodular equivalence of matrices & \ref{subsec:kor}  \\
$\ell_k$ & $\ell/(\ell,k)$ & \ref{subsec:kor} \\
$I^v_\ell (\la), J^v_\ell (\la)$ & Laurent polynomials in Definition~\ref{IandJ} & \ref{subsec:graded_kor} \\
$\MSPEC(R)$ & the set of maximal ideals of a ring $R$ & \ref{CoAl} \\
$\MAT_\ell (R)$, $\MAT_S(R)$ & matrix algebra & \ref{matconv} \\
$1_S$ & identity matrix & \ref{matconv} \\
$\DIAG(\{r_s \mid s\in S\})$ & diagonal matrix & \ref{matconv} \\
$\bigoplus_i M_i$ & block-diagonal matrix & \ref{matconv} \\
$\nu_p$ & $p$-adic valuation & \ref{SSSdvr} \\
$\mathbb N$ & the set of nonnegative integers & \ref{subsubsec:int} \\
$\PRIMES$ & the set of prime numbers & \ref{subsubsec:int} \\
$n_{\Pi}$ & $\Pi$-part of $n$ & \ref{subsubsec:int} \\
$\Pi'$, $p'$ & the complements of $\Pi$, $\{p\}$ in $\PRIMES$ & \ref{subsubsec:int} \\
$(a,b)$ & greatest common divisor of $a$ and $b$ & \ref{subsubsec:int} \\
$\cor$ & the function field $\mathbb Q(v)$ & 
\ref{subsubsec:qi} \\
$\MA$ & the ring of Laurent polynomials $\mathbb Z[v,v^{-1}]$ & \ref{subsubsec:qi} \\
$\BAR$ & bar-involution $v\mapsto v^{-1}$ on $\cor$ & \ref{subsubsec:qi}\\
$\INFL_t$ & the inflation map $v\mapsto v^t$ on $\MA$ & \ref{subsubsec:qi} \\
$[n]_m$ & quantum integer & \ref{subsubsec:qi} \\
$[n]_m !$ & quantum factorial &  \ref{subsubsec:qi} \\
$\equiv_G$ & conjugacy relation in a group $G$ & \ref{SSSgroups} \\
$g_p, g_{p'}$ & $p$-part and $p'$-part of $g$ & \ref{SSSgroups} \\
$m_k (\lambda)$ & multiplicity of $k$ in a partition $\lambda$ & 
\ref{SSSPartitions} \\
$\ell(\lambda)$ & length of a partition $\lambda$ & 
\ref{SSSPartitions} \\
$|\lambda|$ & size of a partition $\lambda$ & \ref{SSSPartitions} \\
$\PAR$, $\PAR(n)$ &  set of partitions & \ref{SSSPartitions} \\
$\CPAR_s(n)$ & the set of $s$-class regular partitions of $n$ & \ref{SSSPartitions} \\
$\RPAR_s (n)$ & the set of $s$-regular partitions of $n$ &
\ref{SSSPartitions} \\ 
$\PAR_m (n)$ & the set of $m$-multipartitions of $n$ & \ref{SSSPartitions} \\
$\PAR_p (n,\nu)$ & the set of partitions of $n$ with ``$p'$-part'' $\nu$ & \ref{SSSPartitions} \\
$\POW_p (n)$ & the set of $p$-power partitions of $n$ & 
\ref{SSSPartitions} \\
$\lambda+\mu$ & sum of two partitions & \ref{SSSPartitions} \\
$\Lambda$ & ring of symmetric functions & \ref{SSSsym} \\
$\chi_V$ & character of $\mathfrak S_n$ afforded by module $V$ & 
\ref{SSSsym} \\
$p_\mu$, $p_k$ & power sum symmetric functions & \ref{SSSsym} \\
$C_\mu$ & conjugacy class corresponding to a partition $\mu$ 
& \ref{SSSsym} \\
$z_{\mu}$ & order of centralizer of an element of $C_\mu$ & \ref{SSSsym} \\
$\CATISO$ & isometry between a Grothendieck group and symmetric functions 
& \ref{SSSsym} \\
$\mathfrak S_\lambda$ & parabolic subgroup of $\mathfrak S_n$ 
& \ref{SSSsym} \\ 
$\TRIV_{\mathfrak S_\la}$ & trivial representation of $\mathfrak S_\lambda$
& \ref{SSSsym} \\
$M_n$ & table of permutation characters of $\mathfrak S_n$ & \ref{SSSsym} \\
$h_{\mu}$ & complete symmetric function & \ref{SSSsym} \\
$m_\mu$ & monomial symmetric function & \ref{SSSsym} \\
$\mathcal{M}_{\lambda,\mu}$ & a certain set of size $M_{\lambda,\mu}$ & \ref{SSSsym} \\
$N_n^{(p)}$ & ``$p$-local'' submatrix of $M_n$ & \ref{SSSploc} \\
$L_n^{(p)}$ & a certain block-diagonal matrix & \ref{SSSploc} \\
$a_{\theta}^{(p)} (n)$ & a rational number from Definition~\ref{apn} &
\ref{SSSploc} \\
$\SYM^m$ & symmetric power functor & \ref{SSSmult} \\
$\MULT_m(\ell)$ & the set of weakly increasing $m$-tuples of elements of $\{1,\dots,\ell\}$ & \ref{SSSmult} \\
$\Omega_{\ell,d}$ & a set of tuples, which is in bijection with $\PAR_{\ell} (d)$ & \ref{SSSmult}  \\
$S^d (A)$ & matrix in Definition~\ref{def:Sd} & \ref{SSSmult} \\
$\Lambda_\ell=\bigotimes_{t=1}^\ell \Lambda^{(t)}$ & $\ell$-colored ring of symmetric functions & \ref{SSSmult} \\
$m_{\mu}^{(t)}$, $h_\mu^{(t)}$, $p_{\mu}^{(t)}$ & images of $m_\mu$, $h_\mu$, $p_\mu$ in $\Lambda^{(t)}$ & 
 \ref{SSSmult} \\
$M_{\ell,d}$, $K_{\ell,d}$ & transition matrices in Definition~\ref{MK}
& \ref{SSSmult} \\
$\MP$, $\MPC$ & weight lattice and its dual & \ref{pretsu2} 
\\
$\Pi$, $\Pi^\vee$ & sets of simple roots and corresponding coroots &
\ref{pretsu2} \\
$Q^+$ & positive part of the root lattice & \ref{pretsu2} \\
$\MP^+$ & set of dominant integral weights & \ref{pretsu2} \\
$\Lambda_i$ & a dominant integral weight & \ref{pretsu2} \\
$W=W(X)$ & Weyl group &  \ref{pretsu2} \\
$U_v = U_v (X)$ & quantum group & \ref{pretsu2}  \\
$U_v^+$, $U_v^0$, $U_v^-$ & subalgebras in the triangular decomposition of $U_v$ & \ref{pretsu2} \\
$V(\lambda)$ & highest weight module & \ref{pretsu2} \\
$1_{\lambda}$ & highest weight vector & \ref{pretsu2} \\
$\langle \cdot,\cdot\rangle_{\QSH}$, 
$\langle \cdot,\cdot\rangle_{\RSH}$ & versions of Shapovalov form from Proposition~\ref{tsu38} & \ref{pretsu2} \\
$P(\lambda)$ & the set of weights of $V(\lambda)$ & \ref{pretsu2} \\
$V(\lambda)_\mu$ & $\mu$-weight space of $V(\lambda)$ & 
\ref{pretsu2} \\
$(U_v^-)^\MA$ & an $\MA$-lattice in $U_v^-$ & \ref{pretsu2} \\
$V(\lambda)^{\MA}$, $V(\lambda)_\nu^\MA$ & 
$\MA$-lattices in $V(\lambda)$, $V(\lambda)_\nu$ 
& \ref{pretsu2} \\
$\QSHM_{\lambda,\mu}$, $\RSHM_{\lambda,\mu}$ 
& Gram matrices of Shapovalov forms, see Definition~\ref{tsu313} & 
\ref{pretsu2} \\
$\GCONG$ & an equivalence relation on matrices & \ref{pretsu2} \\
$\widehat{X}$ & extended Cartan matrix of $X$ & \ref{SSADE} \\
$\delta$ & null-root & \ref{SSADE} \\
$\GMOD{A}$ & category of finite-dimensional graded $A$-modules &
\ref{subsec:gradeddef} \\
$M_n$ & graded $n$-component of a module $M$ &  
\ref{subsec:gradeddef}\\
$M\langle k\rangle$ & graded module $M$ with grading shifted down by $k$ & \ref{subsec:gradeddef} \\
$\mathcal S(A)$ & set of representatives of simple graded $A$-modules  & \ref{subsec:gradeddef}
\\
$C_A^v$ & graded Cartan matrix of $A$ & \ref{subsec:gradeddef} \\
$\GPPP{A}$ & category of projective  graded $A$-modules & \ref{subsec:gradeddef} \\
$\BLOCK_\ell (n)$ & the set of pairs $(\rho,d)$ where $\rho$ is an $\ell$-core and $d\in \mathbb N$ & \ref{subsec:gradeddef} \\
$\COKERR{T}$ & cokernel of the map given by a matrix $T$ & 
\ref{varequi} \\
$\equiv'_R$ & unimodular pseudo-equivalence of matrices, see 
Definition~\ref{DVariants} & \ref{varequi} \vspace{1mm} \\
$\equiv^F_R$ & Fitting equivalence of matrices, see 
Definition~\ref{DVariants} & \ref{varequi} \\
$\FITT_d (T)$ & $d$-th Fitting ideal of a matrix $T$ & \ref{varequi} \\
$\Phi_n$, $\Psi_n$ & cyclotomic polynomial and its scaled version 
& \ref{SSpseudo} \\
$\rho_z^{(p)}$ & function from Definition~\ref{Drho} & 
\ref{SSpseudo}\\
$\varphi_{s,n}$ & a bijection from $s$-regular to $s$-class regular partitions & \ref{SSpseudo} \\
$\beta_M$ & auto-bijection of $\PAR$ from Definition~\ref{keybijection}
& \ref{SSpseudo} \\
$g_{k,t}^{(\ell,p)}$, $f_{k,t}^{(\ell)}$, $I_{\ell,p}^v (\lambda)$ & certain products of quantum integers, see Definition~\ref{def:fg} & 
\ref{SSpseudo} \\
$\mathcal F_{k,t,z}^{(\ell,p)}$, 
$\mathcal G_{k,t,z}^{(\ell,p)}$ & certain sets of integers related to 
$f_{k,t}^{(\ell)}$, $g_{k,t}^{(\ell,p)}$ & \ref{SSpseudo}  \vspace{1mm} \\
$\lambda^{<r},\lambda^{\geq r},\OVERy{\lambda}{r}$ 
& $p$-power partitions from Definition~\ref{sixmats} & \ref{SSevs}
\end{longtable}

\end{document}